\documentclass{amsart}
\usepackage{color}
\usepackage{dsfont}
\usepackage{mathrsfs}
\usepackage{amsmath}
\usepackage{amsfonts}
\usepackage{amscd}
\usepackage{amsthm}
\usepackage[dvips]{graphicx}
\usepackage{xypic}
\usepackage{latexsym,amssymb}
\usepackage{pstricks}
\usepackage{tikz}

\usepackage{caption}
\usepackage{indentfirst}
\usepackage{pifont}

\usepackage{hyperref}
\usepackage{cite}

\usepackage{extarrows}
\usepackage{chemarrow}

\def\A{{\mathbb{A}}}

\newcommand{\upcite}[1]{\textsuperscript{\cite{#1}}}

\newtheorem{thm}{Theorem}[section]
\newtheorem{lem}[thm]{Lemma}
\newtheorem{prop}[thm]{Proposition}
\newtheorem{cor}[thm]{Corollary}
\newtheorem{defn}[thm]{Definition}

\newtheorem{remark}[thm]{Remark}

\numberwithin{equation}{section}

\date{}
\begin{document}

\thispagestyle{empty}

\begin{center}
{\bf{ \LARGE  Hopf PBW-deformations of a new type quantum group}
 \footnotetext { $\dag$
Corresponding author: chenjialei@bjut.edu.cn}
}

\author[Xu]{Yongjun Xu}
\address{School of Mathematical Sciences, Qufu Normal University,
Qufu 273165, P. R. China}
\email{yjxu2002@163.com}

\author[Chen]{Jialei Chen}
\address{Faculty of Science, Beijing University of Technology, Beijing 100124, P. R. China}
\email{chenjialei@bjut.edu.cn}

\bigbreak

\normalsize Yongjun Xu$^{1}$, Jialei Chen$^{\dag 2}$

{\footnotesize\small\sl 1\ \ School of Mathematical Sciences,\
Qufu Normal University , Qufu 273165, P. R. China\\
2\ \  Department of Mathematics,
Faculty of Science, Beijing University of Technology,\\
\footnotesize\sl Beijing 100124, P. R. China }

\end{center}
\begin{quote}
{\noindent\small{\bf Abstract.}
In this paper, we mainly focus on a new type quantum group $U_{q}(\mathfrak{sl}^{*}_2)$ and its Hopf PBW-deformations $U_{q}(\mathfrak{sl}^{*}_2,\kappa)$ in which $U_{q}(\mathfrak{sl}^{*}_2,0) = U_{q}(\mathfrak{sl}^{*}_2)$ and the classical Drinfeld-Jimbo quantum group $U_{q}(\mathfrak{sl}_2)$ is included.
The category of finite dimensional $U_{q}(\mathfrak{sl}^{*}_2)$-modules is proved to be non-semisimple.
We establish a uniform block decomposition of the category $U_{q}(\mathfrak{sl}^{*}_2,\kappa){\mbox -}{\rm \bf mod}_{\rm wt}$ of finite dimensional weight modules for each $U_{q}(\mathfrak{sl}^{*}_2,\kappa)$, and reduce the investigation on $U_{q}(\mathfrak{sl}^{*}_2,\kappa){\mbox -}{\rm \bf mod}_{\rm wt}$ to its principle block(s).
We introduce the notion of primitive object in $U_{q}(\mathfrak{sl}^{*}_2,\kappa){\mbox -}{\rm \bf mod}_{\rm wt}$ which affords a new and elementary way to verify the semisimplicity of the category of finite dimensional $U_{q}(\mathfrak{sl}_2)$-modules.
As the core of this present paper, a tensor equivalence between the principal block(s) of $U_{q}(\mathfrak{sl}^{*}_2,\kappa){\mbox -}{\rm \bf mod}_{\rm wt}$ and the category of finite dimensional representations of (deformed) preprojective algebras of Dynkin type $\A$ is obtained.

}
\end{quote}


\setcounter{section}{-1}

\section{Introduction}

Based on a new associative multiplication on $2 \times 2$ matrices, Aziziheris et al. introduced a new type quantum group $U_{q}(\mathfrak{sl}^{*}_2)$ and a new type quantized coordinate ring $\mathcal{O}_q(SL^{*}_2)$
which possess a Hopf pairing with each other \cite{AFL}. In \cite{WW2014}, Walton and Witherspoon provided
necessary and sufficient conditions for determining PBW-deformations of a smash product algebra $B \sharp H$, where $H$ is a Hopf algebra and $B$ is a Koszul $H$-module algebra.
The new type quantum group $U_{q}(\mathfrak{sl}^{*}_2)$ initially attracts our attention in that
it can be realized as a smash product algebra (Proposition \ref{realization-of-new-type-quantum-group-as-smash-product-alg} ${\rm (3)}$) and the classical Drinfeld-Jimbo quantum group $U_q(\mathfrak{sl}_2)$ is a PBW-deformation of it (Corollary \ref{Hopf-PBW-deformation}).
As the first task of this paper, we determine all the Hopf PBW-deformations $U_{q}(\mathfrak{sl}^{*}_2,\kappa)$ of $U_{q}(\mathfrak{sl}^{*}_2)$ (Corollary \ref{Hopf-PBW-deformation}),
 then show that $U_q(\mathfrak{sl}_2)$ is almost the unique nontrivial one (Remark \ref{classical-is-almost-unique-hopf-pbw-deformation}).
In fact, all the nontrivial Hopf PBW-deformations $U_{q}(\mathfrak{sl}^{*}_2, \kappa)$ of $U_{q}(\mathfrak{sl}^{*}_2)$ have been defined as generalizations of $U_q(\mathfrak{sl}_2)$ in \cite{JW} (see also \cite{WJY2002}).

The question whether the representation theory of smash product algebras is preserved under PBW-deformation is very intresting (cf. \ref{WW2014}).
In \cite{JW}, each finite dimensional module of any nontrivial Hopf PBW-deformation $U_{q}(\mathfrak{sl}^{*}_2, \kappa)$ of $U_{q}(\mathfrak{sl}^{*}_2)$ is proved to be semisimple via the similar methods dealing with $U_q(\mathfrak{sl}_2)$. As the second task of this paper, we characterize all finite dimensional simple $U_{q}(\mathfrak{sl}^{*}_2)$-modules (Proposition \ref{simple-modules-of-new-type-quantum-groups}), then prove that the category $U_{q}(\mathfrak{sl}^{*}_2){\mbox -}{\rm \bf mod}$ of finite dimensional $U_{q}(\mathfrak{sl}^{*}_2)$-modules is not semisimple (Corollary \ref{non-semisimplicity-of-reps-of-prequantum-group}). So the representation theory of $U_{q}(\mathfrak{sl}^{*}_2)$ cannot be preserved under nontrivial Hopf PBW-deformation.

The notion of block is very helpful for organizing objects in module categories which fail to be semisimple (cf. Section 1.13 in \cite{Hum2008}).
 For $U_{q}(\mathfrak{sl}_2)$, each finite dimensional module can be expressed as the direct sum of its weight spaces.
 As the third task, we investigate the category $U_{q}(\mathfrak{sl}^{*}_2,\kappa){\mbox -}{\rm \bf mod}_{\rm wt}$ of finite dimensional weight modules of $U_{q}(\mathfrak{sl}^{*}_2,\kappa)$.
When $\kappa = 0$, $U_{q}(\mathfrak{sl}^{*}_2){\mbox -}{\rm \bf mod}_{\rm wt}$ is a proper and non-semisimple full subcategory of $U_{q}(\mathfrak{sl}^{*}_2){\mbox -}{\rm \bf mod}$;
 when $\kappa \neq 0$, $U_{q}(\mathfrak{sl}^{*}_2,\kappa){\mbox -}{\rm \bf mod}_{\rm wt}$ is just equal to $U_{q}(\mathfrak{sl}^{*}_2,\kappa){\mbox -}{\rm \bf mod}$ (Remark \ref{remarks-about-rep-cat-and-weight-rep-cat}).
Surprisingly, we obtain a uniform block decomposition theorem of the category $U_{q}(\mathfrak{sl}^{*}_2,\kappa){\mbox -}{\rm \bf mod}_{\rm wt}$, whatever $\kappa$ is zero or not (Theorem \ref{block-decomposition-of-weight-rep-cat-of-U}).
Then we prove that each block of $U_{q}(\mathfrak{sl}^{*}_2,\kappa){\mbox -}{\rm \bf mod}_{\rm wt}$ is isomorphic to a principle block (Corollary \ref{reduction-to-principle-blocks}).
It is worth noting that the category $U_{q}(\mathfrak{sl}^{*}_2,\kappa){\mbox -}{\rm \bf mod}_{\rm wt}$ only has one unique principle block when $\kappa = 0$, but two principle blocks when $\kappa \neq 0$ (Definition \ref{definitions-of-principle-blocks}).

The isomorphisms among the blocks of the category $U_{q}(\mathfrak{sl}^{*}_2,\kappa){\mbox -}{\rm \bf mod}_{\rm wt}$ enable us to focus on the indecomposable objects in the principle block(s).
 For this reason, we introduce and study the primitive objects in $U_{q}(\mathfrak{sl}^{*}_2,\kappa){\mbox -}{\rm \bf mod}_{\rm wt}$ which can be used to produce all
indecomposable objects (Proposition \ref{primitive-is-proper-new-type-quantum-group-case} and Theorem \ref{primitive-representations-of-Hopf-PBW-deformations-is-proper} ${\rm (1)}$).
 For $U_{q}(\mathfrak{sl}^{*}_2)$, we can only construct some explicit primitive modules (Theorem \ref{primitive-rep-of-U-with-type-1-q}).
 Whereas for any nontrivial Hopf PBW-deformation $U_{q}(\mathfrak{sl}^{*}_2,\kappa)$ of $U_{q}(\mathfrak{sl}^{*}_2)$, we can make use of its primitive modules to construct all simple objects in $U_{q}(\mathfrak{sl}^{*}_2,\kappa){\mbox -}{\rm \bf mod}$ (Theorem \ref{primitive-representations-of-Hopf-PBW-deformations-is-proper} ${\rm (2)}$) and prove the semisimplicity of $U_{q}(\mathfrak{sl}^{*}_2,\kappa){\mbox -}{\rm \bf mod}$ (Theorem \ref{new-proof-of-semisimplicity-of-fd-representations-of-Hopf-PBW-deformations}), which affords a new treatment for the representation theory of Drinfeld-Jimbo quantum group $U_q(\mathfrak{sl}_2)$.

Assume that $\mathbb{K}$ is a field and $Q$ is a finite quiver with vertex set $Q_0$. The preprojective algebra $\Pi_Q$ of $Q$ was explicitly presented by Ringel in \cite{Ringel1998},
while the deformed preprojective algebra $\Pi^\lambda_Q$ with weight $\lambda \in \mathbb{K}^{|Q_0|}$ was introduced by Crawley-Boevey and Holland in \cite{WCMH605}.
The representation theory of $\Pi_Q$ satisfies a remarkable trichotomy: finite, tame or wild representation type
(cf. \cite{DR200, GLS193, GS163}).
Denote by $\mathfrak{g}_Q$ the Kac-Moody Lie algebra corresponding to $Q$.
 In \cite{WC257}, Crawley-Boevey obtained a seminal result which exactly characterizes the positive roots of $\mathfrak{g}_Q$ corresponding to the simple representations of $\Pi^\lambda_Q$ over an algebraically closed field $\mathbb{K}$.
Semisimple representations of $\Pi^\lambda_Q$ are essentially important since only the moduli space consisting of equivalence classes of them with a fixed dimension vector is a smooth manifold (see \cite{WC257,Silan2018}).

Preprojective algebras and deformed preprojective algebras are deeply related with realizations of quantum groups and their restricted forms.
 Denote by $Q$ the Dynkin quiver of a complex semisimple Lie algebra $\mathfrak{g}$.
For each simply-laced quantum group $U_q(\mathfrak{g})$, Lusztig constructed its negative part $U_q^-(\mathfrak{g})$ endowed with a canonical basis by using the perverse sheaf approach, where the variety of modules over the preprojective algebra $\Pi_Q$ of $Q$ plays an important role \cite{Lus1990,Lus1991}.
For the restricted simply-laced quantum groups $\overline{U}_q(\mathfrak{g})$, Cibils presented their Borel parts $\overline{U}_q^{\leq 0}(\mathfrak{g})$ as quotients of certain path algebras with admissible ideals \cite{Cibils1993,Cibils1997}, while
Huang and Yang realized them as quotients of double quiver algebras with non-admissible ideals \cite{Huang-Yang2005}.
 In fact, the work in \cite{Huang-Yang2005} follows \cite{Yang2004} where the restricted quantum group $\overline{U}_{q}(\mathfrak{sl}_2)$ is
  exactly realized as a quotient of a deformed preprojective algebra attached to a cyclic quiver.

The category of finite dimensional modules of a Hopf algebra with bijective antipode is naturally a tensor category \cite{EGNO2015}.
There are some remarkable and beautiful results that can be considered as tensor-categorical realizations of Drinfeld-Jimbo quantum groups.
In \cite{Dr1990,Dr1991}, the Drinfeld category corresponding to $\mathfrak{g}$ is proved to be tensor equivalent to the category of finite dimensional modules of quantum group $U_\hbar(\mathfrak{g})$ over formal power series.
  Following Drinfeld's work, Kazhdan and Lusztig established the well-known Kazhdan-Lusztig equivalence
   between the category of finite dimensional modules of simply-laced quantum group $U_q(\mathfrak{g})$ with irrational $q \in \mathbb{C}$
    and a tensor category of modules of an affine Lie algebra in \cite{KL1991, KL1993-1,KL1994-3}.
    So far, some extended versions of the Kazhdan-Lusztig equivalence are still a vibrant area of research (cf. \cite{CF2021,EM2002,Gait2021,GL2016,GL2017,NT2011}).

On the basis of aforementioned motivations and tasks, we finally aim to realize all Hopf PBW-deformations of the new type quantum group $U_{q}(\mathfrak{sl}^{*}_2)$ via (deformed) preprojective algebras of Dynkin type $\A$.
To be specific, we obtain the following main observations:
the unique principle block of the category $U_{q}(\mathfrak{sl}^{*}_2){\mbox -}{\rm \bf mod}_{\rm wt}$ of finite dimensional weight modules of $U_{q}(\mathfrak{sl}^{*}_2)$
is tensor equivalent to a tensor category attached to the preprojective algebras of Dynkin type $\A$ (Theorem \ref{tensor-equivalence-between-categories-and-primitive-reps-of-U-and-preprojetive-algs-of-A}),
while the direct sum of two principle blocks of the category $U_{q}(\mathfrak{sl}^{*}_2, \kappa){\mbox -}{\rm \bf mod}$ of finite dimensional $U_{q}(\mathfrak{sl}^{*}_2, \kappa)$-modules
is tensor equivalent to a tensor category attached to two classes of deformed preprojective algebras of Dynkin type $\A$ (Theorem \ref{tensor-equivalence-between-categories-and-primitive-reps-of-U-and-deformed-preprojetive-algs-of-A}).

The paper is organized as follows.
In Section \ref{section-1}, we recall the definition of the new type quantum group $U_{q}(\mathfrak{sl}^{*}_2)$, and determine its all Hopf PBW-deformations $U_{q}(\mathfrak{sl}^{*}_2,\kappa)$ up to isomorphisms.
In Section \ref{section-2}, we show that the category $U_{q}(\mathfrak{sl}^{*}_2){\mbox -}{\rm \bf mod} $ of finite dimensional modules of $U_{q}(\mathfrak{sl}^{*}_2)$ is non-semisimple.
In Section \ref{section-3}, we deal with the category $U_{q}(\mathfrak{sl}^{*}_2,\kappa){\mbox -}{\rm \bf mod}_{\rm wt}$ of finite dimensional weight modules of $U_{q}(\mathfrak{sl}^{*}_2,\kappa)$ including $U_{q}(\mathfrak{sl}^{*}_2)$ and $U_{q}(\mathfrak{sl}^{}_2)$.
We prove the Krull-Schmidt theorem and block decomposition theorem for $U_{q}(\mathfrak{sl}^{*}_2,\kappa){\mbox -}{\rm \bf mod}_{\rm wt}$,
 then reduce the investigation of $U_{q}(\mathfrak{sl}^{*}_2,\kappa){\mbox -}{\rm \bf mod}_{\rm wt}$ to its principle block(s).
In Section \ref{section-4}, we introduce and investigate the primitive objects in $U_{q}(\mathfrak{sl}^{*}_2,\kappa){\mbox -}{\rm \bf mod}_{\rm wt}$. When $\kappa \neq 0$, we reprove the category of finite dimensional modules of $U_{q}(\mathfrak{sl}^{*}_2,\kappa)$ is semisimple in a new way.
In Section \ref{section-5}, we introduce the category of finite dimensional representations of deformed preprojective algebras of Dynkin type $\A$,
 then obtain some tensor-categorical realizations of the principle block(s) of the category $U_{q}(\mathfrak{sl}^{*}_2,\kappa){\mbox -}{\rm \bf mod}_{\rm wt}$.

Throughout the paper, the notations $\mathbb{C}, \mathbb{C}^\times, \mathbb{Z}$, $\mathbb{Z}^{\geq 0}$ and $\mathbb{Z}^{> 0}$ respectively denote the complex field, the set of all nonzero complex numbers, the set of all integers, the set of all nonnegative integers and the set of all positive integers. We always assume that $q\in \mathbb{C}^\times$ is not a root of unity.
For $n,m\in \mathbb{Z}$, we fix the following notation
\begin{eqnarray*}\label{}
[n]_m = \frac{q^{mn}-q^{-mn}}{q^{m}-q^{-m}} = {\mathop\sum\limits_{i = 0}^{n-1}} q^{m(2i+1 - n)}.
\end{eqnarray*}
All linear spaces, algebras, modules and unadorned tensors are over the complex field $\mathbb{C}$. We represent by $(-)^\star$ the dual functor ${\rm Hom}_{\mathbb{C}}(-,\mathbb{C})$.
We understand by tensor category a locally finite monoidal rigid abelian $\mathbb{C}$-linear category with ${\rm End}(\mathds{1}) \cong \mathbb{C}$, where $\mathds{1}$ is the unit object, and by tensor functor a $\mathbb{C}$-linear monoidal functor.
Denote by $({\rm \bf Vec}, \otimes, \mathbb{C}, \mathbf{a},\mathbf{l},\mathbf{r})$ the tensor category of finite dimensional $\mathbb{C}$-vector spaces with unit object $\mathbb{C}$, associativity and unit constraints
$\mathbf{a},\mathbf{l},\mathbf{r}$.

\section{The new type quantum group $U_{q}(\mathfrak{sl}^{*}_2)$ and its Hopf PBW-deformations $U_{q}(\mathfrak{sl}^{*}_2,\kappa)$}\label{section-1}

Firstly we recall the definition of a new type quantum group.

\begin{defn}\upcite{AFL}
The new type quantum group $U_{q}(\mathfrak{sl}^{*}_2)$ is the associative algebra with unit 1 generated by four generators $K, K^{-1}, E, F$ and subject to the following relations
\begin{eqnarray*}\label{}
  KK^{-1}=K^{-1}K=1,\ \ KE=q^2EK,\ \ KF=q^{-2}FK,\ \ EF=FE.
\end{eqnarray*}
\end{defn}

\begin{prop}\upcite{AFL}
The new type quantum group $U_{q}(\mathfrak{sl}^{*}_2)$ is a Hopf algebra with coproduct $\Delta$, counit $\varepsilon$
and antipode $S$ defined by
\begin{gather}
\Delta(K)=K\otimes K,\ \ \Delta(E)=E\otimes K+1\otimes E,\ \  \Delta(F)=F\otimes 1+K^{-1}\otimes F,\nonumber\\
\label{hopf-alg-structure-on-the-pre-quantum-gp} \varepsilon(K)=1,\ \ \varepsilon(E)=0,\ \  \varepsilon(F)=0,\\
S(K)=K^{-1},\ \  S(E)=-EK^{-1},\ \ S(F)=-KF.\nonumber
\end{gather}
\end{prop}

\begin{remark}
The quantum algebra $U_{q}(\mathfrak{sl}^{*}_2)$ can be obtained as the associated graded algebra of the Drinfeld-Jimbo quantum group $U_{q}(\mathfrak{sl}_2)$ under the coradical filtration (cf. Section 5.11 in \cite{EGNO2015}).
\end{remark}

Assume that $\mathbb{C}[E,F]$ is the polynomial algebra generated by $E$ and $F$,
 and $G$ is the abelian group with generators $K,K^{-1}$ and relations $K K^{-1} = K^{-1} K = 1$.
Noting that the algebra $\mathbb{C}[E,F]$ can be considered as a $\mathbb{C} G$-module algebra under the following $\mathbb{C} G$-action
\begin{eqnarray}\label{group-action-on-alg}
\cdot : \mathbb{C} G \otimes \mathbb{C}[E,F] &\longrightarrow& \mathbb{C}[E,F],\\
K \otimes E &\longmapsto& K \cdot E = q^2 E,\nonumber\\
K \otimes F &\longmapsto& K \cdot F = q^{-2} F,\nonumber
\end{eqnarray}
one can realize $U_{q}(\mathfrak{sl}^{*}_2)$ as the smash product algebra $\mathbb{C}[E,F] \sharp \mathbb{C} G$.

\begin{prop}\label{realization-of-new-type-quantum-group-as-smash-product-alg}
${\rm (1)}$ The algebra $U_{q}(\mathfrak{sl}^{*}_2)$ is Noetherian and has no zero divisors.\\
${\rm (2)}$  The set $\{ E^i F^j K^k | i,j \in \mathbb{Z}^{\geq 0}, k\in \mathbb{Z}\}$ is a basis of the algebra $U_{q}(\mathfrak{sl}^{*}_2)$.\\
${\rm (3)}$ As a $\mathbb{C}$-algebra, $U_{q}(\mathfrak{sl}^{*}_2)$ is isomorphic to the smash product algebra $\mathbb{C}[E,F] \sharp \mathbb{C} G$.
\end{prop}

\begin{proof}
It is easy to check that the algebra $U_{q}(\mathfrak{sl}^{*}_2)$ can be realized as the iterated Ore extensions $\mathbb{C}[K,K^{-1}][E,\alpha_1,0][F,\alpha_2,0]$, where $\alpha_1: \mathbb{C}[K,K^{-1}]\longrightarrow\mathbb{C}[K,K^{-1}]$ and $\alpha_2: \mathbb{C}[K,K^{-1}][E,\alpha_1,0]\longrightarrow \mathbb{C}[K,K^{-1}][E,\alpha_1,0]$ are automorphisms respectively defined by
$\alpha_1(K) =  q^{-2} K$ and $\alpha_2(E^i K^k) =  q^{2k} E^i K^k$. Hence the statements in ${\rm (1)}$ and ${\rm (2)}$ hold.

Now we prove ${\rm (3)}$. Define the algebra homomorphism $\varphi : \mathbb{C} \langle K,K^{-1},E,F\rangle \longrightarrow \mathbb{C}[E,F] \sharp \mathbb{C} G$ by
\begin{eqnarray*}\label{}
\varphi(K) = 1 \sharp K, \ \ \varphi(K^{-1}) = 1 \sharp K^{-1}, \ \ \varphi(E) = E \sharp 1 \ \ {\rm and} \ \ \varphi(F) = F \sharp 1,
\end{eqnarray*}
then $\varphi$ is surjective.
Denote by $J$ the ideal of $\mathbb{C} \langle K,K^{-1},E,F\rangle$ generated by
\[KK^{-1} -1, \quad K^{-1}K - 1, \quad KE - q^2EK, \quad KF - q^{-2}FK \quad {\rm and} \quad E F - F E.\]
One can check that $\varphi(J) = 0$, then $\varphi$ can induce a natural surjective algebra homomorphism $$\overline{\varphi}: U_{q}(\mathfrak{sl}^{*}_2) \longrightarrow \mathbb{C}[E,F] \sharp \mathbb{C} G.$$
Since $\{ E^i F^j K^k | i,j \in \mathbb{Z}^{\geq 0}, k\in \mathbb{Z}\}$ (resp. $\{ E^i F^j \sharp K^k | i,j \in \mathbb{Z}^{\geq 0}, k\in \mathbb{Z}\}$) is a basis of $U_{q}(\mathfrak{sl}^{*}_2)$ (resp. $\mathbb{C}[E,F] \sharp \mathbb{C} G$), then $\overline{\varphi}$ is an isomorphism.
Therefore, $U_{q}(\mathfrak{sl}^{*}_2) \cong \mathbb{C}[E,F] \sharp \mathbb{C} G$.
\end{proof}

Let $V$ be the 2-dimensional vector space with a basis $\{E,F\}$. Via the $\mathbb{C} G$-action in $(\ref{group-action-on-alg})$,
 $V$ can be considered as a $\mathbb{C} G$-module. Denote by $T_\mathbb{C}(V)$ the tensor algebra generated by $V$. Then $T_\mathbb{C}(V)$ has a natural $\mathbb{C} G$-module algebra structure and
 one can form the smash product algebra $T_\mathbb{C}(V) \sharp \mathbb{C} G$. The polynomial algebra $\mathbb{C}[E,F]$ can be identified with $T_\mathbb{C}(V) / \langle I\rangle$, where $I = {\rm Span}_\mathbb{C}\{ EF - FE \}$.
 Take $\kappa: I \longrightarrow \mathbb{C} G \oplus (V \otimes \mathbb{C} G)$ to be a $\mathbb{C}$-linear map. The filtered $\mathbb{C}$-algebra $U_{q}(\mathfrak{sl}^{*}_2,\kappa)$ is defined to be the following quotient algebra
\begin{eqnarray}\label{definition-of-relations-deformations-of-prequantum-group}
U_{q}(\mathfrak{sl}^{*}_2,\kappa) = \frac{T_\mathbb{C}(V) \sharp \mathbb{C} G}{\langle EF - FE - \kappa(EF - FE)\rangle},
\end{eqnarray}
where the degree of each element in $\mathbb{C} G$ is 0, and the degrees of $E$ and $F$ are both 1. In fact, $U_{q}(\mathfrak{sl}^{*}_2,\kappa)$ can be presented by four generators $K, K^{-1}, E, F$ and the relations
\begin{gather*}\label{relations-deformations-of-prequantum-group}
  KK^{-1}=K^{-1}K=1,\ \ KE=q^2EK,\ \ KF=q^{-2}FK,\ \ EF - FE = \kappa(EF - FE).
\end{gather*}

Now we characterize when $U_{q}(\mathfrak{sl}^{*}_2,\kappa)$ is a PBW-deformation of $U_{q}(\mathfrak{sl}^{*}_2)$.

\begin{thm}\label{theorem-about-determination-of-PBW-deformations-of-new-type-quantum-groups}
The filtered $\mathbb{C}$-algebra $U_{q}(\mathfrak{sl}^{*}_2,\kappa)$ is a PBW-deformation of $U_{q}(\mathfrak{sl}^{*}_2)$ if and only if
$\kappa(EF - FE) = {\mathop\sum\limits_{l=-m}^{m}} a_l K^{l},$
where $m \in \mathbb{Z}^{\geq0}$ and $a_l\in \mathbb{C}~(-m \leq l \leq m)$.
\end{thm}

\begin{proof}
Set
$f_a = {\mathop\sum\limits_{l=-m_1}^{m_1}} a_l K^{l},
f_b = {\mathop\sum\limits_{l=-m_2}^{m_2}} b_l K^{l}$ and
$f_c = {\mathop\sum\limits_{l=-m_3}^{m_3}} c_l K^{l}$
with $m_i \in \mathds{Z}^{\geq0}~(1\leq i \leq 3)$.
Then we can always assume
$\kappa(EF - FE) = f_a + E \otimes f_b + F \otimes f_c.$
It follows from Theorem 3.1 in \cite{WW2014} that $U_{q}(\mathfrak{sl}^{*}_2,\kappa)$ is a PBW-deformation of $U_{q}(\mathfrak{sl}^{*}_2)$ if and only if
\begin{eqnarray}\label{theorem-about-determination-of-PBW-deformations-of-new-type-quantum-groups-eq-proof}
\kappa(K\cdot (EF - FE)) = K\cdot \kappa(EF - FE),
\end{eqnarray}
where $\cdot$ on the right is the left $\mathbb{C} G$-adjoint action on $\mathbb{C} G \oplus (V \otimes \mathbb{C} G)$ (see \cite[Section 1A]{WW2014}).
Since
\begin{eqnarray*}\label{}
\kappa(K\cdot (EF - FE)) &=& \kappa[(K\cdot E)(K\cdot F) - (K\cdot F)(K\cdot E)] \\
&=& \kappa[(q^2 E)(q^{-2} F) - (q^{-2} F)(q^2 E)]\\
&=& \kappa(EF - FE)\\
&=& f_a + E \otimes f_b + F \otimes f_c,\\
K\cdot \kappa(EF - FE) &=& K\cdot (f_a + E \otimes f_b + F \otimes f_c)\\
&=& f_a + q^2 E \otimes f_b + q^{-2} F \otimes f_c,
\end{eqnarray*}
and $q \in \mathbb{C}^\times$ is not a root of unity, then  $(\ref{theorem-about-determination-of-PBW-deformations-of-new-type-quantum-groups-eq-proof})$ holds if and only if $f_b = f_c =0.$
\end{proof}

The notion of Hopf PBW-deformation for $U_{q}(\mathfrak{sl}^{*}_2)$ is introduced below.

\begin{defn}
A PBW-deformation $U_{q}(\mathfrak{sl}^{*}_2,\kappa)$ of $U_{q}(\mathfrak{sl}^{*}_2)$ is called a Hopf PBW-deformation if it is a Hopf algebra
with $K,K^{-1}$ grouplike elements and $E,F$ skew-primitive elements.
\end{defn}

A corollary of Theorem \ref{theorem-about-determination-of-PBW-deformations-of-new-type-quantum-groups} answers when $U_{q}(\mathfrak{sl}^{*}_2,\kappa)$ is a Hopf PBW-deformation of $U_{q}(\mathfrak{sl}^{*}_2)$.

\begin{cor}\label{Hopf-PBW-deformation}
The algebra $U_{q}(\mathfrak{sl}^{*}_2,\kappa)$
 is a Hopf PBW-deformation of $U_{q}(\mathfrak{sl}^{*}_2)$ if and only if
 \begin{eqnarray}\label{Hopf-PBW-deformation-algebra-def-relations-determined-by-k}
\kappa(EF - FE) = a (K^m - K^{-m}),
\end{eqnarray}
where $a \in \mathbb{C}$ and $m\in \mathbb{Z}$. In this case, the Hopf algebra structure maps of $U_{q}(\mathfrak{sl}^{*}_2,\kappa)$ are given by
\begin{gather}
\Delta(K)=K \otimes K, \quad \Delta(E)=E \otimes K^t + K^s \otimes E, \quad \Delta(F)=F \otimes K^{-s} + K^{-t} \otimes F,\nonumber\\
\label{structure-maps-Hopf-PBW-deformation}\varepsilon(K) = 1, \quad \varepsilon(E) = \varepsilon(F) = 0,\\
S(K) = K^{-1}, \quad S(E) = -K^{-s} E K^{-t}, \quad S(F) = -K^t F K^s,\nonumber
\end{gather}
where $s,t\in \mathbb{Z}$ with $t-s=m$.
\end{cor}

\begin{proof}
Ji et al. in \cite[Proposition 3.3]{JW} showed that when $\kappa(EF - FE) = {\mathop\sum\limits_{l=-m}^{m}} a_l K^{l}$, the algebra $U_{q}(\mathfrak{sl}^{*}_2,\kappa)$ is a Hopf algebra with $K,K^{-1}$ grouplike elements and $E,F$ skew-primitive elements if and only if
$(\ref{Hopf-PBW-deformation-algebra-def-relations-determined-by-k})$ holds, and they explicitly described the Hopf algebra structure on $U_{q}(\mathfrak{sl}^{*}_2,\kappa)$ by the formulas in $(\ref{structure-maps-Hopf-PBW-deformation})$. Then this corollary follows from Theorem \ref{theorem-about-determination-of-PBW-deformations-of-new-type-quantum-groups}.
\end{proof}

Up to isomorphisms, we can classify all the nontrivial Hopf PBW-deformations of $U_{q}(\mathfrak{sl}^{*}_2)$.

\begin{prop}\label{isomorphic-relations-among-Hopf-PBW-deformations}
For any pair $(s,t)$ with $s,t\in \mathbb{Z}$ and $t-s=m$. Denote by $U^{s,t}_{q}(\mathfrak{sl}^{*}_2,\kappa)$ the nontrivial Hopf PBW-deformation of $U_{q}(\mathfrak{sl}^{*}_2)$ defined by
$(\ref{Hopf-PBW-deformation-algebra-def-relations-determined-by-k})$ and $(\ref{structure-maps-Hopf-PBW-deformation})$. Then as a Hopf algebra $$U^{s,t}_{q}(\mathfrak{sl}^{*}_2,\kappa) \cong U^{0,m}_{q}(\mathfrak{sl}^{*}_2,\kappa_m),$$
where $\kappa_m(EF - FE) = \frac{K^m - K^{-m}}{q^m - q^{-m}}$.
\end{prop}

\begin{proof}
To avoid confusion, denote by $K, {K}^{-1}, E, F$ (resp. $K^\prime, {K^\prime}^{-1}, E^\prime, F^\prime$) the generators of $U^{s,t}_{q}(\mathfrak{sl}^{*}_2,\kappa)$ (resp. $U^{0,m}_{q}(\mathfrak{sl}^{*}_2,\kappa_m)$). Set $a_m = \frac{1}{q^m - q^{-m}}$ and define an algebra homomorphism $\vartheta: U^{0,m}_{q}(\mathfrak{sl}^{*}_2,\kappa_m) \longrightarrow U^{s,t}_{q}(\mathfrak{sl}^{*}_2,\kappa)$ via
\begin{eqnarray*}
\vartheta(K^\prime) = K,\ \ \
\vartheta(E^\prime) = (\frac{a_m}{a})^{\frac{1}{2}} K^{-s} E,\ \
\vartheta(F^\prime) = (\frac{a_m}{a})^{\frac{1}{2}} F K^{s}.
\end{eqnarray*}
It is routine to check that $\vartheta: U^{0,m}_{q}(\mathfrak{sl}^{*}_2,\kappa_0) \longrightarrow U^{s,t}_{q}(\mathfrak{sl}^{*}_2,\kappa)$ is a Hopf algebra isomorphism.
\end{proof}

\begin{remark}\label{classical-is-almost-unique-hopf-pbw-deformation}
Corollary $\ref{Hopf-PBW-deformation}$ shows that the Drinfeld-Jimbo quantum group $U_{q}(\mathfrak{sl}_2)$ is exactly the nontrivial Hopf PBW-deformation $U^{0,1}_{q}(\mathfrak{sl}^{*}_2,\kappa_1)$ of $U_{q}(\mathfrak{sl}^{*}_2)$,
while Proposition \ref{isomorphic-relations-among-Hopf-PBW-deformations} shows that each nontrivial Hopf PBW-deformation of $U_{q}(\mathfrak{sl}^{*}_2)$ can be reduced to the form $U^{0,m}_{q}(\mathfrak{sl}^{*}_2,\kappa_m)$ up to isomorphism.
Note that $U^{0,m}_{q}(\mathfrak{sl}^{*}_2,\kappa_m)$ has a Hopf subalgebra generated by $K^m,K^{-m},E,F$ which is just isomorphic to $U_{q^m}(\mathfrak{sl}_2)$.
Therefore, in a sense we can say that the Drinfeld-Jimbo quantum group is almost the unique nontrivial Hopf PBW-deformation of $U_{q}(\mathfrak{sl}^{*}_2)$.
\end{remark}

From now on, we can always assume that
\begin{eqnarray*}\label{}
U_{q}(\mathfrak{sl}^{*}_2,\kappa) = \left\{
 \begin{array}{ll}
U_{q}(\mathfrak{sl}^{*}_2),~\quad~&{\rm if}~\kappa = 0,\\
U^{0,m}_{q}(\mathfrak{sl}^{*}_2,\kappa_m),~\quad~&{\rm if}~\kappa \neq 0.
\end{array}
\right.
\end{eqnarray*}

\section{Non-semisimplicity of the category of finite dimensional modules of $U_{q}(\mathfrak{sl}^{*}_2)$}\label{section-2}

In this section, we will prove that the category $U_{q}(\mathfrak{sl}^{*}_2)\mbox{-}{\rm \bf mod}$ of finite dimensional $U_{q}(\mathfrak{sl}^{*}_2)$-modules is not semisimple, which shows a great difference between $U_{q}(\mathfrak{sl}^{*}_2)$ and $U_{q}(\mathfrak{sl}_2)$.

Analogous as the $U_{q}(\mathfrak{sl}_2)$ case \cite[Proposition 2.1]{Jan96}, we can similarly verify that the actions of $E$ and $F$ on any finite dimensional $U_{q}(\mathfrak{sl}^{*}_2,\kappa)$-module are nilpotent.

\begin{lem}\label{E-F-actions-are-nilpotent}
Let $M$ be a finite dimensional $U_{q}(\mathfrak{sl}^{*}_2,\kappa)$-module. Then there exist integers $r,s > 0$ such that $E^r M = 0$ and $F^s M = 0$.
\end{lem}
%

In the following proposition, we classify all the finite dimensional simple $U_{q}(\mathfrak{sl}^{*}_2)$-modules.

\begin{prop}\label{simple-modules-of-new-type-quantum-groups}
Let $M$ be a finite dimensional simple $U_{q}(\mathfrak{sl}^{*}_2)$-module. Then ${\rm dim}(M) = 1$ and the module structure on $M$ can be given as follows
\begin{eqnarray}\label{simple-module}
K v_0 = \lambda v_0,\quad E v_0 = F v_0 = 0,
\end{eqnarray}
where $\lambda \in \mathbb{C}^\times$ and $\{ v_0 \}$ is a basis of $M$.
\end{prop}

\begin{proof}
By Lemma $\ref{E-F-actions-are-nilpotent}$, we know that 0 is an eigenvalue of the linear transformation $E: M \longrightarrow M$.
Set $M_0 = \left\{v \in M | E v = 0 \right\}$. Since for any $v \in M_0$ one has
 $E F v = F E v = 0$ and $E K v = q^{-2} K E v = 0,$
then $M_0$ is a submodule of $M$. Similarly,
$N_0 = \left\{v \in M_0 | F v = 0 \right\}$ is also a submodule of $M$.
Assume that $\lambda \in \mathbb{C}^\times$ is an eigenvalue of the linear transformation $K: N_0 \longrightarrow N_0$
 and $v_0 \in N_0$ is an eigenvector of $K$ corresponding to $\lambda$, i.e., $K v_0 = \lambda v_0$,
 then $L = \mathbb{C} v_0$ is also a submodule of $M$. Since $M$ is simple, then $M = L$. The proof is finished.
\end{proof}

From now on, we denote by $\mathbb{C}_\lambda$ the simple $U_{q}(\mathfrak{sl}^{*}_2)$-module defined by $(\ref{simple-module})$.
The extension group ${\rm Ext}^1_{U_{q}(\mathfrak{sl}^{*}_2)}(\mathbb{C}_\lambda,\mathbb{C}_\mu)$ of extensions of $\mathbb{C}_\mu$ by $\mathbb{C}_\lambda$
can be characterized as follows.

\begin{prop}\label{1st-ext-group-of-simple-modules}
Assume that $\lambda, \mu \in \mathbb{C}^\times$.\\
${\rm (1)}$ If $\lambda \neq q^2 \mu$ and $\lambda \neq q^{-2} \mu$, then one has
${\rm Ext}^1_{U_{q}(\mathfrak{sl}^{*}_2)}(\mathbb{C}_\lambda,\mathbb{C}_\mu) \cong \left\{
 \begin{array}{ll}
0,~\quad~&{\rm if}~\lambda \neq \mu,\\
\mathbb{C},~\quad~&{\rm if}~\lambda = \mu.
\end{array}
\right.$\\
${\rm (2)}$ If $\lambda = q^2 \mu$ or $\lambda = q^{-2} \mu$, then one has ${\rm Ext}^1_{U_{q}(\mathfrak{sl}^{*}_2)}(\mathbb{C}_\lambda,\mathbb{C}_\mu) \cong \mathbb{C}$.
\end{prop}

\begin{proof}
 Consider a short exact sequence of $U_{q}(\mathfrak{sl}^{*}_2)$-modules
 \begin{eqnarray*}\label{E10.3.1}
\xymatrix{
  0\ar[r] & \mathbb{C}_\lambda \ar[r]^{\sigma} & N  \ar[r]^{\tau} & \mathbb{C}_\mu \ar[r]^{} & 0.
  }
  \end{eqnarray*}
Set $\mathbb{C}_\lambda = \mathbb{C} v_0$ and $\mathbb{C}_\mu = \mathbb{C} w_0$. Then we can choose a basis $\{v_1,v_2\}$ of $N$ such that $\sigma (v_0) = v_1$ and $\tau (v_2) = w_0$. By $(\ref{simple-module})$ one has
\begin{eqnarray}\label{2-dim-ind-rep-action}
\left\{
 \begin{array}{ll}
K v_1 = \lambda v_1,\\
K v_2 = f(K)v_1 + \mu v_2,\\
E v_1 = F v_1 = 0,\\
E {v_2} = f(E) {v_1}, \\
F {v_2} = f(F) {v_1},
\end{array}
\right.
\end{eqnarray}
where $f(K), f(E),f(F)\in \mathbb{C}$. Obviously
$\left\{
 \begin{array}{ll}
K E v_1 = q^2 E K v_1 = 0,\\
K F v_1 = q^{-2} F K v_1 = 0,\\
E F v_i = F E v_i = 0~(i=1,2).
\end{array}
\right.$
Moreover, the condition
$\left\{
 \begin{array}{ll}
K E v_2 = q^2 E K v_2,\\
K F v_2 = q^{-2} F K v_2
\end{array}
\right.$
is equivalent to
\begin{eqnarray}\label{2-dim-rep-condition}
\left\{
 \begin{array}{ll}
f(E)(\lambda - q^2 \mu) = 0,\\
f(F)(\lambda - q^{-2} \mu) = 0.
\end{array}
\right.
\end{eqnarray}

${\rm (1)}$ If $\lambda \neq q^2 \mu$ and $\lambda \neq q^{-2} \mu$, then by $(\ref{2-dim-rep-condition})$ we obtain that $f(E) = f(F) = 0.$
It follows from $(\ref{2-dim-ind-rep-action})$ that the 2-dimensional $U_{q}(\mathfrak{sl}^{*}_2)$-module $N$ can be given as follows
\begin{eqnarray}\label{2-dim-ind-rep-action-case-1}
\left\{
 \begin{array}{ll}
K v_1 = \lambda v_1,\\
K v_2 = f(K)v_1 + \mu v_2,\\
E v_i = F v_i = 0~(i=1,2).
\end{array}
\right.
\end{eqnarray}
Now we know that $N$ is completely determined by $f(K),$ so we rewrite it as $N_{\lfloor f(K)\rfloor}$.

When $\lambda \neq \mu$, we can obtain the following equivalence between two extensions of $\mathbb{C}_\lambda$ by $\mathbb{C}_\mu$
\begin{eqnarray*}\label{}
\xymatrix{
  0\ar[r] & \mathbb{C}_\lambda \ar[d]_{1} \ar[rr]^{\tiny \left[ {\begin{array}{*{20}c}
   1\\
   0\\
\end{array}} \right]} && N_{\lfloor f(K)\rfloor} \ar[d]_{\tiny \left[ {\begin{array}{*{20}c}
   1 & \frac{f(K)}{\lambda-\mu}\\
   0 & 1\\
\end{array}} \right]} \ar[rr]^{\tiny \left[ {\begin{array}{*{20}c}
   0 & 1\\
\end{array}} \right]} && \mathbb{C}_\mu \ar[d]_{1} \ar[r]^{} & 0\\
 0 \ar[r]& \mathbb{C}_\lambda \ar[rr]_{\tiny \left[ {\begin{array}{*{20}c}
   1\\
   0\\
\end{array}} \right]} && \mathbb{C}_\lambda \oplus \mathbb{C}_\mu \ar[rr]_{\tiny \left[ {\begin{array}{*{20}c}
   0 & 1\\
\end{array}} \right]} && \mathbb{C}_\mu  \ar[r]^{} & 0.  }
  \end{eqnarray*}
Therefore, ${\rm Ext}^1_{U_{q}(\mathfrak{sl}^{*}_2)}(\mathbb{C}_\lambda,\mathbb{C}_\mu) = 0$.

When $\lambda = \mu$, it is easy to check that two extensions of $\mathbb{C}_\lambda$ by $\mathbb{C}_\mu$
 \begin{eqnarray*}\label{E1111111}
\xymatrix{
  0\ar[r] & \mathbb{C}_\lambda \ar[r]^{\sigma_a} & N_{\lfloor a\rfloor}  \ar[r]^{\tau_a} & \mathbb{C}_\mu \ar[r]^{} & 0 \quad {\rm and} \quad 0\ar[r] & \mathbb{C}_\lambda \ar[r]^{\sigma_b} & N_{\lfloor b\rfloor}  \ar[r]^{\tau_b} & \mathbb{C}_\mu \ar[r]^{} & 0
  }
  \end{eqnarray*}
are equivalent if and only if $a = b$.
Therefore, ${\rm Ext}^1_{U_{q}(\mathfrak{sl}^{*}_2)}(\mathbb{C}_\lambda,\mathbb{C}_\lambda) = \mathbb{C}$.

${\rm (2)}$ We will only prove this statement for $\lambda = q^2 \mu$, since the other case can be similarly done.
If $\lambda = q^2 \mu$, then $\lambda \neq q^{-2} \mu$ (because $q$ is not a root of unity).
It follows from $(\ref{2-dim-rep-condition})$ that $f(F) = 0.$
Combining with $(\ref{2-dim-ind-rep-action})$,  the 2-dimensional $U_{q}(\mathfrak{sl}^{*}_2)$-module $N$ can be given as follows
\begin{eqnarray}\label{2-dim-ind-rep-action-case-2-1}
\left\{
 \begin{array}{ll}
K v_1 = \lambda v_1,\\
K v_2 = f(K)v_1 + \mu v_2,\\
E v_1 = F v_1 = F v_2 = 0,\\
E {v_2} = f(E) {v_1}.
\end{array}
\right.
\end{eqnarray}
In this case, $N$ is completely determined by $f(K)$ and $f(E)$, so we rewrite it as $N_{\lceil f(K), f(E)\rceil}$.
In particular, when $f(K) = 0$, we rewrite $N$ as $N_{\lceil f(E)\rceil}$.

Fix $f(E) \in \mathbb{C}$, for any $f(K) \in \mathbb{C}$ we have the following commutative diagram
\begin{eqnarray*}\label{}
\xymatrix{
  0\ar[r] & \mathbb{C}_\lambda \ar[d]_{1} \ar[rr]^{\tiny \left[ {\begin{array}{*{20}c}
   1\\
   0\\
\end{array}} \right]} && N_{\lceil f(K), f(E)\rceil} \ar[d]_{\tiny \left[ {\begin{array}{*{20}c}
   1 & \frac{f(K)}{\lambda-\mu}\\
   0 & 1\\
\end{array}} \right]} \ar[rr]^{\tiny \left[ {\begin{array}{*{20}c}
   0 & 1\\
\end{array}} \right]} && \mathbb{C}_\mu \ar[d]_{1} \ar[r]^{} & 0\\
 0 \ar[r]& \mathbb{C}_\lambda \ar[rr]_{\tiny \left[ {\begin{array}{*{20}c}
   1\\
   0\\
\end{array}} \right]} && N_{\lceil f(E)\rceil} \ar[rr]_{\tiny \left[ {\begin{array}{*{20}c}
   0 & 1\\
\end{array}} \right]} && \mathbb{C}_\mu  \ar[r]^{} & 0.  }
  \end{eqnarray*}
On the other hand, it is easy to check that two extensions of $\mathbb{C}_\lambda$ by $\mathbb{C}_\mu$
 \begin{eqnarray*}\label{E1111111}
\xymatrix{
  0\ar[r] & \mathbb{C}_\lambda \ar[r]^{\sigma_a} & N_{\lceil a \rceil}  \ar[r]^{\tau_a} & \mathbb{C}_\mu \ar[r]^{} & 0 \quad {\rm and} \quad 0\ar[r] & \mathbb{C}_\lambda \ar[r]^{\sigma_b} & N_{\lceil b\rceil}  \ar[r]^{\tau_b} & \mathbb{C}_\mu \ar[r]^{} & 0
  }
  \end{eqnarray*}
are equivalent if and only if $a = b$.
Therefore, ${\rm Ext}^1_{U_{q}(\mathfrak{sl}^{*}_2)}(\mathbb{C}_\lambda,\mathbb{C}_{q^{-2}\lambda}) = \mathbb{C}$.
\end{proof}

The next corollary follows immediately from Proposition $\ref{1st-ext-group-of-simple-modules}$ or Remark \ref{remarks-about-rep-cat-and-weight-rep-cat}.

\begin{cor}\label{non-semisimplicity-of-reps-of-prequantum-group}
The category $U_{q}(\mathfrak{sl}^{*}_2)\mbox{-}{\rm \bf mod}$ of finite dimensional $U_{q}(\mathfrak{sl}^{*}_2)$-modules is non-semisimple.
\end{cor}

\section{The category $U_{q}(\mathfrak{sl}^{*}_2,\kappa)\mbox{-}{\rm \bf mod}_{\rm wt}$ of finite dimensional weight modules of $U_{q}(\mathfrak{sl}^{*}_2,\kappa)$}\label{section-3}

In this section, we investigate the finite dimensional weight representation theory of $U_{q}(\mathfrak{sl}^{*}_2,\kappa)$.

\subsection{Finite dimensional weight module and weight group of $U_{q}(\mathfrak{sl}^{*}_2,\kappa)$}\label{section-3-1}

\begin{defn}\label{definiton-of-weight-rep-U}
Let $M$ be a finite dimensional $U_{q}(\mathfrak{sl}^{*}_2,\kappa)$-module.\\
${\rm (1)}$ The nontrivial linear space
\begin{eqnarray*}\label{}
M_\lambda = \left\{v \in M | K v = \lambda v\right\},
\end{eqnarray*}
i.e., the eigenspace of $K$ acting on $M$ for the eigenvalue $\lambda$, is called a weight space of $M$.\\
${\rm (2)}$ If $M$ is the direct sum of its weight spaces, then we call $M$ a weight module of $U_{q}(\mathfrak{sl}^{*}_2,\kappa)$.\\
${\rm (3)}$ Let $M$ be a finite dimensional weight module of $U_{q}(\mathfrak{sl}^{*}_2,\kappa)$, and denote by $$\Lambda_M = \{\lambda_1, \lambda_2, \cdots, \lambda_m\}$$ the set of all mutually different eigenvalues of $K$ acting on $M$. We call $\Lambda_M$ the weight set of $M$.
\end{defn}

 Denote by $U_{q}(\mathfrak{sl}^{*}_2,\kappa)\mbox{-}{\rm \bf mod}$ the category of finite dimensional modules of $U_{q}(\mathfrak{sl}^{*}_2,\kappa)$. The notation $U_{q}(\mathfrak{sl}^{*}_2,\kappa)\mbox{-}{\rm \bf mod}_{\rm wt}$ represents the category of finite dimensional weight modules of $U_{q}(\mathfrak{sl}^{*}_2,\kappa)$ which is a full subcategory of $U_{q}(\mathfrak{sl}^{*}_2,\kappa)\mbox{-}{\rm \bf mod}$.

\begin{defn}\label{twist-of-weight-representation}
Let $M$ be an object in the category $U_{q}(\mathfrak{sl}^{*}_2,\kappa)\mbox{-}{\rm \bf mod}_{\rm wt}$. For any given automorphism $\sigma$ of $U_{q}(\mathfrak{sl}^{*}_2,\kappa)$, we can define a new $U_{q}(\mathfrak{sl}^{*}_2,\kappa)$-module $M^\sigma$ as follows.\\
${\rm (1)}$ As a linear space, $M^\sigma = M$.\\
${\rm (2)}$ The action of $U_{q}(\mathfrak{sl}^{*}_2,\kappa)$ on $M^\sigma$ is given by
  \begin{eqnarray*}\label{}
 \circ_\sigma: U_{q}(\mathfrak{sl}^{*}_2,\kappa) \otimes M^\sigma &\longrightarrow& M^\sigma,\\
 K \otimes m &\longmapsto& K \circ_\sigma m = \sigma(K)m,\\
  E \otimes m &\longmapsto& E \circ_\sigma m = \sigma(E)m,\\
   F \otimes m &\longmapsto& F \circ_\sigma m = \sigma(F)m.
\end{eqnarray*}
The module $M^\sigma$ is called a $\sigma$-twist of $M$.
\end{defn}

 It is easy to see that the category $U_{q}(\mathfrak{sl}^{*}_2,\kappa)\mbox{-}{\rm \bf mod}_{\rm wt}$ has the following properties.

 \begin{prop}\label{properties-of-weight-rep-cat}
 The category $U_{q}(\mathfrak{sl}^{*}_2,\kappa)\mbox{-}{\rm \bf mod}_{\rm wt}$ is closed under taking submodule, quotient, direct summand, direct sum, tensor product and $\sigma_0$-twist,
where $\sigma_0$ is the involutive automorphism of $U_{q}(\mathfrak{sl}^{*}_2,\kappa)$ defined by $\sigma_0(K) = K^{-1}, \sigma_0(E) = F$ and $\sigma_0(F) = E$.
 \end{prop}

 \begin{proof}
 For any object $M$ in $U_{q}(\mathfrak{sl}^{*}_2,\kappa)\mbox{-}{\rm \bf mod}$, denote by $m_M(x)$ the minimal polynomial of
 $K$ acting on $M$.
 It is well known that $M$ is a weight module if and only if $m_M(x)$ has no multiple roots.
  Let $L$ be any submodule of $M$.
 Since $m_L(x)$ and $m_{M / L}(x)$ are factors of $m_M(x)$, then $U_{q}(\mathfrak{sl}^{*}_2,\kappa)\mbox{-}{\rm \bf mod}_{\rm wt}$ is closed under taking submodule, direct summand and quotient.
Now let $M$ and $N$ be two weight modules of $U_{q}(\mathfrak{sl}^{*}_2,\kappa)$ with weight sets $\Lambda_M$ and $\Lambda_N$. Noting that
 \begin{eqnarray*}\label{}
 \left\{
 \begin{array}{ll}
m_{M\oplus N}(x) = {\mathop\prod\limits_{\lambda\in \Lambda_M \cup\Lambda_N}}(x - \lambda),\\
m_{M\otimes N}(x) = {\mathop\prod\limits_{\lambda\in \Lambda_M \diamond\Lambda_N}}(x - \lambda),\\
m_{M^{\sigma_0}}(x) = {\mathop\prod\limits_{\lambda\in \Lambda_M}}(x - \lambda^{-1}),
\end{array}
\right.
\end{eqnarray*}
where $\Lambda_M \diamond\Lambda_N = \{\mu\nu \mid \mu\in \Lambda_M, \nu\in \Lambda_N\}$,
 we obtain that $U_{q}(\mathfrak{sl}^{*}_2,\kappa)\mbox{-}{\rm \bf mod}_{\rm wt}$ is closed under taking direct sum, tensor product and twist.
 \end{proof}

For a given Hopf PBW-deformation $U_{q}(\mathfrak{sl}^{*}_2,\kappa)$, define a subset $\Lambda_{\kappa}$ of $\mathbb{C}^\times$ as follows
\begin{eqnarray*}\label{}
\Lambda_{\kappa} = \left\{\lambda \in \mathbb{C} | \exists ~{\rm an~object}~M~{\rm in}~U_{q}(\mathfrak{sl}^{*}_2,\kappa)\mbox{-}{\rm \bf mod}_{\rm wt}~{\rm such~that}~\lambda\in \Lambda_M\right\}.
\end{eqnarray*}
We call $\Lambda_{\kappa}$ the weight set of $U_{q}(\mathfrak{sl}^{*}_2,\kappa)$. Then we can check that $\Lambda_{\kappa}$ is a subgroup of $\mathbb{C}^\times$.

\begin{cor}\label{weight-set-is-a-group}
The weight set $\Lambda_{\kappa}$ of $U_{q}(\mathfrak{sl}^{*}_2,\kappa)$ is always a subgroup of $\mathbb{C}^\times$.
\end{cor}

\begin{proof}
Since $K$ is invertible, then $\Lambda_{\kappa} \subseteq \mathbb{C}^\times$. There always exists a 1-dimensional module $\mathbb{C}_1 = \mathbb{C} v$ of $U_{q}(\mathfrak{sl}^{*}_2,\kappa)$ with the actions given by
$K v = v$ and $Ev = F v = 0$.
Noting that the weight set $\Lambda_{\mathbb{C}_1}$ of $\mathbb{C}_1$ is equal to $\{1\}$, we deduce that $1\in \Lambda_{\kappa}$ which implies that $\Lambda_{\kappa} \neq \emptyset$.

For any $\mu,\nu \in \Lambda_{\kappa}$, there exist two objects $M$ and $N$ in the category $U_{q}(\mathfrak{sl}^{*}_2,\kappa)\mbox{-}{\rm \bf mod}_{\rm wt}$ such that $\mu \in \Lambda_M$ and $\nu \in \Lambda_N$. By Proposition $\ref{properties-of-weight-rep-cat}$, we know that $\mu \nu \in \Lambda_{M\otimes N}$ and $\mu^{-1} \in \Lambda_{M^\sigma}$, which means that $\Lambda_{\kappa}$ is closed under taking multiplication and inverse.
\end{proof}

\begin{defn}\label{def-of-weight-group}
The weight set $\Lambda_{\kappa}$ of $U_{q}(\mathfrak{sl}^{*}_2,\kappa)$ is called the weight group of $U_{q}(\mathfrak{sl}^{*}_2,\kappa)$.
\end{defn}

Next we will explicitly characterize the weight group of $U_{q}(\mathfrak{sl}^{*}_2,\kappa)$.
For this purpose, we need to prove the following statement.

\begin{prop}\label{reps-are-all-weight-reps-for-classical-quantum-groups}
Suppose that $\kappa \neq 0$. Let $M$ be a finite dimensional $U_{q}(\mathfrak{sl}^{*}_2,\kappa)$-module. Then\\
${\rm (1)}$  $M$ is the direct sum of its weight spaces.\\
 ${\rm (2)}$  All weights of $M$ have the form $q^{c } \omega$, where $c\in \mathbb{Z}$ and $\omega^{2m} = 1$.
\end{prop}

\begin{proof}
Let ${U}_{q^m}(\mathfrak{sl}_2)$ be the subalgebra of $U_{q}(\mathfrak{sl}^{*}_2,\kappa)$ generated by $K^m, K^{-m}, E, F$.
Here ${U}_{q^m}(\mathfrak{sl}_2)$ is just the Drinfeld-Jimbo type quantum group.
By Lemma $\ref{E-F-actions-are-nilpotent}$, there exists an integer $s>0$ such that $F^s M = 0$.
Consider $M$ as a ${U}_{q^m}(\mathfrak{sl}_2)$-module, by Proposition 2.3 in \cite[Section 2.3]{Jan96} and its proof, one has
\begin{eqnarray*}\label{}
\left({\mathop\prod\limits_{j=-(s-1)}^{s-1}}(K^m - q^{-jm})(K^m + q^{-jm})\right)M = 0.
\end{eqnarray*}
So the minimal polynomial $m_M(x)$ of $K$ acting on $M$ divides
${\mathop\prod\limits_{j=-(s-1)}^{s-1}}(x^m - q^{-jm})(x^m + q^{-jm}).$
Hence $m_M(x)$ can split into linear factors $x-q^{c } \omega~(c\in \mathbb{Z},\omega^{2m} = 1)$ with each occurring at most once.
Therefore, the statements in ${\rm (1)}$ and ${\rm (2)}$ hold.
\end{proof}

\begin{remark}\label{remarks-about-rep-cat-and-weight-rep-cat}
Proposition $\ref{reps-are-all-weight-reps-for-classical-quantum-groups}$ shows that the categories $U_{q}(\mathfrak{sl}^{*}_2,\kappa)\mbox{-}{\rm \bf mod}$ and $U_{q}(\mathfrak{sl}^{*}_2,\kappa)\mbox{-}{\rm \bf mod}_{\rm wt}$ are the same when $\kappa \neq 0$.
However, they are different when $\kappa =0$. In fact,
each finite dimensional indecomposable module of $\mathbb{C}[K,K^{-1}]$ can induce an indecomposable object in $U_{q}(\mathfrak{sl}^{*}_2)\mbox{-}{\rm \bf mod}$ with the actions of $E$ and $F$ trivial.
For any integer $n \geq 1$ and $\lambda \in \mathbb{C}^\times$, there always exists a $(n+1)$-dimensional module $M(\lambda,n) = {\mathop\bigoplus\limits_{i=0}^{n}} \mathbb{C} v_i$ of $\mathbb{C}[K,K^{-1}]$ with the action of $K$ given by
\begin{eqnarray*}\label{}
K(v_0,v_1,\cdots,v_n) = (v_0,v_1,\cdots,v_n)\left( {\begin{array}{*{20}c}
   \lambda &  &  & &\\
   1 & \lambda &  & &\\
    & 1 & \ddots & &\\
     &  & \ddots & \lambda & \\
     &  &  & 1 & \lambda\\
\end{array}} \right).
\end{eqnarray*}
Then each $M(\lambda,n)$ can induce an object in $U_{q}(\mathfrak{sl}^{*}_2)\mbox{-}{\rm \bf mod}$ but not in $U_{q}(\mathfrak{sl}^{*}_2)\mbox{-}{\rm \bf mod}_{\rm wt}$.
\end{remark}
\begin{cor}\label{explicit-charaterization-of-weight-group-of-quantum-groups}
The weight group $\Lambda_{\kappa}$ of $U_{q}(\mathfrak{sl}^{*}_2,\kappa)$ can be given as follows
\begin{eqnarray*}\label{}
\Lambda_{\kappa} = \left\{
 \begin{array}{ll}
\mathbb{C}^\times,&{\rm if}~\kappa = 0,\\
\left\{q^{c} \omega | c\in \mathbb{Z},~\omega^{2m} = 1\right\},&{\rm if}~\kappa \neq 0.
\end{array}
\right.
\end{eqnarray*}
\end{cor}

\begin{proof}
It follows from Proposition $\ref{simple-modules-of-new-type-quantum-groups}$ that $\Lambda_{\kappa} = \mathbb{C}^\times$ when $\kappa = 0$.
Now assume $\kappa \neq 0$. Proposition $\ref{reps-are-all-weight-reps-for-classical-quantum-groups}$ implies that $\Lambda_{\kappa} \subseteq \left\{q^{c} \omega | c\in \mathbb{Z},~\omega^{2m} = 1\right\}$.
Moreover, for $U_{q}(\mathfrak{sl}^{*}_2,\kappa)$ there exist a 1-dimensional module $M  = \mathbb{C} v$ given by
$K v = \omega v$ and $Ev = F v = 0$,
and a 2-dimensional module $N  = \mathbb{C} v_1 \oplus \mathbb{C} v_2$ given by
 \begin{eqnarray*}\label{}
 \left\{
 \begin{array}{ll}
K v_1 = q^{-1} v_1,\ \  E v_1 = v_2,\ \  F v_1 = 0,\\
K v_2 = q v_2,\ \ \ \ \ E v_2 = 0,\ \ \ F v_2 = v_1.
\end{array}
\right.
\end{eqnarray*}
Hence $\omega \in \Lambda_M \subseteq \Lambda_{\kappa}$ and $q \in \Lambda_N \subseteq \Lambda_{\kappa}$.
By Corollary $\ref{weight-set-is-a-group}$ one gets $\left\{q^{c} \omega | c\in \mathbb{Z},~\omega^{2m} = 1\right\} \subseteq \Lambda_{\kappa}$.
\end{proof}

\subsection{Krull-Schmidt theorem in the category $U_{q}(\mathfrak{sl}^{*}_2,\kappa)\mbox{-}{\rm \bf mod}_{\rm wt}$ }\label{section-3-2}

\begin{defn}\label{definition-of-q2-rep}
${\rm (1)}$ Let $\Lambda = \{\lambda_1, \lambda_2, \cdots, \lambda_l\}$ be a subset of $\mathbb{C}^\times$.
If there exists $\lambda \in \Lambda$ such that
\begin{eqnarray*}\label{}
\Lambda = \left\{\lambda, q^2 \lambda, \cdots, q^{2(l-1)}\lambda\right\},
\end{eqnarray*}
then we call $\Lambda$ a $q^2$-chain.\\
${\rm (2)}$ Let $M$ be a finite dimensional weight module of $U_{q}(\mathfrak{sl}^{*}_2,\kappa)$. If the weight set $\Lambda_M$ of $M$ is a $q^2$-chain, then we call $M$ a $q^2$-chain module of $U_{q}(\mathfrak{sl}^{*}_2,\kappa)$.
\end{defn}

\begin{prop}\label{fd-indecom-weight-reps-are-q2-chain-reps}
Each finite dimensional indecomposable weight module of $U_{q}(\mathfrak{sl}^{*}_2,\kappa)$ is a $q^2$-chain module.
\end{prop}

\begin{proof}
Let $M$ be a finite dimensional indecomposable weight module of $U_{q}(\mathfrak{sl}^{*}_2,\kappa)$.
Define a relation $\sim$ on the weight set $\Lambda_M = \{\lambda_1, \lambda_2, \cdots, \lambda_m\}$ of $M$ as follows:
$\lambda_i \sim \lambda_j \Longleftrightarrow$
there~exist~\[\lambda_i = \lambda_{i_1},\lambda_{i_2},\cdots,\lambda_{i_r}= \lambda_j\in \Lambda_M~\quad{\rm or}~\quad\lambda_j = \lambda_{i_1},\lambda_{i_2},\cdots,\lambda_{i_r}= \lambda_i\in \Lambda_M\]
such~that~$\lambda_{i_{l+1}} = q^2 \lambda_{i_{l}}$~for~$1\leq l \leq r-1$.
 It is easy to check that the relation $\sim$ is an equivalence relation on $\Lambda_M$.
Denote by $\Lambda_M / \sim = \left\{\Lambda_1, \Lambda_2, \cdots, \Lambda_s\right\}$ the set consisting of all the equivalence classes.
Then $\Lambda_M$ can be expressed as a disjoint union of $\Lambda_1, \Lambda_2, \cdots, \Lambda_s$, i.e.,
$\Lambda_M = {\mathop{\overset{\cdot}\cup}\limits_{1\leq i \leq s}} \Lambda_i$.
Noting that each $\Lambda_i (1\leq i \leq s)$ is a $q^2$-chain, then we can assume that
\[\Lambda_i = \left\{\lambda_{i_0}, q^2 \lambda_{i_0}, \cdots, q^{2(m_i - 1)} \lambda_{i_0} \right\}.\]
Set
$M_{\Lambda_i} = {\mathop\bigoplus\limits_{j=1}^{m_i}} M_{q^{2(j - 1)}\lambda_{i_0}},$
then $M_{\Lambda_i}$ is a submodule of $M$ and a $q^2$-chain module of $U_{q}(\mathfrak{sl}^{*}_2,\kappa)$ by Definition $\ref{definition-of-q2-rep}$.
Since $\Lambda_M = {\mathop{\overset{\cdot}\cup}\limits_{1\leq i \leq s}} \Lambda_i$,
then $M = {\mathop\oplus\limits_{i=1}^s} M_{\Lambda_i}$. Because $M$ is indecomposable, we know that $s = 1$. Therefore, $M$ is a $q^2$-chain module of $U_{q}(\mathfrak{sl}^{*}_2,\kappa)$.
\end{proof}

Now we can state and prove the Krull-Schmidt theorem in the category $U_{q}(\mathfrak{sl}^{*}_2,\kappa)\mbox{-}{\rm \bf mod}_{\rm wt}$.

\begin{thm}\label{unique-decomposition-theorem-of-finite-dimensional-reps-of-U}
Each finite dimensional weight module of $U_{q}(\mathfrak{sl}^{*}_2,\kappa)$ can be uniquely decomposed as the direct sum of some indecomposable $q^2$-chain modules of $U_{q}(\mathfrak{sl}^{*}_2,\kappa)$.
\end{thm}

\begin{proof}
The classical Krull-Schmidt theorem (cf. \cite[Section 12.9]{AF1992}) tells us that each finite dimensional weight module $M$ of $U_{q}(\mathfrak{sl}^{*}_2,\kappa)$ has a unique decomposition $M = M_1 \oplus M_2 \oplus \cdots \oplus M_t$ up to a permutation and isomorphisms,
where $M_1,M_2,\cdots,M_t$ are indecomposable modules of $U_{q}(\mathfrak{sl}^{*}_2,\kappa)$.
By Proposition $\ref{properties-of-weight-rep-cat}$, we know that $M_1,M_2,\cdots,M_t$ are weight modules of $U_{q}(\mathfrak{sl}^{*}_2,\kappa)$.
It follows from Proposition $\ref{fd-indecom-weight-reps-are-q2-chain-reps}$ that $M_1,M_2,\cdots,M_t$ are indecomposable $q^2$-chain modules of $U_{q}(\mathfrak{sl}^{*}_2,\kappa)$.
\end{proof}

\subsection{Block decomposition of the category $U_{q}(\mathfrak{sl}^{*}_2,\kappa)\mbox{-}{\rm \bf mod}_{\rm wt}$ }\label{section-3-3}

\begin{defn}
For any $\lambda, \mu \in \mathbb{C}^\times$, if there exists an integer $l \in \mathbb{Z}$
 such that $\lambda = q^{2l} \mu$, then we say that $\lambda$ and $\mu$ are $q^2$-linked, which we denote by $\lambda \overset{q^2}\sim \mu$.
\end{defn}

It is easy to check that the relation $\overset{q^2}\sim$ on the weight group $\Lambda_{\kappa}$ of $U_{q}(\mathfrak{sl}^{*}_2,\kappa)$ is an equivalence relation.
Denote by $\Lambda_{\kappa} / \overset{q^2}\sim$ the set consisting of all the equivalence classes.
$\Lambda_{\kappa} / \overset{q^2}\sim$ can be considered as the quotient group $\Lambda_{\kappa} / \langle q^2 \rangle$, where $\langle q^2 \rangle$
is the cyclic subgroup of $\Lambda_{\kappa}$ generated by $q^2$.
For each equivalence class $[\lambda]$ in $\Lambda_{\kappa} / \overset{q^2}\sim$, we fix its representative element $\lambda$. Denote by
$[\Lambda_{\kappa} / \overset{q^2}\sim]$ the set consisting of all the above fixed representative elements.
 It can be seen from Corollary $\ref{explicit-charaterization-of-weight-group-of-quantum-groups}$ that $[\Lambda_{\kappa} / \overset{q^2}\sim]$ is
 an infinite set when $\kappa = 0$, while $[\Lambda_{\kappa} / \overset{q^2}\sim]$ is a finite set with cardinal $4m$ when $\kappa \neq 0$.
 In fact, when $\kappa \neq 0$, we can choose arbitrary odd integer $t_0$ and even integer $t_1$ such that
 \begin{eqnarray*}\label{}
[\Lambda_{\kappa} / \overset{q^2}\sim] = \left\{q^{t_0} \omega, q^{t_1} \omega | \omega^{2m} =1 \right\}.
\end{eqnarray*}
In either case, we can always assume that
$1\in [\Lambda_{\kappa} / \overset{q^2}\sim]$.

\begin{defn}\label{def-blocks-of-weight-rep-cat-of-U}
For $\lambda \in [\Lambda_{\kappa} / \overset{q^2}\sim]$, we define the category $\mathcal{O}^{\kappa}_\lambda$ to be the full subcategory of $U_{q}(\mathfrak{sl}^{*}_2,\kappa)\mbox{-}{\rm \bf mod}_{\rm wt}$ with objects $M$ satisfying $\Lambda_M \subseteq [\lambda]$. We call $\mathcal{O}^{\kappa}_\lambda$ a block
of $U_{q}(\mathfrak{sl}^{*}_2,\kappa)\mbox{-}{\rm \bf mod}_{\rm wt}$.
\end{defn}

\begin{lem}\label{q-chain-characterization-by-unique}
For each $q^2$-chain $\Lambda$ contained in $\Lambda_{\kappa}$, there exists a unique tripe $(\lambda,k,l)$ such that
\begin{eqnarray*}\label{}
\Lambda = \left\{q^{2i}\lambda | k \leq i \leq l\right\}\subseteq [\lambda] = \left\{ q^{2i}\lambda | i\in \mathbb{Z}\right\},
\end{eqnarray*}
where $\lambda \in [\Lambda_{\kappa} / \overset{q^2}\sim]$ and $k,l\in \mathbb{Z}$ with $k \leq l$.
\end{lem}

\begin{proof}
For any given $q^2$-chain $\Lambda \subseteq \Lambda_{\kappa}$, we can assume that $\Lambda = \left\{\mu, q^2 \mu, \cdots, q^{2(j-1)}\mu\right\},$
where $\mu \in \mathbb{C}^\times$ and $j \in \mathbb{Z}^{\geq 1}$. Since $\mu \in \Lambda \subseteq \Lambda_{\kappa}$, then there exists a unique $\lambda \in [\Lambda_{\kappa} / \overset{q^2}\sim]$ such that $[\lambda] = [\mu]$, from which we deduce that $\lambda$ and $\mu$ are $q^2$-linked. Because $q$ is not a root of unity, there
exists a unique integer $k$ such that $\mu = q^{2k} \lambda$. Therefore, we can choose $l = k+j-1$.
\end{proof}

\begin{prop}\label{indecomposable-representation-in-unique block}
Each indecomposable object in $U_{q}(\mathfrak{sl}^{*}_2,\kappa)\mbox{-}{\rm \bf mod}_{\rm wt}$ lies in a unique block $\mathcal{O}^{\kappa}_\lambda$ of $U_{q}(\mathfrak{sl}^{*}_2,\kappa)\mbox{-}{\rm \bf mod}_{\rm wt}$.
\end{prop}

\begin{proof}
\label{fd-indecom-weight-rep-is-q2-chain-rep}
By Proposition $\ref{fd-indecom-weight-reps-are-q2-chain-reps}$, each indecomposable object $M$ in $U_{q}(\mathfrak{sl}^{*}_2,\kappa)\mbox{-}{\rm \bf mod}_{\rm wt}$ is a $q^2$-chain module. Since the weight set $\Lambda_M$ of $M$ is a $q^2$-chain contained in $\Lambda_{\kappa}$,
by Lemma $\ref{q-chain-characterization-by-unique}$ there exists a unique $\lambda \in [\Lambda_{\kappa} / \overset{q^2}\sim]$ such that $\Lambda_M \subseteq [\lambda]$.
Therefore, this statement follows from Definition \ref{def-blocks-of-weight-rep-cat-of-U}.
\end{proof}

Now we are ready to obtain the following block decomposition of the category $U_{q}(\mathfrak{sl}^{*}_2,\kappa)\mbox{-}{\rm \bf mod}_{\rm wt}$.

\begin{thm}\label{block-decomposition-of-weight-rep-cat-of-U}
The category $U_{q}(\mathfrak{sl}^{*}_2,\kappa)\mbox{-}{\rm \bf mod}_{\rm wt}$ is the direct sum of the subcategories $\mathcal{O}^{\kappa}_\lambda$ as $\lambda$ ranges over the set $[\Lambda_{\kappa} / \overset{q^2}\sim]$,
 i.e., $U_{q}(\mathfrak{sl}^{*}_2,\kappa)\mbox{-}{\rm \bf mod}_{\rm wt} = {\mathop\bigoplus\limits_{\lambda \in [\Lambda_{\kappa} / \overset{q^2}\sim]}} \mathcal{O}^{\kappa}_\lambda.$
\end{thm}

\begin{proof}
By Theorem $\ref{unique-decomposition-theorem-of-finite-dimensional-reps-of-U}$, we know that
each object $M$ in $U_{q}(\mathfrak{sl}^{*}_2,\kappa)\mbox{-}{\rm \bf mod}_{\rm wt}$ can be decomposed as the direct sum of finitely many indecomposable $q^2$-chain modules $M_1,M_2,\cdots,M_t$ of $U_{q}(\mathfrak{sl}^{*}_2,\kappa)$, that is, $M = {\mathop\bigoplus\limits_{i=1}^t} M_i$.
Then it can be seen from Proposition $\ref{indecomposable-representation-in-unique block}$ that each direct summand $M_i$ of $M$
 lies in a unique block of $U_{q}(\mathfrak{sl}^{*}_2,\kappa)\mbox{-}{\rm \bf mod}_{\rm wt}$.
 Therefore, each object in $U_{q}(\mathfrak{sl}^{*}_2,\kappa)\mbox{-}{\rm \bf mod}_{\rm wt}$ can be decomposed as the direct sum of finitely many objects in the blocks of $U_{q}(\mathfrak{sl}^{*}_2,\kappa)\mbox{-}{\rm \bf mod}_{\rm wt}$.

Now assume that $M \in \mathcal{O}^{\kappa}_\lambda$ and $N \in \mathcal{O}^{\kappa}_\mu$ with $\lambda, \mu \in [\Lambda_{\kappa} / \overset{q^2}\sim]$ not $q^2$-linked.
 Next we only need to show that
${\rm Hom}_{U_{q}(\mathfrak{sl}^{*}_2,\kappa)}(M,N) = 0.$
For any $\phi \in {\rm Hom}_{U_{q}(\mathfrak{sl}^{*}_2,\kappa)}(M,N)$, $\nu\in \Lambda_M$ and $w\in M_\nu$, one has $K \phi(w) = \phi(K w) = \phi(\nu w) = \nu \phi(w)$, which means that $\phi(w) \in N_\nu$.
Noting that $[\lambda] \cap [\mu] = \emptyset$ implies that $\nu \notin \Lambda_N$, one has $N_\nu = 0$. Hence $\phi = 0$.
\end{proof}

\begin{remark}
As it is stated in Section 1.13 in \cite{Hum2008}, the blocks of a module category $\mathcal{C}$ which is both artinian and noetherian always have the following properties:\\
${\rm (1)}$ Two simple modules $M$ and $N$ are in the same block if and only if there is a finite sequence
\begin{eqnarray*}\label{short-exact-seq}
M=M_1,M_2,\cdots,M_n=N
\end{eqnarray*}
of simple modules such that ${\rm Ext}^1_{\mathcal{C}} (M_i,M_{i+1}) \neq 0$ or ${\rm Ext}^1_{\mathcal{C}} (M_{i+1},M_{i}) \neq 0$ for any $1 \leq i < n$.\\
${\rm (2)}$ Any module $M$ in $\mathcal{C}$ belongs to a block if all its composition factors do.\\
${\rm (3)}$ Any module $M$ in $\mathcal{C}$ decomposes uniquely as a direct sum of submodules, each belonging to a single block. In particular, each
indecomposable module belongs to a single block.

When $\kappa = 0$, it follows from Proposition $\ref{1st-ext-group-of-simple-modules}$, Proposition $\ref{indecomposable-representation-in-unique block}$ and
Theorem $\ref{block-decomposition-of-weight-rep-cat-of-U}$ that the blocks of $U_{q}(\mathfrak{sl}^{*}_2)\mbox{-}{\rm \bf mod}_{\rm wt}$ in Definition $\ref{def-blocks-of-weight-rep-cat-of-U}$ satisfy the properties described above.
\end{remark}

\subsection{Isomorphisms between blocks of the category $U_{q}(\mathfrak{sl}^{*}_2,\kappa)\mbox{-}{\rm \bf mod}_{\rm wt}$ }\label{section-3-4}
For $U_{q}(\mathfrak{sl}^{*}_2,\kappa)$, we give some natural automorphisms in the following lemma which can be proved by direct verification.

\begin{lem}\label{equivalence-functors-twist-automorphisms}
${\rm (1)}$ For any $\lambda \in \mathbb{C}^\times$, there is a unique automorphism $\tau_\lambda$ of $U_{q}(\mathfrak{sl}^{*}_2)$
defined by
 \begin{eqnarray*}\label{}
\tau_\lambda(K) = \lambda K,\ \ \
\tau_\lambda(E) = E,\ \ \
\tau_\lambda(F) = F.
\end{eqnarray*}
${\rm (2)}$ For any $\omega$ with $\omega^{2m}=1$, there is a unique automorphism $\sigma_\omega$ of $U_{q}(\mathfrak{sl}^{*}_2,\kappa)$
defined by
 \begin{eqnarray*}\label{}
\sigma_\omega(K) = \omega K,\ \ \
\sigma_\omega(E) = \omega^m E,\ \ \
\sigma_\omega(F) = F.
\end{eqnarray*}
\end{lem}

 In virtue of the twists corresponding to the automorphisms given in Lemma \ref{equivalence-functors-twist-automorphisms}, we can define some functors among the blocks of $U_{q}(\mathfrak{sl}^{*}_2,\kappa)\mbox{-}{\rm \bf mod}_{\rm wt}$.

\begin{defn}\label{defs-of-translation-and-transitive-functors}
${\rm (1)}$ Suppose that $\kappa = 0$.\\
${\rm (i)}$ For any $i \in \mathbb{Z}$, we define the $i$-translation functor ${\Sigma}_i$ as follows
\begin{eqnarray*}\label{}
{\Sigma}_i: \mathcal{O}^{0}_1 &\longrightarrow& \mathcal{O}^{0}_1,\\
M &\longmapsto& {\Sigma}_i(M) = M^{\tau_{q^{2i}}},\\
M \overset{f}\longrightarrow N &\longmapsto& \Sigma_i (M) \overset{\Sigma_i (f)}\longrightarrow \Sigma_i (N),
\end{eqnarray*}
where as a map $\Sigma_i (f) = f$. Especially we redenote ${\Sigma}_1$ by $\Sigma$ and call it the translation functor.\\
${\rm (ii)}$
For any $(\mu,\lambda) \in [\Lambda_{\kappa} / \overset{q^2}\sim] \times [\Lambda_{\kappa} / \overset{q^2}\sim]$, we define the $(\mu,\lambda)$-transitive functor $\Theta_{\mu,\lambda}$ as follows
\begin{eqnarray*}\label{}
\Theta_{\mu,\lambda}: \mathcal{O}^{0}_\mu &\longrightarrow& \mathcal{O}^{0}_\lambda,\\
M &\longmapsto& \Theta_{\mu,\lambda}(M) = M^{\tau_{\mu^{-1} \lambda}},\\
M \overset{f}\longrightarrow N &\longmapsto& \Theta_{\mu,\lambda}(M) \xrightarrow{\Theta_{\mu,\lambda} (f)} \Theta_{\mu,\lambda}(N),
\end{eqnarray*}
where as a map $\Theta_{\mu,\lambda} (f) = f$. Especially we redenote $\Theta_{1,\lambda}$ by $\Theta_{\lambda}$ and call it the transitive functor.

${\rm (2)}$ Suppose that $\kappa \neq 0$. Choose
$[\Lambda_{\kappa} / \overset{q^2}\sim] = \left\{\omega, q \omega | \omega^{2m} =1 \right\}.$
For any $(\mu,\lambda) \in [\Lambda_{\kappa} / \overset{q^2}\sim] \times [\Lambda_{\kappa} / \overset{q^2}\sim]$ with $(\mu^{-1} \lambda)^{2m} = 1$,
 we define the $(\mu,\lambda)$-transitive functor $\Upsilon_{\mu,\lambda}$ as follows
\begin{eqnarray*}\label{}
\Upsilon_{\mu,\lambda}: \mathcal{O}^{\kappa}_{\mu} &\longrightarrow& \mathcal{O}^{\kappa}_{\lambda},\\
M &\longmapsto& \Upsilon_{\mu,\lambda}(M) = M^{\sigma_{\mu^{-1} \lambda}},\\
M \overset{f}\longrightarrow N &\longmapsto& \Upsilon_{\mu,\lambda}(M) \xrightarrow{\Upsilon_{\mu,\lambda} (f)} \Upsilon_{\mu,\lambda}(N),
\end{eqnarray*}
where $\Upsilon_{\mu,\lambda} (f) = f$. For $\lambda_0 \in \{1,q\}$ we redenote $\Upsilon_{\lambda_0,\lambda_0 \omega}$ by $\Upsilon^{\lambda_0}_\omega$ and call it the transitive functor.
\end{defn}

Now we can show that the functors defined above are all isomorphisms.

\begin{thm}\label{equivalence-functors-translation-transitive-functors}
${\rm (1)}$ When $\kappa = 0$, the $i$-translation functor ${\Sigma_i}: \mathcal{O}^{0}_1 \longrightarrow \mathcal{O}^{0}_1$ and the $(\mu,\lambda)$-transitive functor $\Theta_{\mu,\lambda}: \mathcal{O}^{0}_\mu \longrightarrow\mathcal{O}^{0}_\lambda$ are isomorphisms of categories.\\
${\rm (2)}$ When $\kappa \neq 0$, the $(\mu,\lambda)$-transitive functor $\Upsilon_{\mu,\lambda}: \mathcal{O}^{\kappa}_\mu \longrightarrow\mathcal{O}^{\kappa}_{\lambda}$ is an isomorphism of categories.
\end{thm}

\begin{proof}
It is easy to check that ${\Sigma_i}$, ${\Sigma_{-i}}$, $\Theta_{\mu,\lambda}$, $\Theta_{\lambda,\mu}$, $\Upsilon_{\mu,\lambda}$ and $\Upsilon_{\lambda,\mu}$ are well-defined functors, and
 \begin{eqnarray*}\label{}
 \left\{
 \begin{array}{ll}
{\Sigma_i} {\Sigma_{-i}} = {\Sigma_{-i}} {\Sigma_i} = {\rm Id}_{\mathcal{O}^{0}_1},\\
\Theta_{\mu,\lambda} \Theta_{\lambda,\mu} = {\rm Id}_{\mathcal{O}^0_\lambda}, \quad \Theta_{\lambda,\mu} \Theta_{\mu,\lambda} = {\rm Id}_{\mathcal{O}^0_\mu},\\
\Upsilon_{\mu,\lambda} \Upsilon_{\lambda,\mu} = {\rm Id}_{\mathcal{O}^{\kappa}_\lambda}, \quad \Upsilon_{\lambda,\mu} \Upsilon_{\mu,\lambda} = {\rm Id}_{\mathcal{O}^{\kappa}_\mu}.
\end{array}
\right.
\end{eqnarray*}
Hence ${\Sigma_i}$, $\Theta_{\mu,\lambda}$ and $\Upsilon_{\mu,\lambda}$ are all isomorphisms of categories.
\end{proof}

\begin{cor}\label{reduction-to-principle-blocks}
${\rm (1)}$ When $\kappa = 0$, each block $\mathcal{O}^0_\lambda$ is isomorphic to $\mathcal{O}^0_1$.\\
${\rm (2)}$ When $\kappa \neq 0$, the block $\mathcal{O}^{\kappa}_{\omega}$ is isomorphic to $\mathcal{O}^{\kappa}_1$, while $\mathcal{O}^{\kappa}_{q\omega}$ is isomorphic to $\mathcal{O}^{\kappa}_q$, where $\omega^{2m} = 1$.
\end{cor}

Corollary \ref{reduction-to-principle-blocks} shows that to investigate the category $U_{q}(\mathfrak{sl}^{*}_2,\kappa)\mbox{-}{\rm \bf mod}_{\rm wt}$,
we can pay attention to its principle block(s) defined below.

\begin{defn}\label{definitions-of-principle-blocks}
${\rm (1)}$ When $\kappa = 0$, we call the block $\mathcal{O}^0_1$ to be the principle block of $U_{q}(\mathfrak{sl}^{*}_2)\mbox{-}{\rm \bf mod}_{\rm wt}$.\\
${\rm (2)}$ When $\kappa \neq 0$, we call the block $\mathcal{O}^{\kappa}_1$ (resp. $\mathcal{O}^{\kappa}_q$) to be the even (resp. odd) principle block of $U_{q}(\mathfrak{sl}^{*}_2,\kappa)\mbox{-}{\rm \bf mod}_{\rm wt}$.
\end{defn}

\section{Primitive objects in the
category $U_{q}(\mathfrak{sl}^{*}_2,\kappa)\mbox{-}{\rm \bf mod}_{\rm wt}$}\label{section-4}

In this section, we will introduce and study primitive objects in the category $U_{q}(\mathfrak{sl}^{*}_2,\kappa)\mbox{-}{\rm \bf mod}_{\rm wt}$
which can be used to construct all indecomposable objects in $U_{q}(\mathfrak{sl}^{*}_2,\kappa)\mbox{-}{\rm \bf mod}_{\rm wt}$.
As an application, we recover the representation theory of Drinfeld-Jimbo quantum group $U_{q}(\mathfrak{sl}_2)$.

\subsection{Quasi-primitive objects in $U_{q}(\mathfrak{sl}^{*}_2,\kappa)\mbox{-}{\rm \bf mod}_{\rm wt}$}\label{section-4-1}

\begin{defn}
If $M$ is a finite dimensional $q^2$-chain module of $U_{q}(\mathfrak{sl}^{*}_2,\kappa)$
with weight set $\Lambda_M = \{q^{s+2i} \mid 0 \leq i \leq l\}$
for some $s,l \in \mathbb{Z}^{\geq 0}$, then we call $M$ a quasi-primitive module of $U_{q}(\mathfrak{sl}^{*}_2,\kappa)$.
\end{defn}

Assume that $M$ is a $(n+1)$-dimensional quasi-primitive module of $U_{q}(\mathfrak{sl}^{*}_2,\kappa)$ with
 weight set $\Lambda_M = \{q^{s+2i} \mid 0 \leq i \leq l\}$,
then
$M = {\mathop\oplus\limits_{i=0}^{l}} M_{q^{s+2i}}.$
For each $0 \leq i \leq l$, suppose that ${\rm dim} M_{q^{s+2i}} = n_i$ and $\left\{v_{i1}, v_{i2}, \cdots, v_{i n_i} \right\}$ is a basis of $M_{q^{s+2i}}$,
then $n +1 = {\mathop\sum\limits_{i=0}^{l}} n_i$ and
\begin{eqnarray*}\label{}
B_M = \left\{v_{01}, v_{02}, \cdots, v_{0 n_0}, \cdots, v_{i1}, v_{i2}, \cdots, v_{i n_i}, \cdots, v_{l1}, v_{l2}, \cdots, v_{l n_{l}}\right\}
\end{eqnarray*}
is an ordered basis of $M$.
Since $K E = q^2 E K$ and $K F = q^{-2} F K$, then $E M_{q^{s+2i}} \subseteq M_{q^{s+2(i+1)}}$ and $F M_{q^{s+2i}} \subseteq M_{q^{s+2(i-1)}}$.
Therefore, the matrices of $K,E,F$ acting on $M$
relative to the ordered basis $B_M$ are respectively given by
\begin{eqnarray*}\label{}
\mathcal{K} &=& \left( {\begin{array}{*{20}c}
   q^s I_{n_0} &  &  & \\
    & q^{s+2} I_{n_1} &  & \\
    &  & \ddots & \\
     &  &  & q^{s+2l} I_{n_{l}}\\
\end{array}} \right),\\
\mathcal{E} &=& \left( {\begin{array}{*{20}c}
   0 &  &  & & \\
   \mathcal{E}_{0} & 0 &  & & \\
    & \mathcal{E}_{1} & \ddots & \\
    & &  \ddots  &  0 &   \\
    & &  & \mathcal{E}_{l-1} & 0\\
\end{array}} \right),\\
\mathcal{F} &=& \left( {\begin{array}{*{20}c}
   0 & \mathcal{F}_0 &  & &\\
    & 0 & \mathcal{F}_1 & &\\
    &  & \ddots &\ddots &\\
    &  &  & 0& \mathcal{F}_{l-1}\\
    & &  &  & 0\\
\end{array}} \right),
\end{eqnarray*}
where $I_{n_i}$ is the $n_i \times n_i$ identity matrix, $\mathcal{E}_{i}$ is a $n_{i+1} \times n_i$ matrix
and $\mathcal{F}_{i}$ is a $n_i \times n_{i+1}$ matrix. Unless otherwise specified, we always assume that $\mathcal{E}_{i} = 0$ and $\mathcal{F}_{i} = 0$ when $i \leq -1$ or $i \geq l$.

Now we can present the quasi-primitive objects in the category $U_{q}(\mathfrak{sl}^{*}_2,\kappa)\mbox{-}{\rm \bf mod}_{\rm wt}$ as follows.

\begin{prop}\label{q-2-rep-explicit-expression}
Retain the notations as above. If $M$ is a $(n+1)$-dimensional quasi-primitive module of $U_{q}(\mathfrak{sl}^{*}_2,\kappa)$ with $\Lambda_M = \{q^{s+2i} \mid 0 \leq i \leq l\}$,
then the generators $K,E,F$ of $U_{q}(\mathfrak{sl}^{*}_2,\kappa)$ acting on the basis $B_M$ of $M$ can be given as follows
\begin{eqnarray}\label{q2-chain-rep-dim-n}
\left\{
 \begin{array}{ll}
K B_M = B_M \mathcal{K},\\
E B_M = B_M \mathcal{E},\\
F B_M = B_M \mathcal{F},
\end{array}
\right.
\end{eqnarray}
where
\begin{eqnarray}\label{q2-chain-rep-n-conditions}
\left\{
 \begin{array}{ll}
  \mathcal{E} = \mathcal{F} = 0, \quad s = 0,&~{\rm if}~l = 0, \\
\mathcal{E}_{i-1} \mathcal{F}_{i-1} - \mathcal{F}_{i} \mathcal{E}_{i} = (1 - \delta_{\kappa, 0}) [s+2i]_m I_{n_{i}}~(0\leq i \leq l),&~{\rm if}~l \geq 1.
\end{array}
\right.
\end{eqnarray}
\end{prop}

\begin{proof}
It follows from the fact that $\mathcal{E}\mathcal{F} - \mathcal{F}\mathcal{E} = (1 - \delta_{\kappa, 0}) \frac{\mathcal{K}^m - \mathcal{K}^{-m}}{q^m - q^{-m}}$ is equivalent to $(\ref{q2-chain-rep-n-conditions})$.
\end{proof}

The endomorphism algebra of a quasi-primitive $U_{q}(\mathfrak{sl}^{*}_2,\kappa)$-module can be characterized below.

\begin{prop}\label{endomorphism-alg-of-quasi-primitive-rep-of-Hopf-PBW-deformations}
Let $M$ be the $(n+1)$-dimensional quasi-primitive module of $U_{q}(\mathfrak{sl}^{*}_2,\kappa)$ defined by $(\ref{q2-chain-rep-dim-n})$ and $(\ref{q2-chain-rep-n-conditions})$.
 Then the endomorphism algebra ${\rm End}_{U_{q}(\mathfrak{sl}^{*}_2,\kappa)} (M)$ of $M$ is isomorphic to
 \begin{eqnarray}\label{end-alg-of-q-2-rep-with-dim-n-in-general}
 \left\{\left( {\begin{array}{*{20}c}
   A_0 &  &  & \\
    & A_1 &  & \\
    &  & \ddots & \\
     &  &  & A_{l}\\
\end{array}} \right) \Bigg| {\mathcal{E}_i A_i = A_{i+1} \mathcal{E}_i(0\leq i\leq l-1),
\atop A_{i} \mathcal{F}_{i} = \mathcal{F}_{i} A_{i+1} (0\leq i\leq l-1)} \right\},
 \end{eqnarray}
 where $A_i$ is a $n_i \times n_i$ matrix for $0\leq i \leq l$.
\end{prop}

\begin{proof}
Assume $\phi \in {\rm End}_{U_{q}(\mathfrak{sl}^{*}_2,\kappa)} (M)$. Since $\phi(K v) = K \phi(v)$ for any $v \in M$, then $\phi(M_\lambda) \subseteq M_\lambda$ for any $\lambda \in \Lambda_M$, which implies that
\begin{eqnarray*}\label{}
\phi B_M = B_M \left( {\begin{array}{*{20}c}
   A_0 &  &  & \\
    & A_1 &  & \\
    &  & \ddots & \\
     &  &  & A_{l}\\
\end{array}} \right).
 \end{eqnarray*}
Moreover, the fact that $\phi(E v) = E \phi(v)$ and $\phi(F v) = F \phi(v)$ for $v \in M$ is equivalent to say that
\begin{eqnarray*}\label{}
&& \mathcal{E}_i A_i = A_{i+1} \mathcal{E}_i(0\leq i\leq l-1),\\
&& A_{i} \mathcal{F}_{i} = \mathcal{F}_{i} A_{i+1} (0\leq i\leq l-1).
 \end{eqnarray*}
 Hence this proposition holds.
\end{proof}

\subsection{Primitive modules of $U_{q}(\mathfrak{sl}^{*}_2)$}\label{section-4-2}

\begin{defn}\label{primitive-rep-of-new-type-quantum-group}
Let $M$ be a finite dimensional quasi-primitive module of $U_{q}(\mathfrak{sl}^{*}_2)$.
If $M$ is indecomposable and its weight set
$\Lambda_M = \{1, q^{2}, \cdots,q^{2i}, \cdots, q^{2l}\}$
for some integer $l\geq 0$, then we call $M$ a primitive module of $U_{q}(\mathfrak{sl}^{*}_2)$.
\end{defn}

The fundamental importance of primitive modules arises from the fact: each indecomposable module in the category $U_{q}(\mathfrak{sl}^{*}_2)\mbox{-}{\rm \bf mod}_{\rm wt}$
 can be obtained by applying the translation functor ${\Sigma}: \mathcal{O}^0_1 \longrightarrow \mathcal{O}^0_1$ and the transitive functor $\Theta_{\lambda}: \mathcal{O}^0_1 \longrightarrow\mathcal{O}^0_\lambda$ on some primitive one. Precisely speaking, we have the following proposition.

\begin{prop}\label{primitive-is-proper-new-type-quantum-group-case}
For any indecomposable module $N$ in the category $U_{q}(\mathfrak{sl}^{*}_2)\mbox{-}{\rm \bf mod}_{\rm wt}$, there exists a unique triple $(M,\lambda,k)$ such that $N \cong \Theta_{\lambda} \Sigma^k M$, where $M$ is a primitive module of $U_{q}(\mathfrak{sl}^{*}_2)$, $\lambda \in [\Lambda_{\kappa} / \overset{q^2}\sim]$ and $k\in \mathbb{Z}$.
\end{prop}

\begin{proof}
By Proposition $\ref{indecomposable-representation-in-unique block}$, there exists a unique $\lambda \in [\Lambda_{\kappa} / \overset{q^2}\sim]$ such that
$N$ lies in the block $\mathcal{O}^0_\lambda$. By Proposition $\ref{fd-indecom-weight-reps-are-q2-chain-reps}$ and Theorem $\ref{equivalence-functors-translation-transitive-functors}~(1)$,
we know that $\Theta^{-1}_{\lambda}N$ is an indecomposable $q^2$-chain module in the principal block $\mathcal{O}^0_1$.
By Lemma $\ref{q-chain-characterization-by-unique}$, there exists a unique pair $(k,l)$ of integers with $k \leq l$ such that
$\Lambda_{\Theta^{-1}_{\lambda}N} = \left\{q^{2i} | k \leq i \leq l\right\}.$
By Theorem $\ref{equivalence-functors-translation-transitive-functors}~(1)$, $\Sigma^{-k}\Theta^{-1}_{\lambda}N$ is a primitive module.
Therefore, we can choose $M = \Sigma^{-k}\Theta^{-1}_{\lambda}N$ such that $N \cong \Theta_{\lambda} \Sigma^k M$.
The uniqueness of $(M,\lambda,k)$ follows from the definitions of $\Sigma$ and $\Theta_{\lambda}$.
\end{proof}

Applying Proposition \ref{q-2-rep-explicit-expression} and \ref{endomorphism-alg-of-quasi-primitive-rep-of-Hopf-PBW-deformations} to $U_{q}(\mathfrak{sl}^{*}_2)$, we can obtain the following two corollaries.

\begin{cor}\label{cor-of-q-2-rep-explicit-expression}
Let $n \in \mathbb{Z}^{\geq 0}$. If $M$ is a $(n+1)$-dimensional quasi-primitive module of $U_{q}(\mathfrak{sl}^{*}_2)$ with basis $v_0, v_1, \cdots,v_{n}$
 and weight set $\Lambda_{n} = \{q^{2i} | 0\leq i \leq n\}$, then one has
 \begin{eqnarray}\label{necessary-conditions-in-cor-of-q-2-rep-explicit-expression}
\left\{
 \begin{array}{ll}
K v_i = q^{2i} v_i,\\
E v_i = \left\{
 \begin{array}{ll}
\mathcal{E}_i v_{i+1}, & {\rm if}~i < n,\\
0, & {\rm if}~i = n,
\end{array}
\right.\\
F v_i = \left\{
 \begin{array}{ll}
\mathcal{F}_{i-1} v_{i-1}, & {\rm if}~i > 0,\\
0, & {\rm if}~i = 0,
\end{array}
\right.
\end{array}
\right.
\end{eqnarray}
where $\mathcal{E}_i, \mathcal{F}_i \in \mathbb{C}$ satisfy
\begin{eqnarray}\label{condition-indecom-primitive-1}
\mathcal{E}_0 \mathcal{F}_0 = \mathcal{E}_1 \mathcal{F}_1 = \cdots = \mathcal{E}_{n-1} \mathcal{F}_{n-1} = 0.
\end{eqnarray}
\end{cor}

\begin{cor}\label{end-alg-indec-rep-cor-of-q-2-rep-explicit-expression}
Let $L_{\mathcal{E},\mathcal{F}}$ be the $(n+1)$-dimensional quasi-primitive module of $U_{q}(\mathfrak{sl}^{*}_2)$ defined by $(\ref{necessary-conditions-in-cor-of-q-2-rep-explicit-expression})$ and
$(\ref{condition-indecom-primitive-1})$.
 Then the following statements hold.\\
${\rm (1)}$ ${\rm End}_{U_{q}(\mathfrak{sl}^{*}_2)} (L_{\mathcal{E},\mathcal{F}})  \cong \underset{r}{\underbrace{\mathbb{C} \times \mathbb{C} \times \cdots \times \mathbb{C}}},$
where $r = \left\{i | \mathcal{E}_i = \mathcal{F}_i = 0, 0\leq i \leq n-1\right\}^\sharp + 1.$\\
${\rm (2)}$ $L_{\mathcal{E},\mathcal{F}}$ is indecomposable if and only if $\left\{i | \mathcal{E}_i = \mathcal{F}_i = 0, 0\leq i \leq n-1\right\}^\sharp=0.$
\end{cor}

\begin{proof}
${\rm (1)}$ By Proposition \ref{q-2-rep-explicit-expression}, we only need to show that there exists an isomorphism
 \begin{eqnarray}\label{end-alg-of-q-2-rep-with-dim-n}
 \qquad\left\{\left( {\begin{array}{*{20}c}
   a_0 &  &  & \\
    & a_1 &  & \\
    &  & \ddots & \\
     &  &  & a_{n}\\
\end{array}} \right) \Bigg| {a_i \mathcal{E}_i = a_{i+1} \mathcal{E}_i(0\leq i\leq n-1),
\atop a_{i} \mathcal{F}_{i} = a_{i+1} \mathcal{F}_{i}(0\leq i\leq n-1)} \right\}
 \cong \underset{r}{\underbrace{\mathbb{C} \times \mathbb{C} \times \cdots \times \mathbb{C}}}.
 \end{eqnarray}
Consider
\begin{eqnarray}\label{linear-equations-primitive-reps-proof}
 \left\{
 \begin{array}{ll}
a_i \mathcal{E}_i = a_{i+1} \mathcal{E}_i~(0\leq i\leq n-1),\\
a_{i} \mathcal{F}_{i} = a_{i+1} \mathcal{F}_{i}~(0\leq i\leq n-1)
\end{array}
\right.
 \end{eqnarray}
as the set of linear equations about $a_0,a_1,\cdots , a_{n}$,
then the condition in $(\ref{condition-indecom-primitive-1})$ implies that the number of basic solutions of $(\ref{linear-equations-primitive-reps-proof})$ is equal to $r$.
 Hence the isomorphism in $(\ref{end-alg-of-q-2-rep-with-dim-n})$ exists.\\
${\rm (2)}$ $L_{\mathcal{E},\mathcal{F}}$ is indecomposable if and only if ${\rm End}_{U_{q}(\mathfrak{sl}^{*}_2)} (L_{\mathcal{E},\mathcal{F}})$ is local, which is equivalent to
say $r=1$ in $(\ref{end-alg-of-q-2-rep-with-dim-n})$, i.e., $\left\{i | \mathcal{E}_i = \mathcal{F}_i = 0, 0\leq i \leq n-1\right\}^\sharp=0.$
\end{proof}

By using Corollary \ref{cor-of-q-2-rep-explicit-expression} and \ref{end-alg-indec-rep-cor-of-q-2-rep-explicit-expression}, we can explicitly present a class of primitive modules of $U_{q}(\mathfrak{sl}^{*}_2)$ up to isomorphisms.

\begin{thm}\label{primitive-rep-of-U-with-type-1-q}
Assume that $M$ is a $(n+1)$-dimensional primitive module of $U_{q}(\mathfrak{sl}^{*}_2)$ with weight set $\Lambda_{M} = \{q^{2i} | 0\leq i \leq n\}$.
Then $M$ is isomorphic to a $(n+1)$-dimensional module $L_{\mathcal{E}^{\mathbf p},\mathcal{F}^{\mathbf p}}$ with a basis $w_0, w_1, \cdots, w_{n}$ and the actions of $K,E,F$ defined by
\begin{eqnarray}\label{necessary-conditions-in-cor-of-primitive-q-2-rep-explicit-expression}
\left\{
 \begin{array}{ll}
K w_i = q^{2i} w_i,\\
E w_i = \left\{
 \begin{array}{ll}
\mathcal{E}^{\mathbf p}_i w_{i+1}, & {\rm if}~i < n,\\
0, & {\rm if}~i = n,
\end{array}
\right.\\
F w_i = \left\{
 \begin{array}{ll}
\mathcal{F}^{\mathbf p}_{i-1} w_{i-1}, & {\rm if}~i > 0,\\
0, & {\rm if}~i = 0,
\end{array}
\right.
\end{array}
\right.
\end{eqnarray}
where for any $n \geq 1$ and $0\leq i\leq n - 1$ the coefficients $\mathcal{E}^{\mathbf p}_i, \mathcal{F}^{\mathbf p}_i \in \mathbb{C}$ satisfy
\begin{eqnarray}\label{a-class-of-primitive-reps-of-U}
\mathcal{E}^{\mathbf p}_i \mathcal{F}^{\mathbf p}_{i} = 0 \quad {\rm and} \quad
\mathcal{E}^{\mathbf p}_i + \mathcal{F}^{\mathbf p}_{i} = 1.
 \end{eqnarray}
\end{thm}

\begin{proof}
By Corollary \ref{cor-of-q-2-rep-explicit-expression} and \ref{end-alg-indec-rep-cor-of-q-2-rep-explicit-expression} ${\rm (2)}$, $M$ can be presented by $(\ref{necessary-conditions-in-cor-of-q-2-rep-explicit-expression})$
 with $\mathcal{E}_i \mathcal{F}_{i} = 0$ and $\mathcal{E}_i + \mathcal{F}_{i} \neq 0.$
For any $c \in \mathbb{C}$, define
$\delta_c = \left\{
\begin{array}{ll}
1,& {\rm if}~ c\neq 0,\\
0,&{\rm if}~ c = 0,
\end{array}
\right.$
then set
$\mathcal{E}^{\mathbf p}_i = \delta_{\mathcal{E}_i}$ and $\mathcal{F}^{\mathbf p}_i = \delta_{\mathcal{F}_i}$.
Since the following set of linear equations
\begin{eqnarray*}\label{}
 \left\{
 \begin{array}{ll}
\mathcal{E}_i \lambda_{i+1} = \mathcal{E}^{\mathbf p}_i \lambda_i ~(0\leq i\leq n-1),\\
\mathcal{F}_{i} \lambda_{i} = \mathcal{F}^{\mathbf p}_{i} \lambda_{i+1}~(0\leq i\leq n-1)
\end{array}
\right.
\end{eqnarray*}
has nonzero solutions $(\lambda_0, \lambda_1, \cdots, \lambda_n)$ with all $\lambda_i \neq 0$, then it is easy to check that the map $M \xrightarrow{\varphi} L_{\mathcal{E}^{\mathbf p},\mathcal{F}^{\mathbf p}}$ defined by $\varphi(v_i) = \lambda_i w_i$
is an isomorphism of $U_{q}(\mathfrak{sl}^{*}_2)$-modules.
\end{proof}

\subsection{Primitive modules of $U_{q}(\mathfrak{sl}^{*}_2,\kappa)$ with $\kappa \neq 0$}\label{section-4-3}

\begin{defn}\label{primitive-rep-of-new-type-Hopf-PBW-quantum-group}
Let $M$ be a finite dimensional quasi-primitive module of $U_{q}(\mathfrak{sl}^{*}_2,\kappa)$.
If $M$ is simple,
then we call $M$ a primitive module of $U_{q}(\mathfrak{sl}^{*}_2,\kappa)$.
\end{defn}

To describe the primitive modules of $U_{q}(\mathfrak{sl}^{*}_2,\kappa)$, we need the following fundamental proposition.

\begin{prop}\label{quasi-primitive-rep-before-classification-of-primitive-representations-of-Hopf-PBW-deformations}
Let $M$ be a $(n+1)$-dimensional quasi-primitive module of $U_{q}(\mathfrak{sl}^{*}_2,\kappa)$. Assume that the dimensions of all the weight spaces of $M$ are equal.
If $M$ is indecomposable, then the dimensions of all the weight spaces of $M$ are equal to $1$, and the weight set $\Lambda_M$ of $M$ is given by
\begin{eqnarray}\label{a-class-of-primitive-reps-of-U-eq}
\Lambda_M = \{q^{-n}, q^{-n+2}, \cdots, q^{n-2}, q^{n}\}.
\end{eqnarray}
\end{prop}

\begin{proof}
In Proposition $\ref{q-2-rep-explicit-expression}$, any $(n+1)$-dimensional quasi-primitive module $M$ of $U_{q}(\mathfrak{sl}^{*}_2,\kappa)$
is characterized by $(\ref{q2-chain-rep-dim-n})$ and $(\ref{q2-chain-rep-n-conditions})$. In this proof, we will retain the notations in Proposition $\ref{q-2-rep-explicit-expression}$.
Since the dimensions of all the weight spaces of $M$ are equal, then ${\rm dim} M_{q^{s+2i}} = n_0$ for all $0 \leq i \leq l.$
When $l = 0$, noting that $M$ is indecomposable, we can see from Proposition $\ref{q-2-rep-explicit-expression}$ that $n_0 = n = 1$ and $\Lambda_M = \{1\}.$

From now on, we assume that $l \geq 1$. For any $0 \leq i \leq l-1$, by respectively adding the top $i+1$ formulas and the bottom $l-i$ ones in $(\ref{q2-chain-rep-n-conditions})$, one has
\begin{eqnarray}\label{the-matrix-corresponding-to-E-F-are-invert-eq}
\mathcal{F}_i \mathcal{E}_i = a_i I_{n_0} \quad {\rm and} \quad
\mathcal{E}_i \mathcal{F}_i = b_i I_{n_0},
\end{eqnarray}
where \begin{eqnarray}\label{the-matrix-corresponding-to-E-F-are-invert-eq-coefficients}
\left\{
 \begin{array}{ll}
a_i = {\mathop\sum\limits_{k=0}^i}[-s-2k]_m = [-s-i]_m [i+1]_m,\\
b_i = {\mathop\sum\limits_{k=i+1}^l}[s+2k]_m = [l-i]_m [s+l+i+1]_m.
\end{array}
\right.
 \end{eqnarray}
We claim that $a_i = b_i \neq 0$ for any $0 \leq i \leq l-1$.
Obviously, it can be deduced from $(\ref{the-matrix-corresponding-to-E-F-are-invert-eq})$ that $a_i \neq 0$ if and only if $b_i \neq 0$, and in this case $a_i = b_i \neq 0$.
To prove our claim, we only need to exclude the case $a_i = b_i = 0$. In fact, $a_i = b_i = 0$ if and only if $s + i =0$ and $i = l$, which is in contradiction with $0 \leq i \leq l-1$.
Hence for any $0 \leq i \leq l-1$ one has
\begin{eqnarray}\label{quasi-primitive-rep-before-classification-of-primitive-representations-of-Hopf-PBW-deformations-eq1}
\mathcal{E}_i \mathcal{F}_i = \mathcal{F}_i \mathcal{E}_i = a_i I_{n_0}.
\end{eqnarray}
Combining $(\ref{quasi-primitive-rep-before-classification-of-primitive-representations-of-Hopf-PBW-deformations-eq1})$ with Proposition $\ref{endomorphism-alg-of-quasi-primitive-rep-of-Hopf-PBW-deformations}$, one has
\begin{eqnarray}
{\rm End}_{U_{q}(\mathfrak{sl}^{*}_2,\kappa)} (M) \cong {\rm \bf Mat}_{n_0}(\mathbb{C}),
\end{eqnarray}
where ${\rm \bf Mat}_{n_0}(\mathbb{C})$ is the matrix algebra consisting of all $n_0 \times n_0$ complex matrices.
Since $M$ is indecomposable, then ${\rm End}_{U_{q}(\mathfrak{sl}^{*}_2,\kappa)} (M)$ is local, which implies
$n_0 = 1$.

Next we show that $\Lambda_M = \{q^{-n}, q^{-n+2}, \cdots, q^{n-2}, q^{n}\}.$ Since $M$ is $(n+1)$-dimensional, then $n+1 = {\mathop\sum\limits_{i=0}^l} n_i = l+1$. So we can express the weight set of $M$ as follows
\begin{eqnarray*}\label{}
\Lambda_M = \{q^{s}, q^{s+2}, \cdots, q^{s+2(n-1)}, q^{s+2n}\}.
\end{eqnarray*}
By $(\ref{the-matrix-corresponding-to-E-F-are-invert-eq-coefficients})$, one gets
\begin{eqnarray}\label{simple-rep-of-classical-quantum-group-claim-3}
a_i - b_i = {\mathop\sum\limits_{k=0}^{n}}[-s-2k]_m = [-s-n]_m [n+1]_m.
\end{eqnarray}
Because $a_i = b_i$ and $q\in \mathbb{C}^\times$ is not a root of unity, one must have $s=-n$. The proof is finished.
\end{proof}

For all $n\in \mathbb{Z}$, set
\begin{eqnarray*}\label{}
[K;n]_m = \frac{q^{nm} K^{m} -q^{-n m} K^{-m}}{q^{m}-q^{-m}}.
\end{eqnarray*}
It is easy to check that in $U_{q}(\mathfrak{sl}^{*}_2,\kappa)$ one has
\begin{eqnarray}\label{}
\label{pbw-comm-formula-5} F E^{r} - E^{r} F  = -[r]_m E^{r-1} [K;r - 1]_m.
\end{eqnarray}
The center of $U_{q}(\mathfrak{sl}^{*}_2,\kappa)$ is a polynomial algebra generated by the Casimir element
\begin{eqnarray*}\label{}
C_q = EF + \frac{q^{-m}K^m + q^m K^{-m}}{(q^m - q^{-m})^2} = F E + \frac{q^{m}K^m + q^{-m} K^{-m}}{(q^m - q^{-m})^2},
\end{eqnarray*}
which is proved in Theorem 4.15 in \cite{JW}.

With the help of Proposition \ref{quasi-primitive-rep-before-classification-of-primitive-representations-of-Hopf-PBW-deformations}, we can clearly describe the primitive modules of $U_{q}(\mathfrak{sl}^{*}_2,\kappa)$.

\begin{thm}\label{classification-of-primitive-representations-of-Hopf-PBW-deformations}
Let $M$ be a $(n+1)$-dimensional primitive module of $U_{q}(\mathfrak{sl}^{*}_2,\kappa)$. Then the following statements hold.\\
${\rm (1)}$ The dimensions of all the weight spaces of $M$ are equal to $1$.\\
${\rm (2)}$ $M$ is isomorphic to the $(n+1)$-dimensional primitive module $L_{\rm \bf p}(n,+)$ of $U_{q}(\mathfrak{sl}^{*}_2,\kappa)$ with basis $w_0,w_1,\cdots,w_n$ and the actions of $K,E,F$ on $M$ given below
 \begin{eqnarray}\label{necessary-conditions-in-simple-primitive-rep-of-Hopf-PBW-deformations}
\left\{
 \begin{array}{ll}
K w_i = q^{2i-n} w_i,\\
E w_i = \left\{
 \begin{array}{ll}
[n-i]_m [i+1]_m w_{i+1}, & {\rm if}~i < n,\\
0, & {\rm if}~i = n,
\end{array}
\right.\\
F w_i = \left\{
 \begin{array}{ll}
w_{i-1}, & {\rm if}~i > 0,\\
0, & {\rm if}~i = 0.
\end{array}
\right.
\end{array}
\right.
\end{eqnarray}
${\rm (3)}$ $C_q$ acts on $M$ by the same scalar $c_q(n)$ as on $L_{\rm \bf p}(n,+)$, where
\begin{eqnarray*}\label{}
c_q(n) = \frac{q^{(n+1)m} + q^{-(n+1)m}}{(q^m - q^{-m})^2}.
\end{eqnarray*}
\end{thm}

\begin{proof}
${\rm (1)}$ In this proof, we also retain the notations in Proposition $\ref{q-2-rep-explicit-expression}$, and respectively identify the matrices $\mathcal{E}_i$ and $\mathcal{F}_i$ with the linear maps
 $M_{q^{s+2i}} \xrightarrow{E} M_{q^{s+2(i+1)}}$ and $M_{q^{s+2(i+1)}} \xrightarrow{F} M_{q^{s+2i}}$.
By Proposition $\ref{quasi-primitive-rep-before-classification-of-primitive-representations-of-Hopf-PBW-deformations}$,
to prove ${\rm (1)}$, we only need to prove that the dimensions of all the weight spaces of $M$ are equal, i.e.,
\begin{eqnarray}\label{simple-rep-of-classical-quantum-group-claim-1}
{\rm dim} M_{q^{s+2i}} = n_i = n_{i+1} = {\rm dim} M_{q^{s+2(i+1)}}~(0 \leq i \leq l-1).
\end{eqnarray}
When $l=0$, the claim $(\ref{simple-rep-of-classical-quantum-group-claim-1})$ is trivial. Now assume $l\geq 1$, we will prove $(\ref{simple-rep-of-classical-quantum-group-claim-1})$ by induction on $i$.

 When $i = 0,$ by $(\ref{q2-chain-rep-n-conditions})$ one has
 \begin{eqnarray}\label{classification-of-primitive-representations-of-Hopf-PBW-deformations-eq-1-1}
\mathcal{F}_0 \mathcal{E}_0 = [-s]_m I_{n_0}.
\end{eqnarray}
It follows from $F M_{q^{s}} = 0$ and $(\ref{pbw-comm-formula-5})$ that
 \begin{eqnarray}\label{classification-of-primitive-representations-of-Hopf-PBW-deformations-eq-1-0}
F E^{j+1} M_{q^{s}} = \left(E^{j+1} F - [j+1]_m E^j [K;j]_m\right)M_{q^{s}} \subseteq E^j M_{q^{s}}.
\end{eqnarray}
 We claim that $\mathcal{F}_0 \mathcal{E}_0 \neq 0$. Otherwise,
if $\mathcal{E}_0 = 0$, then $M_{q^{s}}$ is a nontrivial submodule of $M$;
  when $\mathcal{E}_0 \neq 0$, it can be seen from $\mathcal{F}_0 \mathcal{E}_0 = 0$ and $(\ref{classification-of-primitive-representations-of-Hopf-PBW-deformations-eq-1-0})$ that
  ${\mathop\bigoplus\limits_{j=0}^{l-1}} E^{j+1} M_{q^{s}}$ is a nontrivial submodule of $M$.
 Now $(\ref{classification-of-primitive-representations-of-Hopf-PBW-deformations-eq-1-1})$ implies that $n_0 \leq n_1$. If $n_0 < n_1$, then ${\rm Im}\mathcal{E}_0 = E M_{q^{s}}$ is a nonzero proper subspace of $M_{q^{s+2}}$.
  It is easy to check that $N = M_{q^{s}} \oplus {\mathop\bigoplus\limits_{j=0}^{l-1}} E^{j+1} M_{q^{s}}$ is a nontrivial submodule of $M$ by
  $(\ref{classification-of-primitive-representations-of-Hopf-PBW-deformations-eq-1-0})$,
  which is in contradiction with the simplicity of $M$. Hence $n_0 = n_1$.

Now assume that $n_0 = n_1 = \cdots = n_{i}$, we will verify that $n_i = n_{i+1}$.
In this case, we can also derive $\mathcal{F}_i \mathcal{E}_i = a_i I_{n_i}$ similarly as $(\ref{the-matrix-corresponding-to-E-F-are-invert-eq})$ and prove that
\begin{eqnarray*}\label{}
N = \left\{
 \begin{array}{ll}
{\mathop\bigoplus\limits_{j=0}^{i}} M_{q^{s+2j}}, & {\rm when}~\mathcal{E}_i = 0,\\
{\mathop\bigoplus\limits_{j=0}^{l-i-1}} E^{j+1} M_{q^{s+2i}}, & {\rm when}~\mathcal{F}_i \mathcal{E}_i = 0 {\rm ~and~} \mathcal{E}_i \neq 0,\\
{\mathop\bigoplus\limits_{j=0}^{i}} M_{q^{s+2j}} \oplus {\mathop\bigoplus\limits_{j=0}^{l-i-1}} E^{j+1} M_{q^{s+2i}}, & {\rm when}~\mathcal{F}_i \mathcal{E}_i \neq 0 {\rm ~and~} n_i < n_{i+1}
\end{array}
\right.
\end{eqnarray*}
 is a nontrivial submodule of $M$. Then we can show $\mathcal{F}_i \mathcal{E}_i \neq 0$ and $n_i = n_{i+1}$ in a similar way as the case $i=0$.

${\rm (2)}$ Firstly, it is easy to check that $L_{\rm \bf p}(n,+)$ is a module of $U_{q}(\mathfrak{sl}^{*}_2,\kappa)$.

Next we will prove that $L_{\rm \bf p}(n,+)$ is primitive.
Otherwise, the length $t$ of $L_{\rm \bf p}(n,+)$ is at least 2, i.e., there exists a composition series of $L_{\rm \bf p}(n,+)$ as follows
\begin{eqnarray*}\label{}
0=L_0 \subset L_1 \subset L_2 \subset \cdots \subset L_t = L_{\rm \bf p}(n,+).
\end{eqnarray*}
Since $L_1$ is simple and $\Lambda_{L_1} \subseteq \Lambda_{L_{\rm \bf p}(n,+)}$, it follows from Proposition \ref{reps-are-all-weight-reps-for-classical-quantum-groups} and \ref{fd-indecom-weight-reps-are-q2-chain-reps} that $L_1$ is primitive.
Then by ${\rm (1)}$ and Proposition \ref{quasi-primitive-rep-before-classification-of-primitive-representations-of-Hopf-PBW-deformations} one obtains
$\Lambda_{L_1} = \{q^{-l_1}, q^{-l_1+2}, \cdots, q^{l_1-2}, q^{l_1}\},$
where $l_1 = {\rm dim}L_1 -1 < n$. Therefore, $L_1 = {\mathop\bigoplus\limits_{i=\frac{n-l_1}{2}}^{\frac{n+l_1}{2}}} \mathbb{C} w_i$.
However, the formulas in $(\ref{necessary-conditions-in-simple-primitive-rep-of-Hopf-PBW-deformations})$ show that $L_1$ is not a submodule of $L_{\rm \bf p}(n,+)$, which is a contradiction.

Finally, we will verify that $M$ is isomorphic to $L_{\rm \bf p}(n,+)$.
By ${\rm (1)}$, Proposition \ref{quasi-primitive-rep-before-classification-of-primitive-representations-of-Hopf-PBW-deformations} and Proposition \ref{q-2-rep-explicit-expression}, $M$ can be presented by
 \begin{eqnarray*}
\left\{
 \begin{array}{ll}
K v_i = q^{2i-n} v_i,\\
E v_i = \left\{
 \begin{array}{ll}
\mathcal{E}_i v_{i+1}, & {\rm if}~i < n,\\
0, & {\rm if}~i = n,
\end{array}
\right.\\
F v_i = \left\{
 \begin{array}{ll}
\mathcal{F}_{i-1}v_{i-1}, & {\rm if}~i > 0,\\
0, & {\rm if}~i = 0,
\end{array}
\right.
\end{array}
\right.
\end{eqnarray*}
where $v_0,v_1,\cdots,v_n$ is a basis of $M$ and $\mathcal{E}_i,\mathcal{F}_{i} \in \mathbb{C}~(0 \leq i \leq n-1)$ satisfy
$\mathcal{E}_i \mathcal{F}_{i} = [n-i]_m [i+1]_m.$
Since the following set of linear equations
\begin{eqnarray*}\label{}
 \left\{
 \begin{array}{ll}
\mathcal{E}_i \lambda_{i+1} = [n-i]_m [i+1]_m \lambda_i~(0\leq i\leq n-1),\\
\mathcal{F}_{i} \lambda_{i} = \lambda_{i+1}~(0\leq i\leq n-1)
\end{array}
\right.
\end{eqnarray*}
has a nonzero solution $(\lambda_0, \lambda_1, \cdots, \lambda_n)$ with all $\lambda_i \neq 0$, then it is easy to check that the map $M \xrightarrow{\phi} L_{\rm \bf p}(n,+)$ defined by $\phi(v_i) = \lambda_i w_i$
is an isomorphism of $U_{q}(\mathfrak{sl}^{*}_2,\kappa)$-modules.

${\rm (3)}$ Noting that $(\ref{the-matrix-corresponding-to-E-F-are-invert-eq})$ and $(\ref{quasi-primitive-rep-before-classification-of-primitive-representations-of-Hopf-PBW-deformations-eq1})$ both hold for $M$ and $L_{\rm \bf p}(n,+)$, and $\Lambda_M = \Lambda_{L_{\rm \bf p}(n,+)}$,
we can deduce that $C_q$ acts on $M$ by the same scalar $c_q(n)$ as on $L_{\rm \bf p}(n,+)$ by direct calculations.
\end{proof}

The word ``primitive" in Definition \ref{primitive-rep-of-new-type-Hopf-PBW-quantum-group} is also proper, because we will soon prove the following two facts:
${\rm (1)}$ Each simple object in $U_{q}(\mathfrak{sl}^{*}_2,\kappa)\mbox{-}{\rm \bf mod}$ can be obtained by applying the transitive functor $\Upsilon^{\lambda_0}_\omega: \mathcal{O}^\kappa_{\lambda_0} \longrightarrow \mathcal{O}^\kappa_{\lambda_0 \omega}$ on some primitive one;
${\rm (2)}$ The category $U_{q}(\mathfrak{sl}^{*}_2,\kappa)\mbox{-}{\rm \bf mod}$ is semisimple.

\begin{thm}\label{primitive-representations-of-Hopf-PBW-deformations-is-proper}
${\rm (1)}$ For any simple $U_{q}(\mathfrak{sl}^{*}_2,\kappa)$-module $L$ of dimension $n+1$, there exists a unique pair $(\lambda_0,\omega)$ such that
$L \cong \Upsilon^{\lambda_0}_\omega (L_{\rm \bf p}(n,+))$, where $\omega^{2m} = 1$ and
$\lambda_0 =
\left\{
 \begin{array}{ll}
1, & {\rm if}~n~{\rm is~even},\\
q, & {\rm if}~n~{\rm is~odd}.
\end{array}
\right.$\\
${\rm (2)}$ Each simple $U_{q}(\mathfrak{sl}^{*}_2,\kappa)$-module of dimension $n+1$ is isomorphic to a $U_{q}(\mathfrak{sl}^{*}_2,\kappa)$-module $L(n,\omega)$ with basis $v_0,v_1,\cdots,v_n$ and $\omega^{2m}=1$ such that for all $0 \leq i \leq n$
 \begin{eqnarray}\label{all-simple-reps-of-Hopf-PBW-deformations}
\left\{
 \begin{array}{ll}
K v_i = \omega q^{2i-n} v_i,\\
E v_i = \left\{
 \begin{array}{ll}
\omega^m [n-i]_m [i+1]_m v_{i+1}, & {\rm if}~i < n,\\
0, & {\rm if}~i = n,
\end{array}
\right.\\
F v_i = \left\{
 \begin{array}{ll}
v_{i-1}, & {\rm if}~i > 0,\\
0, & {\rm if}~i = 0.
\end{array}
\right.
\end{array}
\right.
\end{eqnarray}
\end{thm}

\begin{proof}
${\rm (1)}$ As Definition $\ref{defs-of-translation-and-transitive-functors}$ ${\rm (2)}$, we choose
$[\Lambda_{\kappa} / \overset{q^2}\sim] = \left\{\omega, q \omega | \omega^{2m} =1 \right\}.$
By Proposition $\ref{indecomposable-representation-in-unique block}$, there exists a unique $\lambda_0 \omega \in [\Lambda_{\kappa} / \overset{q^2}\sim]$ such that
$L$ lies in the block $\mathcal{O}_{\lambda_0 \omega}^{\kappa}$, where $\lambda_0 \in \{1,q\}$. By Theorem $\ref{equivalence-functors-translation-transitive-functors}~(2)$,
we know that $(\Upsilon^{\lambda_0}_{\omega})^{-1}L$ is a primitive module in the principal block $\mathcal{O}^{\kappa}_{\lambda_0}$.
It follows from Theorem $\ref{classification-of-primitive-representations-of-Hopf-PBW-deformations}$ ${\rm (2)}$ that $(\Upsilon^{\lambda_0}_{\omega})^{-1}L \cong L_{\rm \bf p}(n,+)$.
Hence $L \cong \Upsilon^{\lambda_0}_\omega (L_{\rm \bf p}(n,+))$. The uniqueness of $(\lambda_0,\omega)$ arises from the uniqueness of $\lambda_0 \omega$.

${\rm (2)}$ The definition of $\Upsilon^{\lambda_0}_\omega$ implies that $L(n,\omega) = \Upsilon^{\lambda_0}_\omega (L_{\rm \bf p}(n,+))$. Hence ${\rm (2)}$ follows from ${\rm (1)}$.
\end{proof}

\begin{cor}\label{cor-classification-of-primitive-representations-of-Hopf-PBW-deformations}
${\rm (1)}$ Each simple $U_{q}(\mathfrak{sl}^{*}_2,\kappa)$-module of dimension $n+1$ in the even (resp. odd) principle block $\mathcal{O}_1^{\kappa}$ (resp. $\mathcal{O}_q^{\kappa}$) is isomorphic to the module $L(n,1)$ with $n$ even (resp. odd).\\
${\rm (2)}$ For a given $\lambda_0 \in \{1,q\}$, let $L$ and $L^\prime$ be finite dimensional simple $U_{q}(\mathfrak{sl}^{*}_2,\kappa)$-modules in the principle block $\mathcal{O}_{\lambda_0}^{\kappa}$.
 If $C_q$ acts on $L$ by the same scalar as on $L^\prime$, then $L$ is isomorphic to $L^\prime$.
\end{cor}

\begin{proof}
${\rm (1)}$ It directly follows from Theorem $\ref{primitive-representations-of-Hopf-PBW-deformations-is-proper}$ ${\rm (2)}$.

${\rm (2)}$ Suppose that ${\rm dim}L = n+1$ and ${\rm dim}L^\prime = n^\prime+1$, then by ${\rm (1)}$ one has
$L \cong L(n,1) = L_{\rm \bf p}(n,+)$ and $L^\prime \cong L(n^\prime,1) = L_{\rm \bf p}(n^\prime,+).$
By Theorem $\ref{classification-of-primitive-representations-of-Hopf-PBW-deformations}$ ${\rm (3)}$,
$C_q$ acts on $L$ (resp. $L^\prime$) by the same scalar $c_q(n)$ (resp. $c_q(n^\prime)$) as on $L_{\rm \bf p}(n,+)$ (resp. $L_{\rm \bf p}(n^\prime,+)$).
If $C_q$ acts on $L$ by the same scalar as on $L^\prime$, then $c_q(n) = c_q(n^\prime)$. By direct calculations, $c_q(n) = c_q(n^\prime)$ if and only if
\begin{eqnarray*}\label{}
q^{-(n+1)m}(q^{(n + n^\prime +2)m} - 1)(q^{(n-n^\prime)m}-1)=0,
\end{eqnarray*}
which is equivalent to say $n=n^\prime$. Therefore, $L \cong L(n,1) \cong L^\prime$.
\end{proof}

\begin{thm}\label{new-proof-of-semisimplicity-of-fd-representations-of-Hopf-PBW-deformations}
Each finite dimensional $U_{q}(\mathfrak{sl}^{*}_2,\kappa)$-module is semisimple.
\end{thm}

\begin{proof}
By Proposition $\ref{reps-are-all-weight-reps-for-classical-quantum-groups}$, Theorem $\ref{unique-decomposition-theorem-of-finite-dimensional-reps-of-U}$ and Theorem $\ref{equivalence-functors-translation-transitive-functors}$ ${\rm (2)}$, we only need to show that
 each indecomposable $U_{q}(\mathfrak{sl}^{*}_2,\kappa)$-module $M$ in the principle blocks $\mathcal{O}^\kappa_1$ and $\mathcal{O}^\kappa_q$ is simple, i.e., the length $l(M)$ of $M$ is 1.

Assume that
$g(x) = (x - \mu_1)^{r_1} (x - \mu_2)^{r_2} \cdots (x - \mu_s)^{r_s}$
is the characteristic polynomial of $C_q$ acting on $M$. Then $M$ is the direct sum of the generalized eigenspaces for $C_q$, i.e.,
$M = {\mathop\bigoplus\limits_{k=1}^{s}} M^{\mu_i},$
where $M^{\mu_i} = \left\{ v\in M | (C_q - \mu_i)^{r_i}v = 0\right\}$. Since $C_q$ is central in $U_{q}(\mathfrak{sl}^{*}_2,\kappa)$, each $M^{\mu_i}$ is a submodule of $M$.
Hence $M = M^{\mu} = \left\{ v\in M | (C_q - \mu)^{r}v = 0\right\}$ for some $\mu$ because $M$ is indecomposable.

Suppose that $l(M) = l$. Pick a composition series
\begin{eqnarray}\label{composition-series-of-semisimple-proof-of-module}
0 = M_0 \subset M_1 \subset M_2 \subset \cdots \subset M_l = M
\end{eqnarray}
of $M$. Since $M = M^{\mu}$, then $C_q - \mu$ acts nilpotently on each $M_i / M_{i-1}(1 \leq i \leq l)$. On the other hand, by Schur lemma $C_q$ acts by a scalar $\nu_i$ on $M_i / M_{i-1}$.
Hence for all $1 \leq i \leq l$ one has $\nu_i = \mu$.
Moreover, it can be seen from the proof of Proposition $\ref{properties-of-weight-rep-cat}$ that each $M_i / M_{i-1}(1 \leq i \leq l)$
is in the same principle block as $M$. Therefore, by Corollary $\ref{cor-classification-of-primitive-representations-of-Hopf-PBW-deformations}$ there exists an integer $n_0 \geq 0$ such that each $M_i / M_{i-1}(1 \leq i \leq l)$ is isomorphic to $L(n_0,1)$.

Let $N$ be a submodule of $M$. Since ${\rm dim}M_\nu = {\rm dim}N_\nu + {\rm dim}(M/N)_\nu$ for any $\nu \in \Lambda_M$,
 then we apply this repeatedly to the composition series $(\ref{composition-series-of-semisimple-proof-of-module})$ and obtain
 \begin{eqnarray*}\label{}
{\rm dim}M_\nu = {\mathop\sum\limits_{i=1}^l} {\rm dim}(M_i / M_{i-1})_\nu = l {\rm dim}L(n_0,1)_\nu = l
\end{eqnarray*}
 for any $\nu \in \Lambda_M$. It follows from Proposition $\ref{quasi-primitive-rep-before-classification-of-primitive-representations-of-Hopf-PBW-deformations}$
that the dimensions of all the weight spaces of $M$ are equal to $1$, i.e., $l =1$. Therefore, $M$ is simple.
\end{proof}

\section{Tensor-categorical realizations of the principle blocks of $U_{q}(\mathfrak{sl}^{*}_2,\kappa)\mbox{-}{\rm \bf mod}_{\rm wt}$}\label{section-5}

In this section, we will introduce some categories attached to certain sequences of (deformed) preprojective algebras of Dynkin type $\A$, then make use
of them to realize the principle block(s) of the category $U_{q}(\mathfrak{sl}^{*}_2,\kappa)\mbox{-}{\rm \bf mod}_{\rm wt}$ on the level of tensor category.

\subsection{The category of finite dimensional representations of deformed preprojective algebras}\label{section-5-1}

Let $(Q_\infty,I_{\infty})$ be the following infinite bound quiver with
\begin{center}
\begin{tikzpicture}[scale=0.6]
  \tikzstyle{every node}=[draw,thick,circle,fill=white,minimum size=5pt, inner sep=1pt]
  \node (-3) at (-2.5,3)  {};
  \node (-2) at (0.5,3)  {};
   \node (-1) at (2,3)  {};
    \node (0) at (5,3)  {};
    \node (1) at (8,3) {};
    \node (2) at (9.5,3) {};
    \node (3) at (12.55,3)  {};
    \node (d) at (16.2,2.8) [minimum size=0pt,inner sep=0pt,label=right:] {};

    \node (v0) at (-6,3) [minimum size=0pt,inner sep=0pt,label=right:$Q_\infty:$] {};
     \node (v-1) at (-3,2.4) [minimum size=0pt,inner sep=0pt,label=right:$-m$] {};
    \node (v-2) at (-0.3,2.4) [minimum size=0pt,inner sep=0pt,label=right:$1-m$] {};
    \node (v1) at (1.65,2.4) [minimum size=0pt,inner sep=0pt,label=right:$-1$] {};
    \node (v2) at (4.65,2.4) [minimum size=0pt,inner sep=0pt,label=right:$0$] {};
    \node (v3) at (7.65,2.4) [minimum size=0pt,inner sep=0pt,label=right:$1$] {};
    \node (v4) at (8.5,2.4) [minimum size=0pt,inner sep=0pt,label=right:$m-1$] {};
    \node (v5) at (12,2.4) [minimum size=0pt,inner sep=0pt,label=right:$m$] {};

 \node (A0) at (-1.7,3.45) [minimum size=0pt,inner sep=0pt,label=right:$\alpha^*_{-m}$] {};
    \node (A01) at (-1.7,2.5) [minimum size=0pt,inner sep=0pt,label=right:$\alpha_{-m}$] {};
    \node (A1) at (3.2,3.45) [minimum size=0pt,inner sep=0pt,label=right:$\alpha^*_{-1}$] {};
    \node (A11) at (3.2,2.5) [minimum size=0pt,inner sep=0pt,label=right:$\alpha_{-1}$] {};
    \node (A2) at (6.2,3.45) [minimum size=0pt,inner sep=0pt,label=right:$\alpha^*_0$] {};
    \node (A21) at (6.2,2.5) [minimum size=0pt,inner sep=0pt,label=right:$\alpha_0$] {};
    \node (A3) at (10.2,3.45) [minimum size=0pt,inner sep=0pt,label=right:$\alpha^*_{m-1}$] {};
     \node (A31) at (10.2,2.5) [minimum size=0pt,inner sep=0pt,label=right:$\alpha_{m-1}$] {};

\draw [dotted,line width=1pt] (-3.7,3)--(-2.7,3);
\draw[->, line width=0.7pt] (-2.3,2.9) -- (0.3,2.9);
  \draw[->, line width=0.7pt] (0.3,3.1) -- (-2.3,3.1);
  \draw [dotted,line width=1pt] (1.7,3)--(0.7,3);
  \draw[->, line width=0.7pt] (2.2,2.9) -- (4.8,2.9);
  \draw[->, line width=0.7pt] (4.8,3.1) -- (2.2,3.1);
  \draw[->, line width=0.7pt] (5.2,2.9) -- (7.8,2.9);
  \draw[->, line width=0.7pt] (7.8,3.1) -- (5.2,3.1);
  \draw [dotted,line width=1pt] (8.3,3)--(9.3,3);
  \draw[->, line width=0.7pt] (9.7,2.9) -- (12.3,2.9);
  \draw[->, line width=0.7pt] (12.3,3.1) -- (9.7,3.1);
  \draw [dotted,line width=1pt] (12.8,3)--(13.8,3);
\end{tikzpicture}
\end{center}
and the set $I_{\infty}$ of relations given by
\begin{eqnarray*}\label{}
\quad \alpha_i \alpha_i^* - \alpha_{i+1}^* \alpha_{i+1} - \lambda_{i+1} s_{i+1},
\end{eqnarray*}
where $\lambda_i \in \mathbb{C}$ and $s_i$ are the trivial paths for $i \in \mathbb{Z}$.

\begin{defn}\label{definition-of-fd-reps-of-preprojective-algs-of-type-A-and-morphism}
A finite dimensional representation $V$ of the infinite bound quiver $(Q_\infty,I_{\infty})$ is defined by the following data:\\
${\rm (1)}$ To each vertex $i$ in $Q_\infty$ is associated a finite dimensional vector space $V_i$ with ${\rm dim} V_i = n_i$, and there are only finite nonzero vector spaces $V_i~(i \in \mathbb{Z})$.\\
${\rm (2)}$ To each arrow $i \xrightarrow{\alpha_i} i+1$ (resp. $i \overset{\alpha^*_i}\longleftarrow i+1$) in $Q_\infty$ is associated a linear map $V_i \xrightarrow{E^V_i} V_{i+1}$ (resp. $V_{i+1} \xrightarrow{F^V_i} V_i$), and $E^V_i, F^V_i(i\in \mathbb{Z})$ satisfy
 \begin{eqnarray}\label{relation-definition-of-the-cat-of-fd-reps-of-preprojective-algs-of-type-A}
E^V_i F^V_i - F^V_{i+1} E^V_{i+1} = \lambda_{i+1} I_{n_{i+1}}~(i \in \mathbb{Z}).
\end{eqnarray}
Such a representation is denoted by $V = (V_i, E^{V}_i, F^{V}_i)_{i\in \mathbb{Z}}$.
\end{defn}

The category of finite dimensional representations of deformed preprojective algebras of Dynkin type $\A$ is defined as follows.

\begin{defn}\label{definition-of-the-cat-of-fd-reps-of-preprojective-algs-of-type-A}
Let $\overrightarrow{\lambda} = (\lambda_i)_{i \in \mathbb{Z}}$. Define a category ${\rm \bf rep}\Pi^{\overrightarrow{\lambda}}_\infty$ as follows: \\
${\rm (1)}$ The objects in ${\rm \bf rep}\Pi^{\overrightarrow{\lambda}}_\infty$ are just all the finite dimensional representations of $(Q_\infty,I_{\infty})$.\\
${\rm (2)}$ For any two objects $V = (V_i, E^{V}_i, F^{V}_i)_{i\in \mathbb{Z}}$ and $W = (W_i, E^{W}_i, F^{W}_i)_{i\in \mathbb{Z}}$ in ${\rm \bf rep}\Pi^{\overrightarrow{\lambda}}_\infty$, one has
\begin{eqnarray*}\label{}
{\rm Hom}_{{\rm \bf rep}\Pi^{\overrightarrow{\lambda}}_\infty}(V,W) = \left\{f=(f_i)_{i\in \mathbb{Z}} \Big| E^W_i f_i = f_{i+1} E^{V}_i, F^W_i f_{i+1} = f_{i} F^{V}_i ~{\rm for}~i \in \mathbb{Z} \right\}.
\end{eqnarray*}
We call ${\rm \bf rep}\Pi^{\overrightarrow{\lambda}}_\infty$ the category of finite dimensional representations of deformed preprojective algebras of Dynkin type $\A$ with weight $\overrightarrow{\lambda}$.
In particular, when $\overrightarrow{\lambda} = 0$, we call ${\rm \bf rep}\Pi^{0}_\infty$ the category of finite dimensional representations of preprojective algebras of Dynkin type $\A$.
\end{defn}

For any two integers $k \leq l$, let $(Q_{[k,l]},I_{[k,l]})$ be the following bound quiver with
\begin{center}
\begin{tikzpicture}[scale=0.6]
  \tikzstyle{every node}=[draw,thick,circle,fill=white,minimum size=5pt, inner sep=1pt]
   \node (1) at (2,3)  {};
    \node (2) at (5,3)  {};
    \node (3) at (8,3) {};
    \node (4) at (10,3) {};
    \node (5) at (13,3)  {};
    \node (6) at (16,3)  {};
    \node (d) at (16.2,2.8) [minimum size=0pt,inner sep=0pt,label=right:] {};

    \node (v0) at (0,3) [minimum size=0pt,inner sep=0pt,label=right:$Q_{[k,l]}:$] {};
    \node (v1) at (1.65,2.4) [minimum size=0pt,inner sep=0pt,label=right:$k$] {};
    \node (v2) at (4.3,2.4) [minimum size=0pt,inner sep=0pt,label=right:$k+1$] {};
    \node (v3) at (7.3,2.4) [minimum size=0pt,inner sep=0pt,label=right:$k+2$] {};
    \node (v4) at (8.9,2.4) [minimum size=0pt,inner sep=0pt,label=right:$l-2$] {};
    \node (v5) at (12.15,2.4) [minimum size=0pt,inner sep=0pt,label=right:$l-1$] {};
    \node (v6) at (15.65,2.4) [minimum size=0pt,inner sep=0pt,label=right:$l$] {};

    \node (A1) at (3.2,3.45) [minimum size=0pt,inner sep=0pt,label=right:$\alpha^*_k$] {};
    \node (A11) at (3.2,2.5) [minimum size=0pt,inner sep=0pt,label=right:$\alpha_k$] {};
    \node (A2) at (5.8,3.45) [minimum size=0pt,inner sep=0pt,label=right:$\alpha^*_{k+1}$] {};
    \node (A21) at (5.8,2.5) [minimum size=0pt,inner sep=0pt,label=right:$\alpha_{k+1}$] {};
    \node (A3) at (10.6,3.45) [minimum size=0pt,inner sep=0pt,label=right:$\alpha^*_{l-2}$] {};
     \node (A31) at (10.6,2.5) [minimum size=0pt,inner sep=0pt,label=right:$\alpha_{l-2}$] {};
    \node (A4) at (13.9,3.45) [minimum size=0pt,inner sep=0pt,label=right:$\alpha^*_{l-1}$] {};
 \node (A41) at (13.9,2.5) [minimum size=0pt,inner sep=0pt,label=right:$\alpha_{l-1}$] {};

  \draw[->, line width=0.7pt] (2.2,2.9) -- (4.8,2.9);
  \draw[->, line width=0.7pt] (4.8,3.1) -- (2.2,3.1);
  \draw[->, line width=0.7pt] (5.2,2.9) -- (7.8,2.9);
  \draw[->, line width=0.7pt] (7.8,3.1) -- (5.2,3.1);
  \draw [dotted,line width=1pt] (3)--(4);
  \draw[->, line width=0.7pt] (10.2,2.9) -- (12.8,2.9);
  \draw[->, line width=0.7pt] (12.8,3.1) -- (10.2,3.1);
  \draw[->, line width=0.7pt] (13.2,2.9) -- (15.8,2.9);
  \draw[->, line width=0.7pt] (15.8,3.1) -- (13.2,3.1);
\end{tikzpicture}
\end{center}
and the set $I_{[k,l]}$ of relations given by
\begin{eqnarray*}\label{}
\alpha_k^* \alpha_k - \lambda_k s_k, \quad \alpha_{i} \alpha_{i}^* - \alpha_{i+1}^* \alpha_{i+1} - \lambda_{i+1} s_{i+1}~(k \leq i \leq l-2),  \quad \alpha_{l-1} \alpha_{l-1}^* - \lambda_l s_l.
\end{eqnarray*}
The bound quiver algebra $k Q_{[k,l]} / I_{{[k,l]}}$ is just the deformed preprojective algebra $\Pi^{{\lambda}_{[k,l]}}$ of Dynkin type $\A_{l-k+1}$ introduced in \cite{WCMH605}, where ${\lambda}_{[k,l]} = (\lambda_i)_{k \leq i \leq l}$.
The algebra $\Pi^{{\lambda}_{[k,l]}}$ is finite dimensional (cf. \cite{WCMH605,DR200,GLS193}), and finitely many objects in ${\rm \bf rep}\Pi^{\overrightarrow{\lambda}}_\infty$ can be simultaneously considered as finite dimensional $\Pi^{{\lambda}_{[k,l]}}$-modules for some integers $k \leq l$. Hence we can obtain the following proposition.

\begin{prop}\label{cat-f-d-reps-local-finite-linear-abel-cat}
The category ${\rm \bf rep}\Pi^{\overrightarrow{\lambda}}_\infty$ is a locally finite $\mathbb{C}$-linear abelian category.
\end{prop}

\begin{defn}\label{chain-rep}
Let $V = \left(V_i, E^{V}_i, F^{V}_i\right)_{i\in \mathbb{Z}}$ be an object in ${\rm \bf rep}\Pi^{\overrightarrow{\lambda}}_\infty$.
If there exist two integers $k$ and $l$ with $k \leq l$ such that $V_i \neq 0$ if and only if $k \leq i \leq l$, then we call $V$ a chain representation.
\end{defn}

According to Definition \ref{chain-rep}, we can immediately obtain the following proposition.

\begin{prop}\label{indecom-reps-are-all-chain-representations}
If $V = \left(V_i, E^{V}_i, F^{V}_i\right)_{i\in \mathbb{Z}}$ is an indecomposable object in ${\rm \bf rep}\Pi^{\overrightarrow{\lambda}}_\infty$,
 then $V$ is a chain representation in ${\rm \bf rep}\Pi^{\overrightarrow{\lambda}}_\infty$.
\end{prop}

\subsection{Tensor-categorical realization of the principle block of $U_{q}(\mathfrak{sl}^{*}_2)\mbox{-}{\rm \bf mod}_{\rm wt}$}\label{section-5-2}

\begin{defn}\label{definition-of-primitive-rep-of-preprojective-algebra-of-type-A}
Let $V = \left(V_i, E^{V}_i, F^{V}_i\right)_{i \in \mathbb{Z}}$ be an object in the category ${\rm \bf rep}\Pi^{0}_\infty$.
If $V$ is indecomposable and there exists an integer $m \geq 1$ such that $V_i \neq 0$ if and only if $0 \leq i \leq m-1$, then we call $V$ a primitive object.
\end{defn}

For any $k \in \mathbb{Z}$, define a shift functor as follows
\begin{eqnarray*}\label{}
[k]: {\rm \bf rep}\Pi^{0}_\infty &\longrightarrow& {\rm \bf rep}\Pi^{0}_\infty, \\
V &\longmapsto& V[k],\\
V \overset{f}\longrightarrow W &\longmapsto& V[k] \overset{f[k]}\longrightarrow W[k],
\end{eqnarray*}
where $V[k] = \left(V[k]_i, E[k]_i, F[k]_i\right)_{i \in \mathbb{Z}} = \left(V_{i-k}, E^{V}_{i-k}, F^{V}_{i-k}\right)_{i \in \mathbb{Z}}$ and $f[k] = \left( f[k]_i \right)_{i \in \mathbb{Z}} = \left( f_{i-k} \right)_{i \in \mathbb{Z}}$.

The following proposition shows that each indecomposable object in the category ${\rm \bf rep}\Pi^{0}_\infty$ can
be obtained by applying some shift functor on a primitive one.

\begin{prop}
Let $V = \left(V_i, E^V_i, F^V_i\right)_{i \in \mathbb{Z}}$ be an indecomposable object in the category ${\rm \bf rep}\Pi^{0}_\infty$. Then
there exists a primitive object $W = \left(W_j, E^W_j, F^W_j\right)_{j \in \mathbb{Z}}$ in ${\rm \bf rep}\Pi^{0}_\infty$ such that $W[-i_0] = V$, where $i_0 = {\rm min}\left\{i\in \mathbb{Z}| V_i \neq 0\right\}$.
\end{prop}

\begin{proof}
By Proposition $\ref{indecom-reps-are-all-chain-representations}$, $V$ is a chain representation. Take $W = V[i_0]$, then it is easy to check that
$W$ is a primitive object such that $W[-i_0] = V$.
\end{proof}

Define a functor $\Omega: {\rm \bf rep}\Pi^{0}_\infty \longrightarrow \mathcal{O}^0_1$ as follows
\begin{eqnarray*}\label{}
\Omega: {\rm \bf rep}\Pi^{0}_\infty &\longrightarrow& \mathcal{O}^0_1,\\
V &\longmapsto& \Omega(V),\\
V \overset{f}\longrightarrow W &\longmapsto& \Omega(V) \overset{\Omega (f)}\longrightarrow \Omega(W),
\end{eqnarray*}
where as a vector space $\Omega(V) = {\mathop\bigoplus\limits_{i \in \mathbb{Z}}} V_i$ and the action of $U_{q}(\mathfrak{sl}^{*}_2)$ on $\Omega(V)$ is given by
\begin{eqnarray*}\label{}
\left\{
 \begin{array}{ll}
K v = q^{2i} v,\\
E v = E^{V}_i(v),\\
F v = F^{V}_{i-1}(v)
\end{array}
\right.
\end{eqnarray*}
for any $v\in V_i$, while $\Omega (f) = {\mathop\bigoplus\limits_{i \in \mathbb{Z}}} f_i$.
Define a functor $\Omega^{-1}:  \mathcal{O}^0_1 \longrightarrow {\rm \bf rep}\Pi^{0}_\infty$ as follows
\begin{eqnarray*}\label{}
\Omega^{-1}: \mathcal{O}^0_1 &\longrightarrow& {\rm \bf rep}\Pi^0_\infty, \\
M &\longmapsto& \Omega^{-1}(M),\\
M \overset{f}\longrightarrow N &\longmapsto& \Omega^{-1}(M) \overset{\Omega^{-1} (f)}\longrightarrow \Omega^{-1}(N),
\end{eqnarray*}
where $\Omega^{-1}(M) = V = (V_i, E^{V}_i, F^{V}_i)_{i\in \mathbb{Z}}$ is given by
\begin{eqnarray*}\label{}
\left\{
 \begin{array}{ll}
V_i = M_{q^{2i}},\\
E^{V}_i := M_{q^{2i}} \overset{E}\longrightarrow M_{q^{2(i+1)}},\\
F^{V}_i := M_{q^{2(i+1)}} \overset{F}\longrightarrow M_{q^{2i}},
\end{array}
\right.
\end{eqnarray*}
 while $\Omega^{-1} (f) = (g_i)_{i\in \mathbb{Z}}$ with $g_i$ the restriction of $f$ on $M_{q^{2i}}$.

Now we can realize the principle block $\mathcal{O}^{0}_1$ of $U_{q}(\mathfrak{sl}^{*}_2)\mbox{-}{\rm \bf mod}_{\rm {wt}}$ on the level of category.

\begin{thm}\label{equivalence-between-categories-and-primitive-reps-of-U-and-preprojetive-algs-of-A}
The functor $\Omega: {\rm \bf rep}\Pi^{0}_\infty \longrightarrow \mathcal{O}^0_1$ is an isomorphism of categories which preserves the primitive objects.
\end{thm}

\begin{proof}
It is easy to check that $\Omega$ is additive, $\Omega \Omega^{-1} = {\rm Id}_{\mathcal{O}^0_1}$ and $\Omega^{-1} \Omega = {\rm Id}_{{\rm \bf rep}\Pi^0_\infty}$.
 Hence $\Omega: {\rm \bf rep}\Pi^{0}_\infty \longrightarrow \mathcal{O}^{0}_1$ is an isomorphism of categories.
 Moreover, it follows from Definition $\ref{primitive-rep-of-new-type-quantum-group}$, Definition $\ref{definition-of-primitive-rep-of-preprojective-algebra-of-type-A}$ and the definition of $\Omega$
  that $V$ is primitive if and only if $\Omega(V)$ is primitive.
\end{proof}

Theorem $\ref{equivalence-between-categories-and-primitive-reps-of-U-and-preprojetive-algs-of-A}$ tells us that the theory of preprojective algebras of Dynkin type $\A$ can afford many useful information for investigating the new type quantum group $U_{q}(\mathfrak{sl}^{*}_2)$, and vice versa.

Denote by $\Pi_m$ the preprojective algebra of Dynkin type $\A_m$ corresponding to the bound quiver $(Q_{[0,m-1]},I_{[0,m-1]})$,
 and by ${\rm \bf rep}\Pi_m$ the category of finite dimensional representations of $\Pi_m$.
For any object $V = \left(V_i, E_i, F_i\right)_{0\leq i\leq m-1}$ in ${\rm \bf rep}\Pi_m$,
denote by ${\rm \underline{dim}}V \in \mathbb{N}^m$ the dimension vector of $V$.
Let $\left(-, -\right)$ be the symmetric bilinear form on $\mathbb{Z}^m$ defined by
\begin{eqnarray*}\label{}
\left(\beta, \gamma\right) = {\mathop\sum\limits_{i=0}^{m-1}}2\beta_i \gamma_i - {\mathop\sum\limits_{i \rightarrow j \in Q_{[0,m-1]}}} \beta_i \gamma_j,
\end{eqnarray*}
where $\beta = \left(\beta_i\right)_{0 \leq i \leq m-1}$ and $\gamma = \left(\gamma_i\right)_{0 \leq i \leq m-1}$.
For any two objects $V$ and $W$ in ${\rm \bf rep}\Pi_m$, Crawley-Boevey proved the following homological formula for preprojective algebras in \cite{WC1027}
\begin{eqnarray}\label{Crawley-Boevey-homological-formula-for-preprojective algebras}
{\rm dim}{\rm Ext}_{\Pi_m}^1(V,W) = {\rm dim}{\rm Hom}_{\Pi_m}(V,W) + {\rm dim}{\rm Hom}_{\Pi_m}(W,V) - \left( {\rm \underline{dim}}V, {\rm \underline{dim}}W \right).
\end{eqnarray}
By Theorem $\ref{block-decomposition-of-weight-rep-cat-of-U}$, Theorem $\ref{equivalence-functors-translation-transitive-functors}$, Theorem $\ref{equivalence-between-categories-and-primitive-reps-of-U-and-preprojetive-algs-of-A}$ and $(\ref{Crawley-Boevey-homological-formula-for-preprojective algebras})$,
we can obtain the following formulas in the category $U_{q}(\mathfrak{sl}^{*}_2)\mbox{-}{\rm \bf mod}_{\rm wt}$, which are a little different from those in the category $U_{q}(\mathfrak{sl}^{*}_2)\mbox{-}{\rm \bf mod}$ obtained in Proposition $\ref{1st-ext-group-of-simple-modules}$.

\begin{prop}\label{1st-ext-group-of-simple-modules-in-the-category-of-weight-reps}
Assume that $q \in \mathbb{C}^\times$ is not a root of unity, and $\lambda, \mu \in \mathbb{C}^\times$.
Then one has
\begin{eqnarray*}\label{}
{\rm Ext}^1_{U_{q}(\mathfrak{sl}^{*}_2)\mbox{-}{\rm \bf mod}_{\rm wt}}(\mathbb{C}_\lambda,\mathbb{C}_\mu) \cong \left\{
 \begin{array}{ll}
\mathbb{C},~\quad~&{\rm if}~\lambda = q^2 \mu ~{\rm or}~ \lambda = q^{-2} \mu,\\
0,~\quad~&{\rm otherwise}.
\end{array}
\right.
\end{eqnarray*}
\end{prop}

In the following we aim to establish a tensor equivalence between the principle block $\mathcal{O}^0_1$ of $U_{q}(\mathfrak{sl}^{*}_2)\mbox{-}{\rm \bf mod}_{\rm wt}$ and the category ${\rm \bf rep}\Pi^{0}_\infty$ of finite dimensional representations of preprojective algebras of Dynkin type $\A$.

\begin{thm}\label{rep-cats-of-pre-quantum-group-and-preprojective-algs-are-tensor-cats}
${\rm (1)}$ The principle block $\mathcal{O}^0_1$ of $U_{q}(\mathfrak{sl}^{*}_2)\mbox{-}{\rm \bf mod}_{\rm wt}$ is a tensor category.\\
${\rm (2)}$ The category ${\rm \bf rep}\Pi^{0}_\infty$ of finite dimensional representations of preprojective algebras of Dynkin type $\A$ is a tensor category.
\end{thm}

\begin{proof}
${\rm (1)}$ Since $U_{q}(\mathfrak{sl}^{*}_2)$ is a Hopf algebra with bijective antipode,
then the category $U_{q}(\mathfrak{sl}^{*}_2)\mbox{-}{\rm \bf mod}$ of finite dimensional $U_{q}(\mathfrak{sl}^{*}_2)$-modules is a tensor category (see Section 1.1 in \cite{BK2000}, Corollary 5.3.7 in \cite{EGNO2015} or Proposition {\bf X1.3.1} in \cite{Kass95}).
The principle block $\mathcal{O}^0_1$ of $U_{q}(\mathfrak{sl}^{*}_2)\mbox{-}{\rm \bf mod}_{\rm wt}$ is just a tensor subcategory of $U_{q}(\mathfrak{sl}^{*}_2)\mbox{-}{\rm \bf mod}$, where the associativity and unit constraints
$\mathbf{a},\mathbf{l},\mathbf{r}$ are just those of the category ${\rm \bf Vec}$ of finite dimensional $\mathbb{C}$-vector spaces.

${\rm (2)}$
For any two objects $V = (V_i, E^{V}_i, F^{V}_i)_{i\in \mathbb{Z}}, W = (W_j, E^{W}_j, F^{W}_j)_{j\in \mathbb{Z}} \in {\rm \bf rep}\Pi^{0}_\infty$,
define the tensor product $V \otimes W = ((V \otimes W)_k, E^{V \otimes W}_k, F^{V \otimes W}_k)_{k\in \mathbb{Z}}$ by
\begin{eqnarray*}\label{}
\left\{
 \begin{array}{ll}
(V \otimes W)_k = {\mathop\bigoplus\limits_{i + j = k}} (V_i \otimes W_j),\\
E^{V \otimes W}_k = {\mathop\bigoplus\limits_{i + j = k}} (E^{V}_i \otimes q^{2j} 1_{W_j} + 1_{V_i} \otimes E^{W}_j),\\
F^{V \otimes W}_k = {\mathop\bigoplus\limits_{i + j = k +1}} (F^{V}_{i-1} \otimes 1_{W_j} + q^{-2i} 1_{V_i} \otimes F^{W}_{j-1}).
\end{array}
\right.
\end{eqnarray*}
For any two morphisms
\begin{eqnarray*}\label{}
&& V = (V_i, E^{V}_i, F^{V}_i)_{i\in \mathbb{Z}} \xrightarrow{f = (f_i)_{i\in \mathbb{Z}}} V^\prime= (V^\prime_i, E^{V^\prime}_i, F^{V^\prime}_i)_{i\in \mathbb{Z}}, \\
&& W = (W_j, E^{W}_j, F^{W}_j)_{j\in \mathbb{Z}} \xrightarrow{g = (g_j)_{j\in \mathbb{Z}}} W^\prime= (W^\prime_j, E^{W^\prime}_j, F^{W^\prime}_j)_{j\in \mathbb{Z}}
\end{eqnarray*}
in ${\rm \bf rep}\Pi^{0}_\infty$, define the tensor product $f \otimes g = ((f \otimes g)_k)_{k\in \mathbb{Z}}$ by
\begin{eqnarray*}\label{}
(f \otimes g)_k = {\mathop\bigoplus\limits_{i + j = k}} (f_i \otimes g_j).
\end{eqnarray*}
Take $\mathds{1} = (\mathds{1}_i, E^{\mathds{1}}_i, F^{\mathds{1}}_i)_{i\in \mathbb{Z}}  \in {\rm \bf rep}\Pi^{0}_\infty$ with $\mathds{1}_i = 0$ except $\mathds{1}_0 = \mathbb{C}$ and $E^{\mathds{1}}_i, F^{\mathds{1}}_i = 0$ for all $i \in \mathbb{Z}$.
Obviously ${\rm End}(\mathds{1}) \cong \mathbb{C}$.
For any objects $U, V, W \in {\rm \bf rep}\Pi^{0}_\infty$, define the associativity constraint $\mathfrak{a}_{U,V,W}: (U\otimes V)\otimes W \longrightarrow U\otimes (V\otimes W)$,
 the left unit constraint $\mathfrak{l}_V: \mathds{1} \otimes V \longrightarrow V$ and the right unit constraint $\mathfrak{r}_V: V \otimes \mathds{1} \longrightarrow V$ by
\begin{eqnarray*}
  (\mathfrak{a}_{U,V,W})_s = {\mathop\bigoplus\limits_{i + j + k = s}}\mathbf{a}_{U_i,V_j,W_k}, \ \ (\mathfrak{l}_V)_s = \mathbf{l}_{V_s}, \ \ (\mathfrak{r}_V)_s = \mathbf{r}_{V_s}
\end{eqnarray*}
where $\mathbf{a}_{U_i,V_j,W_k}, \mathbf{l}_{V_s}$ and $\mathbf{r}_{V_s}$ are respectively the associativity constraint, left unit constraint and right unit constraint in the category ${\rm \bf Vec}$.
For any object $V = (V_i, E^{V}_i, F^{V}_i)_{i\in \mathbb{Z}} \in {\rm \bf rep}\Pi^{0}_\infty$, define the left dual $V^* = (V^*_i, E^{V^*}_i, F^{V^*}_i)_{i\in \mathbb{Z}}$
 and right dual ${^*V} = ({^*V}_i, E^{^*V}_i, F^{^*V}_i)_{i\in \mathbb{Z}}$ of $V$ as follows
 \begin{eqnarray*}\label{}
\left\{
 \begin{array}{ll}
~V^*_i = (V_{-i})^\star,\ \ \ \ E^{V^*}_i = - q^{2(i+1)} (E^{V}_{-i-1})^\star,\ \ \ F^{V^*}_i = - q^{-2(i+1)} (F^{V}_{-i-1})^\star,\\
{^*V}_i = (V_{-i})^\star,\ \ \ \ E^{^*V}_i = - q^{2i} (E^{V}_{-i-1})^\star,\ \ \ \ \ \ \ \ F^{^*V}_i = - q^{-2i} (F^{V}_{-i-1})^\star.
\end{array}
\right.
\end{eqnarray*}
For the left dual $V^*$, the evaluation ${\rm \bf ev}_V = (({\rm \bf ev}_V)_i)_{i\in \mathbb{Z}}: V^* \otimes V \longrightarrow \mathds{1}$ is determined by
\begin{eqnarray*}\label{}
({\rm \bf ev}_V)_0: {\mathop\bigoplus\limits_{i\in \mathbb{Z}}} \left(V^*_i \otimes V_{-i}\right) &\longrightarrow& \mathbb{C},\\
f \otimes v &\longmapsto& f(v),
\end{eqnarray*}
for any $f \otimes v \in V^*_i \otimes V_{-i}$, while the coevaluation ${\rm \bf coev}_V: \mathds{1} \longrightarrow V \otimes V^*$ is determined by
\begin{eqnarray*}\label{}
({\rm \bf coev}_V)_0: \mathbb{C} &\longrightarrow& {\mathop\bigoplus\limits_{i\in \mathbb{Z}}} \left(V_{i} \otimes V^*_{-i}\right),\\
1 &\longmapsto& {\mathop\sum\limits_{i\in \mathbb{Z}}} {\mathop\sum\limits_{j = 1}^{\dim V_i}}\left(v_{ij} \otimes f_{ij}\right),
\end{eqnarray*}
where $\{v_{ij} | 1 \leq j \leq \dim V_i\}$ is a basis of $V_i$ and $\{f_{ij} | 1 \leq j \leq \dim V_i\}$ is the dual basis of $V^*_{-i}$.
The evaluation and coevaluation for the right dual ${^*V}$ can be defined similarly.
It can be directly checked that $\left({\rm \bf rep}\Pi^{0}_\infty, \otimes, \mathds{1}, \mathfrak{a}, \mathfrak{l}, \mathfrak{r}\right)$ is a rigid monoidal category.
Now the statement in ${\rm (2)}$ follows from Proposition \ref{cat-f-d-reps-local-finite-linear-abel-cat}.
\end{proof}

\begin{thm}\label{tensor-equivalence-between-categories-and-primitive-reps-of-U-and-preprojetive-algs-of-A}
The functor $\Omega: {\rm \bf rep}\Pi^{0}_\infty \longrightarrow \mathcal{O}^0_1$ is a tensor equivalence.
\end{thm}

\begin{proof}
Obviously $\Omega(\mathds{1}) = \mathbb{C}_1$, where $\mathbb{C}_1$ defined by $(\ref{simple-module})$ is the unit object in $\mathcal{O}^0_1$.
Define a natural isomorphism $\mathrm{J}: \Omega(-) \otimes \Omega(-) \longrightarrow \Omega(- \otimes -)$ as follows
\begin{eqnarray}\label{natural isomorphism in tensor equivalence}
\mathrm{J}_{V,W}: \Omega(V) \otimes \Omega(W) = {\mathop\bigoplus\limits_{i\in \mathbb{Z}}} V_{i} \otimes {\mathop\bigoplus\limits_{j\in \mathbb{Z}}} W_{j}
 &\longrightarrow& \Omega(V \otimes W) = {\mathop\bigoplus\limits_{k\in \mathbb{Z}}}{\mathop\bigoplus\limits_{i + j = k}} \left(V_{i} \otimes W_j \right),\\
(v_i)_{i\in \mathbb{Z}} \otimes (w_j)_{j\in \mathbb{Z}} &\longmapsto& \left({\mathop\sum\limits_{i + j = k}}\left(v_{i} \otimes w_{j}\right)\right)_{k\in \mathbb{Z}},\nonumber
\end{eqnarray}
where $V,W \in {\rm \bf rep}\Pi^{0}_\infty$. Then we can directly check that $(\Omega, \mathrm{J})$ is a tensor functor from ${\rm \bf rep}\Pi^{0}_\infty$ to $\mathcal{O}^0_1$.
Therefore, it follows from Theorem \ref{equivalence-between-categories-and-primitive-reps-of-U-and-preprojetive-algs-of-A} that $\Omega$ is a tensor equivalence.
\end{proof}

\subsection{Tensor-categorical realization of the principle blocks of $U_{q}(\mathfrak{sl}^{*}_2,\kappa)\mbox{-}{\rm \bf mod}$ with $\kappa \neq 0$}\label{section-5-3}

\begin{defn}
Let $V = \left(V_i, E^V_i, F^V_i\right)_{k \leq i \leq l}$ be a finite dimensional simple representation of the deformed preprojective algebra $\Pi^{{\lambda}_{[k,l]}}$ of Dynkin type $\A_{l-k+1}$.
If $V_i \neq 0$ for all $k\leq i \leq l$, then we call $V$ a primitive representation of $\Pi^{{\lambda}_{[k,l]}}$.
\end{defn}

In this subsection, we always fix the following notations
\begin{eqnarray*}\label{}
&& \lambda_{\overline{0}}: = (\lambda_i)_{i \in \mathbb{Z}} = ([2i]_m)_{i \in \mathbb{Z}},\\
&& \lambda_{\overline{1}}: = (\lambda_i)_{i \in \mathbb{Z}} = ([2i+1]_m)_{i \in \mathbb{Z}}.
\end{eqnarray*}

\begin{prop}\label{charaterizations-of-primitive-representations-of-two-kinds-of-deformed-preprojective-algebras}
${\rm (1)}$
$V = \left(V_i, E_i^V, F_i^V\right)_{k \leq i \leq l}$ is a primitive representation of $\Pi^{{\lambda_{\overline{0}}}_{[k,l]}}$
 if and only if the dimension vector $\alpha_V$ of $V$ is equal to $(\underset{l-k+1}{\underbrace{{1,1,\cdots,1}}})$ and $k + l = 0$.\\
${\rm (2)}$
$V = \left(V_i, E_i^V, F_i^V\right)_{k \leq i \leq l}$ is a primitive representation of $\Pi^{{\lambda_{\overline{1}}}_{[k,l]}}$
 if and only if the dimension vector $\alpha_V$ of $V$ is equal to $(\underset{l-k+1}{\underbrace{{1,1,\cdots,1}}})$ and $k + l +1 = 0$.
\end{prop}

\begin{proof}
We will only prove ${\rm (1)}$, because the proof of ${\rm (2)}$ is similar. The classification result on the dimension vectors of simple representations of deformed preprojective algebras obtained by Crawley-Boevey (See Theorem 1.2 in \cite{WC257}) implies that $V$ is a primitive representation of $\Pi^{{\lambda_{\overline{0}}}_{[k,l]}}$ if and only if the dimension vector $\alpha_V = (\underset{l-k+1}{\underbrace{{1,1,\cdots,1}}})$
and ${{\lambda_{\overline{0}}}_{[k,l]}} \cdot \alpha_V = 0$.
Since
\begin{eqnarray*}\label{}
{{\lambda_{\overline{0}}}_{[k,l]}} \cdot \alpha_V = [l-k+1]_m [l+k]_m
\end{eqnarray*}
and $q \in \mathbb{C}^\times$ is not a root of unity, then ${\rm (1)}$ holds.
\end{proof}

\begin{defn}\label{definition-of-balanced-representations-of-deformed-preprojective-algebras}
${\rm (1)}$
If $V = \left(V_i, E_i^V, F_i^V\right)_{i \in \mathbb{Z}} \in {\rm \bf rep}\Pi^{{\lambda_{\overline{0}}}}_\infty$ satisfies
\begin{eqnarray}\label{}
\left\{
 \begin{array}{ll}
 {\rm dim}V_i = {\rm dim}V_{-i} & {\rm for}~\forall ~ i \in \mathbb{Z},\\
 {\rm dim}V_i \geq {\rm dim}V_{j} & {\rm for}~\forall ~ 0 \leq i \leq j,
 \end{array}
\right.
\end{eqnarray}
then $V = \left(V_i, E_i^V, F_i^V\right)_{i \in \mathbb{Z}}$ is said to be balanced in ${\rm \bf rep}\Pi^{{\lambda_{\overline{0}}}}_\infty$.\\
${\rm (2)}$
 If $V = \left(V_i, E_i^V, F_i^V\right)_{i \in \mathbb{Z}} \in {\rm \bf rep}\Pi^{{\lambda_{\overline{1}}}}_\infty$ satisfies
\begin{eqnarray}\label{}
\left\{
 \begin{array}{ll}
 {\rm dim}V_i = {\rm dim}V_{-(i+1)} & {\rm for}~\forall ~ i \in \mathbb{Z},\\
 {\rm dim}V_i \geq {\rm dim}V_{j} & {\rm for}~\forall ~ 0 \leq i \leq j,
 \end{array}
\right.
\end{eqnarray}
then $V = \left(V_i, E_i^V, F_i^V\right)_{i \in \mathbb{Z}}$ is said to be balanced in ${\rm \bf rep}\Pi^{{\lambda_{\overline{1}}}}_\infty$.
\end{defn}

\begin{prop}
${\rm (1)}$ Each simple object in the categories ${\rm \bf rep}\Pi^{{\lambda_{\overline{0}}}}_\infty$ and ${\rm \bf rep}\Pi^{{\lambda_{\overline{1}}}}_\infty$ is balanced.\\
${\rm (2)}$ Each object in the categories ${\rm \bf rep}\Pi^{{\lambda_{\overline{0}}}}_\infty$ and ${\rm \bf rep}\Pi^{{\lambda_{\overline{1}}}}_\infty$ is balanced.
\end{prop}

\begin{proof}
${\rm (1)}$ Assume that $V = \left(V_i, E_i^V, F_i^V\right)_{i \in \mathbb{Z}}$ is a simple object in ${\rm \bf rep}\Pi^{{\lambda_{\overline{0}}}}_\infty$.
By Proposition $\ref{indecom-reps-are-all-chain-representations}$, $V$ is a chain representation, i.e., there exist two integers $k$ and $l$ with $k \leq l$
such that $V_i \neq 0$ if and only if $k \leq i \leq l$. Hence $V$ can be considered as a primitive representation of $\Pi^{{\lambda_{\overline{0}}}_{[k,l]}}$.
 By Proposition $\ref{charaterizations-of-primitive-representations-of-two-kinds-of-deformed-preprojective-algebras}~{\rm (1)}$ and Definition $\ref{definition-of-balanced-representations-of-deformed-preprojective-algebras}~{\rm (1)}$,
  $V$ is balanced. Similarly, each simple object in the category ${\rm \bf rep}\Pi^{{\lambda_{\overline{1}}}}_\infty$ is also balanced.

${\rm (2)}$ Every object $V = \left(V_i, E_i^V, F_i^V\right)_{i \in \mathbb{Z}} \in {\rm \bf rep}\Pi^{{\lambda_{\overline{0}}}}_\infty$ can be seen as a finite dimensional representation of the deformed preprojective algebra $\Pi^{{\lambda_{\overline{0}}}_{[k,l]}}$ for some $k \leq l$.
 Pick a composition series
 \begin{eqnarray}\label{composition-series-of-balanced-proof-of-rep-f-d}
0 = \mathcal{V}_0 \subset \mathcal{V}_1 \subset \mathcal{V}_2 \subset \cdots \subset \mathcal{V}_m = V
\end{eqnarray}
of $V$. Since the dimension vector ${\rm \underline{dim}}V = {\mathop\sum\limits_{i=0}^{m-1}} {\rm \underline{dim}} \left(\mathcal{V}_{i+1} / \mathcal{V}_{i}\right)$,
then by Proposition $\ref{charaterizations-of-primitive-representations-of-two-kinds-of-deformed-preprojective-algebras}~{\rm (1)}$ and Definition $\ref{definition-of-balanced-representations-of-deformed-preprojective-algebras}~{\rm (1)}$, $V$ is balanced. Similarly, each object in ${\rm \bf rep}\Pi^{{\lambda_{\overline{1}}}}_\infty$ is also balanced.
\end{proof}

We define a functor $\Omega_1: {\rm \bf rep}\Pi^{{\lambda_{\overline{0}}}}_{\infty} \longrightarrow \mathcal{O}^{\kappa}_1$ as follows
\begin{eqnarray*}\label{}
\Omega_1: {\rm \bf rep}\Pi^{{\lambda_{\overline{0}}}}_{\infty} &\longrightarrow& \mathcal{O}^{\kappa}_1,\\
V &\longmapsto& \Omega_1(V),\\
V \overset{f}\longrightarrow W &\longmapsto& \Omega_1(V) \overset{\Omega_1 (f)}\longrightarrow \Omega_1(W),
\end{eqnarray*}
where as a vector space $\Omega_1(V) = {\mathop\bigoplus\limits_{i \in \mathbb{Z}}} V_i$ and the action of $U_{q}(\mathfrak{sl}^{*}_2,\kappa)$ on $\Omega_1(V)$ is given by
\begin{eqnarray*}\label{}
\left\{
 \begin{array}{ll}
K v = q^{2i} v,\\
E v = E^{V}_i(v),\\
F v = F^{V}_{i-1}(v)
\end{array}
\right.
\end{eqnarray*}
for any $v\in V_i$, while $\Omega (f) = {\mathop\bigoplus\limits_{i \in \mathbb{Z}}} f_i$ (there are only finite nonzero $f_i$ in this expression).
Define a functor $\Omega_1^{-1}:  \mathcal{O}^\kappa_1 \longrightarrow {\rm \bf rep}\Pi^{{\lambda_{\overline{0}}}}_{\infty}$ as follows
\begin{eqnarray*}\label{}
\Omega_1^{-1}: \mathcal{O}^\kappa_1 &\longrightarrow& {\rm \bf rep}\Pi^{{\lambda_{\overline{0}}}}_{\infty},\\
M &\longmapsto& \Omega_1^{-1}(M),\\
M \overset{f}\longrightarrow N &\longmapsto& \Omega_1^{-1}(M) \overset{\Omega_1^{-1} (f)}\longrightarrow \Omega_1^{-1}(N),
\end{eqnarray*}
where $\Omega_1^{-1}(M) = V = (V_i, E^{V}_i, F^{V}_i)_{i\in \mathbb{Z}}$ is given by
\begin{eqnarray*}\label{}
\left\{
 \begin{array}{ll}
V_i = M_{q^{2i}},\\
E^{V}_i := M_{q^{2i}} \overset{E}\longrightarrow M_{q^{2(i+1)}},\\
F^{V}_i := M_{q^{2(i+1)}} \overset{F}\longrightarrow M_{q^{2i}}
\end{array}
\right.
\end{eqnarray*}
for $i\in \mathbb{Z}$, while $\Omega_1^{-1} (f) = (g_i)_{i\in \mathbb{Z}}$ with $g_i$ the restriction of $f$ on $M_{q^{2i}}$.

Similarly, we can define another functor $\Omega_q: {\rm \bf rep}\Pi^{{\lambda_{\overline{1}}}}_{\infty} \longrightarrow \mathcal{O}^{\kappa}_q$ as follows
\begin{eqnarray*}\label{}
\Omega_q: {\rm \bf rep}\Pi^{{\lambda_{\overline{1}}}}_{\infty} &\longrightarrow& \mathcal{O}^{\kappa}_q,\\
V &\longmapsto& \Omega_q(V),\\
V \overset{f}\longrightarrow W &\longmapsto& \Omega_q(V) \overset{\Omega_q (f)}\longrightarrow \Omega_q(W),
\end{eqnarray*}
where as a vector space $\Omega_q(V) = {\mathop\bigoplus\limits_{i \in \mathbb{Z}}} V_i$ and the action of $U_{q}(\mathfrak{sl}^{*}_2,\kappa)$ on $\Omega_q(V)$ is given by
\begin{eqnarray*}\label{}
\left\{
 \begin{array}{ll}
K v = q^{2i+1} v,\\
E v = E^{V}_i(v),\\
F v = F^{V}_{i-1}(v)
\end{array}
\right.
\end{eqnarray*}
for any $v\in V_i$, while $\Omega_q (f) = {\mathop\bigoplus\limits_{i \in \mathbb{Z}}} f_i$ (there are only finite nonzero $f_i$ in this expression).
Define a functor $\Omega_q^{-1}:  \mathcal{O}^\kappa_q \longrightarrow {\rm \bf rep}\Pi^{{\lambda_{\overline{1}}}}_{\infty}$ as follows
\begin{eqnarray*}\label{}
\Omega_q^{-1}: \mathcal{O}^\kappa_q &\longrightarrow& {\rm \bf rep}\Pi^{{\lambda_{\overline{1}}}}_{\infty},\\
M &\longmapsto& \Omega_q^{-1}(M),\\
M \overset{f}\longrightarrow N &\longmapsto& \Omega_q^{-1}(M) \overset{\Omega_q^{-1} (f)}\longrightarrow \Omega_q^{-1}(N),
\end{eqnarray*}
where $\Omega_q^{-1}(M) = V = (V_i, E^{V}_i, F^{V}_i)_{i\in \mathbb{Z}}$ is given by
\begin{eqnarray*}\label{}
\left\{
 \begin{array}{ll}
V_i = M_{q^{2i+1}},\\
E^{V}_i := M_{q^{2i+1}} \overset{E}\longrightarrow M_{q^{2i+3}},\\
F^{V}_i := M_{q^{2i+3}} \overset{F}\longrightarrow M_{q^{2i+1}}
\end{array}
\right.
\end{eqnarray*}
for $i\in \mathbb{Z}$, while $\Omega^{-1} (f) = (g_i)_{i\in \mathbb{Z}}$ with $g_i$ the restriction of $f$ on $M_{q^{2i+1}}$.

Now we can realize the principle blocks $\mathcal{O}^{\kappa}_1$ and $\mathcal{O}^{\kappa}_q$ of $U_{q}(\mathfrak{sl}^{*}_2,\kappa)\mbox{-}{\rm \bf mod}~(\kappa \neq 0)$ as categories.

\begin{thm}\label{equivalence-between-categories-and-primitive-reps-of-U-and-deformed-preprojetive-algs-of-A}
The functors $\Omega_1: {\rm \bf rep}\Pi^{{\lambda_{\overline{0}}}}_{\infty} \longrightarrow \mathcal{O}^{\kappa}_1$ and $\Omega_q: {\rm \bf rep}\Pi^{{\lambda_{\overline{1}}}}_{\infty} \longrightarrow \mathcal{O}^{\kappa}_q$ are isomorphisms of categories which preserve the simple objects.
\end{thm}

\begin{proof}
It is easy to check that
\begin{eqnarray*}\label{}
\left\{
 \begin{array}{ll}
\Omega_1 \Omega_1^{-1} = {\rm Id}_{\mathcal{O}^\kappa_1},\ \ \
\Omega_1^{-1} \Omega_1 = {\rm Id}_{{\rm \bf rep}\Pi^{{\lambda_{\overline{0}}}}_{\infty}},\\
\Omega_q \Omega_q^{-1} = {\rm Id}_{\mathcal{O}^\kappa_q},\ \ \
\Omega_q^{-1} \Omega_q = {\rm Id}_{{\rm \bf rep}\Pi^{{\lambda_{\overline{1}}}}_{\infty}}.
\end{array}
\right.
\end{eqnarray*}
 Hence $\Omega_1$ and $\Omega_q$ are isomorphisms of categories. Moreover, it follows from the definitions of $\Omega_1$ and $\Omega_q$ that they both preserve the simple objects.
\end{proof}

The following corollary follows from the definitions of $\Omega_1$ and $\Omega_q$, Theorem \ref{new-proof-of-semisimplicity-of-fd-representations-of-Hopf-PBW-deformations} and \ref{equivalence-between-categories-and-primitive-reps-of-U-and-deformed-preprojetive-algs-of-A}.

\begin{cor}
The categories ${\rm \bf rep}\Pi^{{\lambda_{\overline{0}}}}_\infty$ and ${\rm \bf rep}\Pi^{{\lambda_{\overline{1}}}}_\infty$ are semisimple.
\end{cor}

Next we will establish a tensor equivalence between the categories $\mathcal{O}^{\kappa}_1 \oplus \mathcal{O}^{\kappa}_q$ and ${\rm \bf rep}\Pi^{{\lambda_{\overline{0}}}}_\infty \oplus {\rm \bf rep}\Pi^{{\lambda_{\overline{1}}}}_\infty$.

\begin{thm}
${\rm (1)}$ The direct sum $\mathcal{O}^{\kappa}_1 \oplus \mathcal{O}^{\kappa}_q$ of the principle blocks of $U_{q}(\mathfrak{sl}^{*}_2,\kappa)\mbox{-}{\rm \bf mod}_{\rm wt}$ is a tensor category.\\
${\rm (2)}$ The category ${\rm \bf rep}\Pi^{\lambda_{\overline{0}}}_\infty \oplus {\rm \bf rep}\Pi^{\lambda_{\overline{1}}}_\infty$ is a tensor category.
\end{thm}

\begin{proof}
${\rm (1)}$ By similar arguments in Theorem \ref{rep-cats-of-pre-quantum-group-and-preprojective-algs-are-tensor-cats} ${\rm (1)}$,
the category $\mathcal{O}^{\kappa}_1 \oplus \mathcal{O}^{\kappa}_q$ is a tensor subcategory of $U_{q}(\mathfrak{sl}^{*}_2, \kappa)\mbox{-}{\rm \bf mod}$ with the same associativity and unit constraints
$\mathbf{a},\mathbf{l},\mathbf{r}$ as the category ${\rm \bf Vec}$ of finite dimensional $\mathbb{C}$-vector spaces.
 In this case, the unit object $\mathds{1}_\kappa$ in $\mathcal{O}^{\kappa}_1 \oplus \mathcal{O}^{\kappa}_q$ is the primitive representation $L_{\rm \bf p}(0,+)$ defined in Theorem \ref{classification-of-primitive-representations-of-Hopf-PBW-deformations} ${\rm (2)}$.

${\rm (2)}$
 Assume that $U = U^{\overline{0}} \oplus U^{\overline{1}}, V = V^{\overline{0}} \oplus V^{\overline{1}}, W = W^{\overline{0}} \oplus W^{\overline{1}},
  X = X^{\overline{0}} \oplus X^{\overline{1}}, Y = Y^{\overline{0}} \oplus Y^{\overline{1}} \in {\rm \bf rep}\Pi^{\lambda_{\overline{0}}}_\infty \oplus {\rm \bf rep}\Pi^{\lambda_{\overline{1}}}_\infty$.
Define the tensor product
 $$V \otimes W = (V^{\overline{0}} \oplus V^{\overline{1}}) \otimes (W^{\overline{0}} \oplus W^{\overline{1}}) = T^{\overline{0}} \oplus T^{\overline{1}} = (T^{\overline{0}}_k, E^{T^{\overline{0}}}_k, F^{T^{\overline{0}}}_k)_{k\in \mathbb{Z}}
  \oplus (T^{\overline{1}}_k, E^{T^{\overline{1}}}_k, F^{T^{\overline{1}}}_k)_{k\in \mathbb{Z}},$$ where
\begin{eqnarray*}\label{}
T^{\overline{0}}_k &=& {\mathop\bigoplus\limits_{i + j = k}} (V^{\overline{0}}_i \otimes W^{\overline{0}}_j) \oplus {\mathop\bigoplus\limits_{i + j = k - 1}} (V^{\overline{1}}_i \otimes W^{\overline{1}}_j),\ \ \ \ \ \
T^{\overline{1}}_k =  {\mathop\bigoplus\limits_{i + j = k}} (V^{\overline{0}}_i \otimes W^{\overline{1}}_j \oplus V^{\overline{1}}_i \otimes W^{\overline{0}}_j),\\
E^{T^{\overline{0}}}_k &=& {\mathop\bigoplus\limits_{i + j = k}} (E^{V^{\overline{0}}}_i \otimes q^{2jm} 1_{W^{\overline{0}}_j} + 1_{V^{\overline{0}}_i} \otimes E^{W^{\overline{0}}}_j)
 \oplus {\mathop\bigoplus\limits_{i + j = k-1}} (E^{V^{\overline{1}}}_i \otimes q^{(2j+1)m} 1_{W^{\overline{1}}_j} + 1_{V^{\overline{1}}_i} \otimes E^{W^{\overline{1}}}_j),\\
E^{T^{\overline{1}}}_k &=& {\mathop\bigoplus\limits_{i + j = k}}\left[(E^{V^{\overline{0}}}_i \otimes q^{(2j+1)m} 1_{W^{\overline{1}}_j} + 1_{V^{\overline{0}}_i}\otimes E^{W^{\overline{1}}}_j)
 \oplus (E^{V^{\overline{1}}}_i \otimes q^{2jm} 1_{W^{\overline{0}}_j} + 1_{V^{\overline{1}}_i} \otimes E^{W^{\overline{0}}}_j)\right],\\
F^{T^{\overline{0}}}_k &=& {\mathop\bigoplus\limits_{i + j = k}} (F^{V^{\overline{0}}}_{i-1} \otimes 1_{W^{\overline{0}}_j} + q^{-2im}1_{V^{\overline{0}}_i} \otimes F^{W^{\overline{0}}}_{j-1})
 \oplus {\mathop\bigoplus\limits_{i + j = k-1}} (F^{V^{\overline{1}}}_{i-1} \otimes 1_{W^{\overline{1}}_j} + q^{-(2i+1)m} 1_{V^{\overline{1}}_i} \otimes F^{W^{\overline{1}}}_{j-1}),\\
F^{T^{\overline{1}}}_k &=&  {\mathop\bigoplus\limits_{i + j = k}}\left[(F^{V^{\overline{0}}}_{i-1} \otimes  1_{W^{\overline{1}}_j} + q^{-2im}1_{V^{\overline{0}}_i} \otimes F^{W^{\overline{1}}}_{j-1})
 \oplus (F^{V^{\overline{1}}}_{i-1} \otimes  1_{W^{\overline{0}}_j} + q^{-(2i+1)m}1_{V^{\overline{1}}_i} \otimes F^{W^{\overline{0}}}_{j-1})\right].
\end{eqnarray*}
For any two morphisms
\begin{eqnarray*}\label{}
&& V = V^{\overline{0}} \oplus V^{\overline{1}} \xrightarrow{f = f^{\overline{0}} \oplus f^{\overline{1}} = (f^{\overline{0}}_i)_{i\in \mathbb{Z}} \oplus (f^{\overline{1}}_i)_{i\in \mathbb{Z}}} X = X^{\overline{0}} \oplus X^{\overline{1}}, \\
&& W = W^{\overline{0}} \oplus W^{\overline{1}} \xrightarrow{g = g^{\overline{0}} \oplus g^{\overline{1}} = (g^{\overline{0}}_i)_{i\in \mathbb{Z}} \oplus (g^{\overline{1}}_i)_{i\in \mathbb{Z}}} Y = Y^{\overline{0}} \oplus Y^{\overline{1}}
\end{eqnarray*}
in ${\rm \bf rep}\Pi^{\lambda(1)}_\infty \oplus {\rm \bf rep}\Pi^{\lambda(q)}_\infty$, define the tensor product
 $$f \otimes g = (f \otimes g)^{\overline{0}} \oplus (f \otimes g)^{\overline{1}} = ((f \otimes g)^{\overline{0}}_k)_{k\in \mathbb{Z}} \oplus ((f \otimes g)^{\overline{1}}_k)_{k\in \mathbb{Z}},$$
 where $(f \otimes g)^{\overline{0}}_k = {\mathop\bigoplus\limits_{i + j = k}} (f^{\overline{0}}_i \otimes g^{\overline{0}}_j) \oplus {\mathop\bigoplus\limits_{i + j = k - 1}} (f^{\overline{1}}_i \otimes g^{\overline{1}}_j)$ and
$(f \otimes g)^{\overline{1}}_k = {\mathop\bigoplus\limits_{i + j = k}} (f^{\overline{0}}_i \otimes g^{\overline{1}}_j \oplus f^{\overline{1}}_i \otimes g^{\overline{0}}_j).$
Take $\mathds{1} = (\mathds{1}_i, E^{\mathds{1}}_i, F^{\mathds{1}}_i)_{i\in \mathbb{Z}}  \in {\rm \bf rep}\Pi^{\lambda_{\overline{1}}}_\infty$ with $\mathds{1}_i = 0$ except $\mathds{1}_0 = \mathbb{C}$ and $E^{\mathds{1}}_i, F^{\mathds{1}}_i = 0$ for all $i \in \mathbb{Z}$.
Obviously ${\rm End}(\mathds{1}) \cong \mathbb{C}$.
Define the associativity constraint $\mathfrak{a}_{U,V,W}: (U\otimes V)\otimes W \longrightarrow U\otimes (V\otimes W)$,
 the left unit constraint $\mathfrak{l}_{V}: \mathds{1} \otimes V \longrightarrow V$ and the right unit constraint $\mathfrak{r}_{V}: V \otimes \mathds{1} \longrightarrow V$ by
\begin{eqnarray*}
  && (\mathfrak{a}_{U,V,W})^{\overline{0}}_s = {\mathop\bigoplus\limits_{{i + j + k = s}}} \mathbf{a}_{U^{\overline{0}}_i,V^{\overline{0}}_j,W^{\overline{0}}_k} \oplus {\mathop\bigoplus\limits_{\substack{(d_1,d_2,d_3) \in \mathbb{Z}^3_2 \setminus \{(\overline{0},\overline{0},\overline{0})\}\\ d_1 + d_2 + d_3 = {\overline{0}} \\ {i + j + k = s-1}}}} \mathbf{a}_{U^{d_1}_i,V^{d_2}_j,W^{d_3}_k},\\
  && (\mathfrak{a}_{U,V,W})^{\overline{1}}_s = {\mathop\bigoplus\limits_{\substack{(d_1,d_2,d_3) \in \mathbb{Z}^3_2 \setminus \{(\overline{1},\overline{1},\overline{1})\}\\ d_1 + d_2 + d_3 = {\overline{1}} \\ {i + j + k = s}}}} \mathbf{a}_{U^{d_1}_i,V^{d_2}_j,W^{d_3}_k} \oplus {\mathop\bigoplus\limits_{{i + j + k = s-1}}} \mathbf{a}_{U^{\overline{1}}_i,V^{\overline{1}}_j,W^{\overline{1}}_k}, \\
  && (\mathfrak{l}_V)^{\overline{0}}_s = \mathbf{l}_{V^{\overline{0}}_s},\ \ \ \ (\mathfrak{l}_V)^{\overline{1}}_s = \mathbf{l}_{V^{\overline{1}}_s},\ \ \ \
   (\mathfrak{r}_V)^{\overline{0}}_s = \mathbf{r}_{V^{\overline{0}}_s}, \ \ \ \ (\mathfrak{r}_V)^{\overline{1}}_s = \mathbf{r}_{V^{\overline{1}}_s},
\end{eqnarray*}
where $\mathbf{a}, \mathbf{l}$ and $\mathbf{r}$ are respectively the associativity constraint, left unit constraint and right unit constraint in the category ${\rm \bf Vec}$.
For any object $V = V^{\overline{0}} \oplus V^{\overline{1}} \in {\rm \bf rep}\Pi^{\lambda_{\overline{0}}}_\infty \oplus {\rm \bf rep}\Pi^{\lambda_{\overline{1}}}_\infty$,
 define the left dual $V^* = V^{*\overline{0}} \oplus V^{*\overline{1}} = (V^{*\overline{0}}_i, E^{V^{*\overline{0}}}_i, F^{V^{*\overline{0}}}_i)_{i\in \mathbb{Z}} \oplus (V^{*\overline{1}}_i, E^{V^{*\overline{1}}}_i, F^{V^{*\overline{1}}}_i)_{i\in \mathbb{Z}}$
 and right dual ${^*V} = {^*V}^{\overline{0}} \oplus {^*V}^{\overline{1}} = ({^*V}^{\overline{0}}_i, E^{^*V^{\overline{0}}}_i, F^{^*V^{\overline{0}}}_i)_{i\in \mathbb{Z}} \oplus ({^*V}^{\overline{1}}_i, E^{^*V^{\overline{1}}}_i, F^{^*V^{\overline{1}}}_i)_{i\in \mathbb{Z}}$ of $V$ by
\begin{eqnarray*}\label{}
\left\{
 \begin{array}{ll}
 V^{*\overline{0}}_i = (V^{\overline{0}}_{-i})^\star, \ \ \ \ \ \ \
 E^{V^{*\overline{0}}}_i = - q^{2(i+1)m} (E^{V^{\overline{0}}}_{-i-1})^\star,\ \ \ \
 F^{V^{*\overline{0}}}_i = - q^{-2(i+1)m} (F^{V^{\overline{0}}}_{-i-1})^\star,\\
V^{*\overline{1}}_i = (V^{\overline{1}}_{-i-1})^\star, \ \ \ \  E^{V^{*\overline{1}}}_i = - q^{(2i+3)m} (E^{V^{\overline{1}}}_{-i-2})^\star,\ \ \ \  F^{V^{*\overline{1}}}_i = - q^{-(2i+3)m} (F^{V^{\overline{0}}}_{-i-2})^\star,
\end{array}
\right.
\end{eqnarray*}
 and
\begin{eqnarray*}\label{}
\left\{
 \begin{array}{ll}
 {^*V}^{\overline{0}}_i = (V^{\overline{0}}_{-i})^\star,  \ \ \ \ \ \ \
 E^{{^*V}^{\overline{0}}}_i = - q^{2im} (E^{V^{\overline{0}}}_{-i-1})^\star,\ \ \ \ \ \ \ \ \
 F^{{^*V}^{\overline{0}}}_i = - q^{-2im} (F^{V^{\overline{0}}}_{-i-1})^\star, \\
  {^*V}^{\overline{1}}_i = (V^{\overline{1}}_{-i-1})^\star, \ \ \ \
  E^{{^*V}^{\overline{1}}}_i = - q^{(2i+1)m} (E^{V^{\overline{1}}}_{-i-2})^\star, \ \ \ \
  F^{{^*V}^{\overline{1}}}_i = - q^{-(2i+1)m} (F^{V^{\overline{1}}}_{-i-2})^\star.
\end{array}
\right.
\end{eqnarray*}
For the left dual $V^*$, the evaluation $V^* \otimes V \xrightarrow{{\rm \bf ev}_V = ({\rm \bf ev}_V)^{\overline{0}} \oplus ({\rm \bf ev}_V)^{\overline{1}}} \mathds{1}$ is determined by
\begin{eqnarray*}\label{}
({\rm \bf ev}_V)^{\overline{0}}_0: {\mathop\bigoplus\limits_{i \in \mathbb{Z}}} \left[(V^{*\overline{0}}_i \otimes V^{\overline{0}}_{-i}) \oplus (V^{*\overline{1}}_i \otimes V^{\overline{1}}_{-i-1})\right] &\longrightarrow& \mathbb{C},\\
f \otimes v + g \otimes w &\longmapsto& f(v) + g(w),
\end{eqnarray*}
while the coevaluation
$\mathds{1} \xrightarrow{{\rm \bf coev}_V = ({\rm \bf coev}_V)^{\overline{0}} \oplus ({\rm \bf coev}_V)^{\overline{1}}} V \otimes V^*$ is determined by
\begin{eqnarray*}\label{}
({\rm \bf coev}_V)^{\overline{0}}_0: \mathbb{C} &\longrightarrow& {\mathop\bigoplus\limits_{i \in \mathbb{Z}}} \left[(V^{\overline{0}}_i \otimes V^{*\overline{0}}_{-i}) \oplus (V^{\overline{1}}_i \otimes V^{*\overline{1}}_{-i-1})\right],\\
1 &\longmapsto& {\mathop\sum\limits_{i\in \mathbb{Z}}} \left[{\mathop\sum\limits_{j = 1}^{\dim V^{\overline{0}}_i}}\left(v_{ij} \otimes f_{ij}\right) + {\mathop\sum\limits_{j = 1}^{\dim V^{\overline{1}}_i}}\left(w_{ij} \otimes g_{ij}\right)\right],
\end{eqnarray*}
where $\{v_{ij} | 1 \leq j \leq \dim V^{\overline{0}}_i\}$ (resp. $\{w_{ij} | 1 \leq j \leq \dim V^{\overline{1}}_i\}$) is a basis of $V^{\overline{0}}_i$ (resp. $V^{\overline{1}}_i$)
 and $\{f_{ij} | 1 \leq j \leq \dim V^{\overline{0}}_i\}$ (resp. $\{g_{ij} | 1 \leq j \leq \dim V^{\overline{1}}_i\}$) is the dual basis of $V^{*\overline{0}}_{-i}$ (resp. $V^{*\overline{1}}_{-i-1}$).
The evaluation and coevaluation for the right dual ${^*V}$ can be defined similarly.
It can be directly checked that $\left({\rm \bf rep}\Pi^{\lambda_{\overline{0}}}_\infty \oplus {\rm \bf rep}\Pi^{\lambda_{\overline{1}}}_\infty, \otimes, \mathds{1}, \mathfrak{a}, \mathfrak{l}, \mathfrak{r}\right)$ is a rigid monoidal category.
Now the statement in ${\rm (2)}$ follows from Proposition \ref{cat-f-d-reps-local-finite-linear-abel-cat}.
\end{proof}

Define a new functor $\Omega^\lambda: {\rm \bf rep}\Pi^{\lambda_{\overline{0}}}_\infty \oplus {\rm \bf rep}\Pi^{\lambda_{\overline{1}}}_\infty \longrightarrow \mathcal{O}^{\kappa}_1 \oplus \mathcal{O}^{\kappa}_q$ as follows
\begin{eqnarray*}\label{}
\Omega^\lambda: {\rm \bf rep}\Pi^{\lambda_{\overline{0}}}_\infty \oplus {\rm \bf rep}\Pi^{\lambda_{\overline{1}}}_\infty &\longrightarrow& \mathcal{O}^{\kappa}_1 \oplus \mathcal{O}^{\kappa}_q,\\
V = V^{\overline{0}} \oplus V^{\overline{1}} &\longmapsto& \Omega^\lambda(V) = \Omega_1(V^{\overline{0}}) \oplus \Omega_q(V^{\overline{1}}),\\
V \xrightarrow{f = f^{\overline{0}} \oplus f^{\overline{1}}} W &\longmapsto& \Omega^\lambda(V) \xrightarrow{\Omega^\lambda(f) = \Omega_{1}(f^{\overline{0}}) \oplus \Omega_{q}(f^{\overline{1}})} \Omega^\lambda(W).
\end{eqnarray*}

\begin{thm}\label{tensor-equivalence-between-categories-and-primitive-reps-of-U-and-deformed-preprojetive-algs-of-A}
The functor $\Omega^\lambda: {\rm \bf rep}\Pi^{\lambda_{\overline{0}}}_\infty \oplus {\rm \bf rep}\Pi^{\lambda_{\overline{1}}}_\infty \longrightarrow \mathcal{O}^{\kappa}_1 \oplus \mathcal{O}^{\kappa}_q$ is a tensor equivalence.
\end{thm}

\begin{proof}
Obviously $\Omega^\lambda(\mathds{1}) = \mathds{1}_\kappa$,  where $\mathds{1}_\kappa \in \mathcal{O}^{\kappa}_1 \oplus \mathcal{O}^{\kappa}_q$ is the primitive representation $L_{\rm \bf p}(0,+)$ defined in Theorem \ref{classification-of-primitive-representations-of-Hopf-PBW-deformations} ${\rm (2)}$.
Define a natural isomorphism $\mathrm{J}^\lambda: \Omega^\lambda(-) \otimes \Omega^\lambda(-) \longrightarrow \Omega^\lambda(- \otimes -)$ as follows
{\footnotesize
\begin{eqnarray*}\label{}
\mathrm{J}^\lambda_{V,W}: \Omega^\lambda(V) \otimes \Omega^\lambda(W) = {\mathop\bigoplus\limits_{(d_1,d_2)\in \mathbb{Z}_2^2}} \left({\mathop\bigoplus\limits_{i\in \mathbb{Z}}} V^{d_1}_{i} \otimes {\mathop\bigoplus\limits_{j\in \mathbb{Z}}} W^{d_2}_{j}\right)
 &\longrightarrow& \Omega^\lambda(V \otimes W) = {\mathop\bigoplus\limits_{(d_1,d_2)\in \mathbb{Z}_2^2}}{\mathop\bigoplus\limits_{k\in \mathbb{Z},\atop i + j = k}} \left(V^{d_1}_{i} \otimes W^{d_2}_j \right),\\
(v_i)_{i\in \mathbb{Z}} \otimes (w_j)_{j\in \mathbb{Z}} &\longmapsto& \left({\mathop\sum\limits_{i + j = k}}\left(v_{i} \otimes w_{j}\right)\right)_{k\in \mathbb{Z}},
\end{eqnarray*}}
where $V,W \in {\rm \bf rep}\Pi^{\lambda_{\overline{0}}}_\infty \oplus {\rm \bf rep}\Pi^{\lambda_{\overline{1}}}_\infty$, $(v_i)_{i\in \mathbb{Z}} \in {\mathop\bigoplus\limits_{i\in \mathbb{Z}}} V^{d_1}_{i}$ and $(w_j)_{j\in \mathbb{Z}} \in {\mathop\bigoplus\limits_{j\in \mathbb{Z}}} W^{d_2}_{j}$. Then we can check that $(\Omega^\lambda, \mathrm{J}^\lambda)$ is a tensor functor.
It follows from Theorem \ref{equivalence-between-categories-and-primitive-reps-of-U-and-deformed-preprojetive-algs-of-A} that $\Omega^\lambda$ is a tensor equivalence.
\end{proof}

\section*{Acknowledgements}

Y.J. Xu was partially supported by the National Natural Science Foundation of China (Nos. 11501317,
12271292), the China Postdoctoral Science Foundation (No.2016M600530), and the Natural
Science Foundation of Qufu Normal University (No.BSQD20130142).
J.L. Chen was partially supported by the National Natural Science Foundation of China (No.11701019)
 and the Science and Technology Project of Beijing Municipal Education Commission (No.KM202110005012).

Both authors are grateful to Huixiang Chen, Libin Li, Zhankui Xiao for some useful comments during the 2021 National Conference on Hopf Algebra hosted by Southwest University in China.
Both authors would like to thank Dingguo Wang and Shilin Yang for some useful suggestions.
Y.J. Xu would like to thank Xiaowu Chen and Zerui Zhang for some helpful conversations.
Y.J. Xu would like to thank Bangming Deng and the department of Mathematical Sciences of Tsinghua University for the hospitality during his visit.

\section*{Declarations}
\textbf{Conflict of Interest.} The authors declare that they have no conflict of interest.

\end{document}